\documentclass[11pt,reqno]{amsart}
\usepackage[top=1.25in,bottom=1.25in,left=1in,right=1in]{geometry}
\usepackage{amsfonts, amsmath, amssymb, amscd, amsthm, graphicx}
\usepackage{enumitem}
\usepackage[final, 
colorlinks=true]{hyperref}
\usepackage[all,cmtip]{xy}
 
\DeclareRobustCommand{\SkipTocEntry}[5]{}

\usepackage{csquotes}
\MakeOuterQuote{"} 
 
\usepackage{pgfplots}
\usepackage{tikz,tikz-cd,amscd}
\usetikzlibrary{calc}
\usepackage{mathdots,color}
\usetikzlibrary{decorations.pathreplacing,decorations.markings}
\usetikzlibrary{shapes,shapes.arrows,positioning}

\tikzset{->-/.style={decoration={
  markings,
  mark=at position #1 with {\arrow{>}}},postaction={decorate}}}
\tikzset{->-/.default=.5}

\newcounter{alphathm}

\theoremstyle{plain}
\newtheorem{theorem}{Theorem}

\newtheorem{lemma}[theorem]{Lemma}
\newtheorem{proposition}[theorem]{Proposition}
\newtheorem{corollary}[theorem]{Corollary}

\newtheorem*{claim*}{Claim}

\theoremstyle{definition}
\newtheorem{definition}[theorem]{Definition}

\newtheorem{example}[theorem]{Example}
\newtheorem{remark}[theorem]{Remark}

\newtheorem{remark-convention}[theorem]{Remark-Convention}

\numberwithin{equation}{section}
\numberwithin{theorem}{section}

\newcommand{\fakeenv}{} 

\newenvironment{restate}[2]  
{ 
 \renewcommand{\fakeenv}{#2} 
 \theoremstyle{plain} 
 \newtheorem*{\fakeenv}{#1~\ref{#2}} 
 \begin{\fakeenv}
}
{
 \end{\fakeenv}
}

\newcommand{\RR}{\mathbb{R}}

\newcommand{\ZZ}{\mathbb{Z}} 

\newcommand{\calA}{\mathcal{A}}

\newcommand{\calC}{\mathcal{C}}
\newcommand{\calD}{\mathcal{D}}

\newcommand{\calL}{\mathcal{L}}

\newcommand{\calN}{\mathcal{N}}
\newcommand{\calO}{\mathcal{O}}
\newcommand{\calP}{\mathcal{P}}

\newcommand{\calT}{\mathcal{T}}

\newcommand{\calZ}{\mathcal{Z}}

\newcommand{\he}{\hat{e}}
\newcommand{\hv}{\hat{v}}
\newcommand{\hw}{\hat{w}}

\newcommand{\oY}{\overline{Y}}

\newcommand{\hT}{\widehat T}

\newcommand{\hZ}{\widehat Z}

\newcommand{\tM}{{\widetilde M}}

\newcommand{\tU}{{\widetilde U}}

\newcommand{\abs}[1]{\left\lvert {#1} \right\rvert} 
 
\newcommand{\I}[1]{\langle #1 \rangle}
\newcommand{\bd}{\partial}
\newcommand{\from}{\colon\thinspace}

\newcommand{\FF}{\mathrm{FF}}
\newcommand{\ZF}{\mathcal{Z}\mathrm{F}}
\newcommand{\coS}{\mathrm{CoS}}

\newcommand{\param}%
	{{\mathchoice{\mkern1mu\mbox{\raise2.2pt\hbox{$\centerdot$}}\mkern1mu}%
	{\mkern1mu\mbox{\raise2.2pt\hbox{$\centerdot$}}\mkern1mu}%
	{\mkern1.5mu\centerdot\mkern1.5mu}{\mkern1.5mu\centerdot\mkern1.5mu}}}

\DeclareMathOperator{\Aut}{Aut}

\DeclareMathOperator{\Inc}{Inc}

\DeclareMathOperator{\Isom}{Isom}
\DeclareMathOperator{\Mon}{Mon}
\DeclareMathOperator{\Out}{Out}
\DeclareMathOperator{\Stab}{Stab}
\DeclareMathOperator{\Sys}{Sys}
\DeclareMathOperator{\Teich}{Teich}
\DeclareMathOperator{\Wh}{Wh}
\DeclareMathOperator{\val}{val}
\DeclareMathOperator{\Mod}{Mod}

\title{Bounded projections to the $\calZ$--factor graph}

\author[M.~Clay]{Matt Clay}
\address{Department of Mathematical Sciences \\
University of Arkansas\\
Fayetteville, AR 72701}
\email{\href{mailto:mattclay@uark.edu}{mattclay@uark.edu}}

\author[C.~Uyanik]{Caglar Uyanik}
\address{Department of Mathematics \\
University of Wisconsin, Madison\\
480 Lincoln Drive\\ 
Madison, WI 53706}
\email{\href{mailto:caglar@math.wisc.edu}{caglar@math.wisc.edu}}

\begin{document}

\begin{abstract} 
Suppose $G$ is a free product $G = A_1 * A_2* \cdots * A_k * F_N$, where each of the groups $A_i$ is torsion-free and $F_N$ is a free group of rank $N$.  Let $\calO$ be the deformation space associated to this free product decomposition.  We show that the diameter of the projection of the subset of $\calO$ where a given element has bounded length to the $\calZ$--factor graph is bounded, where the diameter bound depends only on the length bound.  This relies on an analysis of the boundary of $G$ as a hyperbolic group relative to the collection of subgroups $A_i$ together with a given non-peripheral cyclic subgroup.  The main theorem is new even in the case that $G = F_N$, in which case $\calO$ is the Culler--Vogtmann outer space.  In a future paper, we will apply this theorem to study the geometry of free group extensions. 
\end{abstract}

\maketitle



\section{Introduction}\label{sec:intro}

Let $F$ be a finitely generated free group.  Given a subgroup $H \subseteq \Out(F)$, there is an extension of $F$ by $H$, denoted $E_H$, obtained via the following diagram
\begin{equation}\label{eq:ses}
\xymatrix{1\ar[r] & F\ar[r]\ar@{=}[d] & \Aut(F) \ar[r]^p &  \Out(F) \ar[r] &  1 \\
1\ar[r] & F\ar[r]  & E_H\ar[r]\ar[u]  & H\ar[r]\ar[u] & 1}
\end{equation}
where $E_H = p^{-1}(H)$.  Characterizing the subgroups $H \subseteq \Out(F_N)$ such that the extension $E_H$ is hyperbolic is an open question in geometric group theory.  A characterization of hyperbolicity for these types of free group extension allows one to determine when any free group extension is hyperbolic~\cite[Section~2.5]{ar:DT18}.

There is a long history behind this problem starting with Thurston's work on 3--manifolds that fiber over the circle~\cite{thurston98b}. For a surface $\Sigma$ of genus $g \geq 2$, the short exact sequence in the top row of \eqref{eq:ses} is the well known Birman exact sequence~\cite{ar:Birman}: 
\begin{equation*}
\xymatrix{1\ar[r] & \pi_1(\Sigma,*)\ar[r] & \Mod^{\pm}(\Sigma,*)\ar[r]^p & \Mod^{\pm}(\Sigma) \ar[r] & 1.
}
\end{equation*}
As for subgroups of $\Out(F)$, every subgroup $H<\Mod^{\pm}(\Sigma)$ gives rise to an extension $E_H = p^{-1}(H)$ of $\pi_{1}(\Sigma,*)$.  The seminal result of Thurston alluded to above implies that when $H = \I{f}$ is infinite cyclic the extension $E_H$ is hyperbolic precisely when $f$ is pseudo-Anosov.  This characterization was extended to finitely generated subgroups $H$ by combined work of Farb--Mosher~\cite{ar:FM02} and Hamenst\"adt~\cite{un:Ham}.  Specifically, their work shows that $E_H$ is hyperbolic if and only if $H$ is finitely generated and the orbit map into the Teichm\"uller space has quasi-convex image. Such a subgroup $H$ is called \emph{convex co-compact}.  Every convex co-compact subgroup is itself hyperbolic and purely pseudo-Anosov, meaning that every infinite order element is pseudo-Anosov~\cite{ar:FM02}.  
Hamenst\"adt~\cite{un:Ham} and Kent--Leininger~\cite{ar:KentLeininger} independently proved that a subgroup $H \subseteq \Mod^{\pm}(\Sigma)$ is convex co-compact if and only if $H$ is finitely generated and the orbit map $H \to \calC(\Sigma)$ into the curve complex is a quasi-isometric embedding.  Hence, for a finitely generated subgroup $H \subseteq \Mod^{\pm}(\Sigma)$, an extension $E_H$ is hyperbolic if and only if the orbit map $H \to \calC(\Sigma)$ is a quasi-isometric embedding.   

The situation in the setting of free group extensions has some partial progress analogous to the setting of surfaces, but a full characterization is still not known. There are two distinct analogs of pseudo-Anosov maps for free groups that play unique roles in the theory: fully irreducible elements and atoroidal elements.  An element $\varphi \in \Out(F)$ is called \emph{fully irreducible} if no power of $\varphi$ fixes the conjugacy class of a proper free factor of $F$.  These are elements of $\Out(F)$ that act on the closure of the Outer space with North-South dynamics~\cite{ar:LL}, and are precisely the elements of $\Out(F)$ that act as loxodromic isometries on the free factor graph \cite{ar:BF14}.  An element $\varphi\in\Out(F)$ is called \emph{atoroidal}
if no power of $\varphi$ fixes the conjugacy class of a non-trivial element in $F$.    Brinkmann proved that when $H =\I{\varphi}$ is infinite cyclic, the extension $E_{H}$ is hyperbolic precisely when $\varphi$ is atoroidal~\cite{ar:Brink}. 

Dowdall--Taylor were the first to produce a sufficient dynamical criterion on finitely generated subgroups $H \subseteq \Out(F)$ that ensures hyperbolicity of $E_H$.  Specifically, they prove that if $H$ is finitely generated and the orbit map $H \to \coS(F)$ into the co-surface graph is a quasi-isometric embedding, then $E_H$ is hyperbolic~\cite{ar:DT17}.  The co-surface graph (an electrification of the free factor graph) is hyperbolic and the loxodromic isometries of  $\coS(F)$ are precisely the fully irreducible and atoroidal elements~\cite{ar:DT17}.  Hence the existence of a quasi-isometric embedding also implies that $H$ is itself hyperbolic, and every infinite order element is both fully irreducible and atoroidal. The converse of the theorem by Dowdall--Taylor does not hold, even for infinite cyclic subgroups.  In essence, the property of being atoroidal is a mixing condition on the conjugacy classes of elements of $F$ whereas the property of being fully irreducible is a mixing condition on conjugacy classes of free factors of $F$.  There exist atoroidal but not fully irreducible $\varphi\in\Out(F)$ in all ranks at least 4, and for such a subgroup $\langle \varphi\rangle$, $E_{\I{\varphi}}$ is hyperbolic but the orbit map $\I{\varphi} \to \coS(F)$ has bounded image.    

Hence, to produce a more robust condition for hyperbolicity of free group extensions, one must keep track of the invariant free factors and focus on the dynamics relative to these subgroups.  We make a first step towards this in this paper.  Since we are only concerned with the dynamics relative to a collection of subgroups and not any dynamics within them, we do not necessarily only focus on free groups and can work in a more general setting.  To this end, in this work we consider free products of the form
\begin{equation*}
G = A_1 \ast \cdots \ast A_k \ast F_N
\end{equation*}
where $F_N$ is a free group of rank $N$.  Non-trivial elements or subgroups of $G$ that are conjugate into one of the $A_i$'s are called \emph{peripheral}.  The set of $G$--conjugacy classes of the factors $A_i$ is denoted by $\calA$.  

Associated to the pair $(G,\calA)$ are two spaces of trees that are the focus of this paper.  One of these is the \emph{relative outer space} $\calO = \calO(G,\calA)$, also known as a \emph{deformation space}.  This space parametrizes actions of $G$ on metric simplicial trees where each of the $A_i$ fixes a unique vertex, the vertex stabilizers are either trivial or conjugate to one of the $A_i$, and the stabilizer of an edge is trivial.  Such trees are called \emph{Grushko trees}.  In the case when $\calA = \emptyset$, so that $G$ is free, the space $\calO$ is the well-known Culler--Vogtmann Outer space~\cite{ar:CV86}.  The relative outer space $\calO$ plays the role of the Teichm\"uller space for a closed surface $\Sigma$.  The other space of interest for this paper is the \emph{$\calZ$--factor graph} $\ZF = \ZF(G,\calA)$.  In the case where $\calA= \emptyset$, this graph is quasi-isometric to the co-surface graph $\coS(F)$ mentioned above that was defined by Dowdall--Taylor~\cite[Proposition~4.4]{ar:DT17}.  Vertices in $\ZF$ are actions on simplicial trees where each of the $A_i$ fixes a unique vertex and the stabilizer of an edge is either trivial, or cyclic and non-peripheral.  Such trees are called \emph{$\calZ$--splittings}.  Two vertices $T_0$ and $T_1$ in $\ZF$ are connected by an edge, if either $T_1$ and $T_2$ are \emph{compatible}, or there is a non-peripheral element $g\in G$ that is elliptic in both $T_0$ and $T_1$. The $\calZ$--factor graph is one of a handful of hyperbolic graphs for free products that play the role of the curve complex for a finite-type surface.  For a detailed description and exposition of the connections between many of these, see the work of Guirardel--Horbez~\cite[Section ~2]{ar:GH22}.

The exact definitions of the relative outer space $\calO$ and the $\calZ$--factor graph $\ZF$ appear in Section~\ref{sec:graphs}.  For now, we observe that if one forgets the metric on a Grushko tree $T \in \calO$, the resulting simplicial tree is a $\calZ$--splitting and hence a vertex in $\ZF$.  Thus we have a well-defined projection map $\pi \from \calO \to \ZF$.  This projection map plays the role in the context of surfaces of the systole map $\Sys \from \Teich(\Sigma) \to \calC(\Sigma)$, which assigns to a marked hyperbolic surface $X \in \Teich(\Sigma)$ one of its shortest curves.

Loosely speaking, the main result of this paper states that the projection of the region of $\calO$ where a given non-peripheral element has bounded length has bounded diameter in $\ZF$, where the diameter bound depends only on the length bound.  To state it precisely, we require the following notation.  If $g \in G$ is non-peripheral and $T \in \calO$ is a Grushko tree, then $g$ has an invariant line in $T$ called its \emph{axis}, denoted $T_g$.  By $\abs{g}_T$ we denote the number of edges in a fundamental domain for the action of $g$ on $T_g$.  We call this quantity the \emph{combinatorial length} of $g$.  For a real number $L > 0$, we set $\calO_L(g)  = \{ T \in \calO \mid \abs{g}_T \leq L \}$.  

We can now state our main theorem.

\begin{theorem}\label{thm:bounded projections}
Let $(G,\calA)$ be a non-sporadic torsion-free free product.  For all $L > 0$, there is a $D > 0$ such that for any non-peripheral element $g \in G$, the diameter of $\pi(\calO_L(g)) \subset \ZF$ is at most $D$.
\end{theorem}

The notion of non-sporadic appears in Section~\ref{sec:graphs}.  Theorem~\ref{thm:bounded projections} is new even in the case where $\calA = \emptyset$ and thus $G$ is free. As alluded to earlier, Theorem~\ref{thm:bounded projections} will be applied in a subsequent paper to provide a new sufficient condition for the hyperbolicity of the extension $E_H$ generalizing the one due to Dowdall--Taylor.  Essentially, Theorem~\ref{thm:bounded projections} implies that if the orbit map along a geodesic in the subgroup $H$ to $\ZF$ makes definite progress, then no element of $F$ can stay short for long while moving along the corresponding path.  This type of statement will be promoted to an ``annuli flaring'' condition ensuring hyperbolicity of $E_H$.

For the remainder of this introduction, we will give an outline of the proof of Theorem~\ref{thm:bounded projections} and of the rest of the paper.  In broad strokes, our proof is analogous to how one might prove the similar statement in the setting of surfaces concerning the systole map $\Sys \from \Teich(\Sigma) \to \calC(\Sigma)$.  We will state this theorem for surfaces (which is known but does not appear in the literature), give a sketch of a proof, and comment on its relation to the proof of Theorem~\ref{thm:bounded projections}.

\begin{theorem}\label{thm:bounded surface projections}
Let $\Sigma$ be a finite-type surface with negative Euler characteristic.  For all $L > 0$, there is a $D > 0$ such that for any curve $\gamma$ on $\Sigma$, the diameter of $\Sys(\Teich_L(\gamma)) \subset \calC(\Sigma)$ is at most $D$.
\end{theorem}

There are three cases corresponding to the topological type of the curve $\gamma \subset \Sigma$: simple, non-filling, and filling.  We will deal with these one at a time in parallel with the free product setting.

\medskip

\noindent {\it Simple:} First, suppose that $\gamma$ is simple.  It follows from basic hyperbolic geometry that there is a constant $C$ such that the geometric intersection number between $\gamma$ and $\Sys(X)$ is at most $CL$ when $X \in \Teich_L(\gamma)$.  As the distance between simple closed curves in the curve complex is bounded above by their geometric intersection number, for $X \in \Teich_L(\gamma)$, we have that the distance in the curve complex between $\gamma$ and $\Sys(X)$ is at most $CL$.  Using the curve $\gamma$ as a central point to measure distances, it now follows that the diameter of $\Sys(\Teich_L(\gamma))$ is at most $2CL$.

In the setting of free products, we also consider simple elements first.  In this context, a non-peripheral element $g \in G$ is called \emph{simple} if it is elliptic in a $\calZ$--splitting where all edge stabilizers are trivial.  Such a $\calZ$--splitting is called a \emph{free splitting}.  There is a non-empty subset of $\ZF$ with diameter equal to one consisting of all of the $\calZ$--splittings in which $g$ is elliptic.  This subset plays the role of $\gamma$ as a central point to which we will measure distance from for a tree in $\pi(\calO_L(g))$.  The main tool we exploit here is the ubiquitous notion of a \emph{Whitehead graph}.  Using an appropriate notion of a Whitehead graph for free-products due to Guirardel--Horbez (Section~\ref{sec:whitehead}) and its properties for simple elements, it will follow that if $T \in \calO_L(g)$, then there is a free splitting $S \in \ZF$ for which $g$ is elliptic and where $d(\pi(T),S) \leq L$.  From this, we conclude that the diameter of $\pi(\calO_L(g))$ is at most $2L + 1$ (Proposition~\ref{prop:length bounded simple}).

\medskip

\noindent {\it Non-filling:} Next, suppose that $\gamma$ is contained in a proper subsurface $\Sigma_\gamma \subset \Sigma$ and let $\alpha$ be one of the boundary curves of $\Sigma_\gamma$.  Then we see that $\abs{\alpha}_X \leq \abs{\gamma}_X$ for any $X \in \Teich(\Sigma)$ and hence $\Teich_L(\gamma) \subseteq \Teich_L(\alpha)$ for any $L$.  As $\alpha$ is simple, we have that the diameter of $\Sys(\Teich_L(\gamma))$ is bounded in terms of $L$ by the first case.

In the setting of free products, the analogous elements are called \emph{$\calZ$--simple}.  These are the elements $g \in G$ that are elliptic in some $\calZ$--splitting.  There are subcases here depending on whether or not the element $g$ is \emph{quadratic}, that is, if there is a geometric model for $(G,\calA)$ in which the conjugacy class of $g$ corresponds to a boundary component (Definition~\ref{def:quad}).  In both of these subcases, the key point is to try to find a $\calZ$--splitting $S$ where $g$ is elliptic and where the length of an edge stabilizer in $S$ is bounded by some function of the length of $g$.  As edge stabilizers in $\calZ$--splittings are simple, we can use the first case to get a bound on $\pi(\calO_L(g))$ much like in the surface case.
  
However, the details in the two subcases are very different.  In the quadratic case, we use the surface from the definition directly to find this $\calZ$--splitting (Proposition~\ref{prop:length bounded quadratic}).  To this end, we provide a novel characterization of quadratic elements in terms of Whitehead graphs (Proposition~\ref{prop:quadratic}) which generalizes the characterization due to Otal in the case when $\calA$ is empty~\cite[Theorem~2]{ar:Otal92} (cf.~\cite[Theorem~6.1]{ar:CM11}).  

In the non-quadratic case, we study a quotient of the boundary of $G$ as a relatively hyperbolic group with respect to $\calA$ where we identify pairs of points parameterized by the conjugates of $g$.  Such a space is called a \emph{decomposition space}.  This space is in fact the boundary of $G$ as a hyperbolic group relative to a new collection of subgroups that includes $\calA$.  The study of the decomposition space employs techniques from and adds to the existing literature regarding using the boundary of a hyperbolic or a relatively hyperbolic group to understand $\calZ$--splittings of the group as developed by Bowditch~\cite{ar:Bowditch98}, Cashen--Macura~\cite{ar:CM11}, Cashen~\cite{ar:Cashen16}, Haulmark~\cite{ar:Haulmark19}, and Haulmark--Hruska~\cite{ar:HH}.  This analysis is carried out in Section~\ref{sec:short} and the proof of Theorem~\ref{thm:bounded projections} for general $\calZ$--simple elements appears as Proposition~\ref{prop:length bounded elliptic}.  Prior to these sections, we must develop a finite model for working with decomposition spaces.  This includes a generalization of the Guirardel--Horbez notion of a Whitehead graph that is not focused on a single vertex of a Grushko tree, but takes into account an entire locally finite subtree of a Grushko tree.  This construction takes place in Section~\ref{sec:model} and forms the basis for the analysis in Sections~\ref{sec:quadratic elements} and~\ref{sec:short}.

\medskip

\noindent {\it Filling:} Finally, suppose that $\gamma$ is a filling curve.  Then the diameter of $\Teich_L(\gamma)$ with the Teichm\"uller metric is bounded as a subset of $\Teich(X)$.  Since the systole map $\Sys \from \Teich(\Sigma) \to \calC(\Sigma)$ is coarsely Lipschitz, it follows that $\Sys(\Teich_L(\gamma))$ has bounded diameter.  Up to homeomorphism of $\Sigma$, there are only finitely many curves on $\Sigma$ that have length at most $L$ on some hyperbolic surface $X$.  Thus we can obtain a bound on the diameter of $\Sys(\Teich_L(\gamma))$ that depends only on $L$ and not on $\gamma$.   

In the setting of free products, we take a related, but different approach.  Partly this is for efficiency, but partly this is also due to necessity.  Indeed, it is not true that up to the action of outer automorphisms of $G$ that preserve $\calA$ that there are only finitely many conjugacy classes of non-peripheral elements in $G$ that have combinatorial length at most $L$ in some Grushko tree $T$.  Thus proving boundedness of $\pi(\calO_L(g))$ for a given non-peripheral element $g$ is not sufficient.  To this end, we use contradiction and assume there is a sequence of elements $(g_n)$ and sequences of Grushko trees $(S_n)$ and $(T_n)$ where $\abs{g_n}_{S_n},\abs{g_n}_{T_n}  \leq L$ but yet the distance between $\pi(S_n)$ and $\pi(T_n)$ is unbounded.  We show that with this set-up we can find a single non-peripheral element $g \in G$ where $\abs{g}_{S_n}$ and $\abs{g}_{T_n}$ are both bounded.  This enables us to find a tree appearing in the closure of the relative outer space $\calO$ in which $g$ is elliptic.  As a consequence of the Rips machine, it follows that such an element $g$ is necessarily $\calZ$--simple.  This is now a contradiction to the above case.  This argument appears in Section~\ref{sec:main proof}.    

\addtocontents{toc}{\SkipTocEntry}
\subsection*{Acknowledgements}
First, we thank Derrick Wigglesworth who started this project with us and with whom we had many useful and deep conversations.  We would also like to thank Chris Cashen, Matthew Haulmark, Camille Horbez, Chris Hruska, and Yo'av Rieck for answering questions related to their work and for discussions pertaining to this work.  We also thank the referees for their careful readings and feedback.


\section{Relative outer space and relative factor graphs}\label{sec:graphs}

In this section, we introduce the basic setting and notation necessary for Theorem~\ref{thm:bounded projections} and its proof.  

Let $A_1,\ldots,A_k$ be countably infinite torsion-free groups and let $G$ denote the free product
\begin{equation*}
G = A_1 \ast \cdots \ast A_k \ast F_N
\end{equation*}
where $F_N$ is a free group of rank $N$.  Any subgroup of $G$ conjugate into one of $A_1,\ldots,A_k$ is called a \emph{peripheral subgroup} and any element of $G$ contained in a peripheral subgroup is called \emph{peripheral} as well.  The collection $\calA = \{[A_1],\ldots,[A_k]\}$ of conjugacy classes is called the \emph{peripheral factor system of $G$}.  

The \emph{complexity of $(G,\calA)$} is defined as the quantity $\xi(G,\calA) = 2k + 3N - 3$.  Throughout this paper we will always assume that $\xi(G,\calA) \geq 3$.  This excludes the following five cases which are called \emph{sporadic}:
\begin{itemize}
\item $\xi(G,\calA) = -3$: $G = \{1\}$ and $\calA = \emptyset$

\item $\xi(G,\calA) = -1$: $G = A_1$ and $\calA = \{[A_1]\}$

\item $\xi(G,\calA) = 0$: $G = \ZZ$ and $\calA = \emptyset$

\item $\xi(G,\calA) = 1$: $G = A_1 \ast A_2$ and $\calA = \{[A_1],[A_2]\}$

\item $\xi(G,\calA) = 2$: $G = A_1 \ast \ZZ$ and $\calA = \{[A_1]\}$
\end{itemize}

The subgroup of $\Out(G)$ consisting of outer automorphisms $\varphi$ such that $\varphi([A_i]) = [A_i]$ for each $i = 1,\ldots,k$ is denoted $\Out(G,\calA)$.

A \emph{$(G,\calA)$--tree} is a metric simplicial tree $T$ equipped with a minimal cocompact action of $G$ by isometries without inversions such that each peripheral subgroup of $G$ fixes a point in $T$.  Two $(G,\calA)$--trees are considered equivalent if there is a $G$--equivariant isometry between them.  We will (almost) always assume that if some vertex of a $(G,\calA)$--tree has degree two, then there is an element in $G$ that fixes this vertex and interchanges the two incident edges.   

For a point $p \in T$ in a $(G,\calA)$--tree, a \emph{direction at $p$} is a connected component of $T - \{p\}$.  We will also refer to a connected component $Y \subset T - X$ as a \emph{direction at $X$} where $X \subset T$ is a connected closed subset.  There is a unique point $\bd_0Y \in X$ such that $Y$ is a direction at $\bd_0Y$.

Given $(G,\calA)$--trees $T$ and $T'$, a \emph{collapse} is a continuous $G$--equivariant function $f \from T \to T'$ such that the pre-image of any point in $T'$ is a connected subset of $T$.  Two $(G,\calA)$--trees $T_0$ and $T_1$ are \emph{compatible} if there exists a $(G,\calA)$--tree $T$ and a pair of collapse maps $f_0 \from T \to T_0$ and $f_1 \from T \to T_1$.  

There is a right action of $\Out(G,\calA)$ on the set of equivalence classes of $(G,\calA)$--trees.  Namely, given $\varphi \in \Out(G,\calA)$ one first fixes an automorphism $\Phi \in \Aut(G)$ that represents $\varphi$.  Then given a $(G,\calA)$--tree $T$ with action homomorphism $\rho \from G \to \Isom(T)$, we define a new action on $T$ by the homomorphism $\rho_\Phi = \rho \circ \Phi$.  It is straightforward to check that this defines an action satisfying the requirements of a $(G,\calA)$--tree and that the equivalence class of the new $(G,\calA)$--tree is independent of the choices made.  We denote this new $(G,\calA)$--tree by $T\varphi$.

A $(G,\calA)$--tree $T$ is a called a \emph{Grushko tree} if the following conditions hold:
\begin{enumerate}
\item every elliptic element for $T$ is peripheral, and 
\item the stabilizer of any edge in $T$ is trivial.
\end{enumerate}
We observe that this implies that each peripheral subgroup fixes a unique vertex of $T$ and that if $v$ is a vertex of $T$ whose stabilizer, $\Stab_T(v)$ is non-trivial, then $\Stab_T(v)$ is conjugate to one of the $A_i$.   

If $T$ is a Grushko tree, then the number of edges in $T/G$ is at most the complexity $\xi(G,\calA)$, when $(G,\calA)$ is non-sporadic.  Indeed, the maximum occurs precisely when the vertices with non-trivial stabilizer have degree one and each other vertex has degree three.

\begin{definition}\label{def:relative outer space}
The set of equivalence classes of Grushko trees for $(G,\calA)$ is called the \emph{relative outer space} and is denoted $\calO = \calO(G,\calA)$.
\end{definition} 
There is a natural way to equip this space with a topology where the $\Out(G,\calA)$--action is by homeomorphisms.  See the work of Guirardel--Levitt for more details~\cite{ar:GL07}.  

If $T \in \calO$ is a Grushko tree for $(G,\calA)$ and $g \in G$ is non-peripheral, it acts on $T$ as a hyperbolic isometry and by $T_g$ we denote its axis.  Further, by $\abs{g}_T$ we denote the number of edges in a fundamental domain for $g$ on $T_g$.  In other words, $\abs{g}_T$ is the number of edges in $T_g/\I{g}$. For $g \in G$ non-peripheral and $L > 0$, we denote:
\begin{equation*}
\calO_L(g) = \{ T \in \calO \mid \abs{g}_T \leq L\}.
\end{equation*}
We note that this set could possibly be empty.

A \emph{$\calZ$--splitting} of $(G,\calA)$ is a $(G,\calA)$--tree where every edge has length one and the stabilizer of any given edge is either trivial or cyclic and non-peripheral.  A non-peripheral element $g \in G$ is \emph{$\calZ$--simple} if there exists a $\calZ$--splitting of $(G,\calA)$ where $g$ is elliptic.  

A $\calZ$--splitting where the stabilizer of any given edge is trivial is called a \emph{free-splitting} of $(G,\calA)$.   Notice that a Grushko tree where every edge has length one is a free-splitting.  A non-peripheral element $g \in G$ is \emph{simple} if there is a free--splitting of $(G,\calA)$ where $g$ is elliptic. 

We note that every simple element is also $\calZ$--simple.  Here is one way to construct $\calZ$--simple elements that are not simple.  Consider a free factorization of a finitely generated free group $F = A \ast B$ where $A$ has rank at least two.  Fix an element $a \in A$ that is not simple in $A$.  Let $S$ be the Bass--Serre tree of the amalgamated free product decomposition $A \ast_{\I{a}} \I{a,B}$; this is a $\calZ$--splitting of $F$.  Any element in $\I{a,B}$ that is not simple in $\I{a,B}$ is an example of a $\calZ$--simple element of $F$ as it is elliptic in $S$ but is not simple in $F$, as can be seen using Whitehead graphs.

The following lemma of Horbez regarding edge stabilizers in $\calZ$--splittings shows that the above construction is essentially the only way.

\begin{lemma}[{\cite[Lemma~6.11]{ar:Horbez17}}]\label{lem:edge stabs in Z-splittings}
If $g$ fixes an edge in a $\calZ$--splitting, then $g$ is simple.
\end{lemma}

When $\calA = \emptyset$ and thus $G$ is a free group, this is a well-known fact originally due to Shenitzer~\cite{ar:Shenitzer55} and Swarup~\cite{ar:Swarup86}.  

\begin{definition}\label{def:zf}
The set of equivalence classes of $\calZ$--splittings of $(G,\calA)$ form the vertex set of the \emph{$\calZ$--factor graph}, denoted $\ZF = \ZF(G,\calA)$.  There is an edge joining the equivalence classes of two $\calZ$--splittings $T_0$ and $T_1$ if either:
\begin{enumerate}
\item $T_0$ and $T_1$ are compatible, or
\item there exists a non-peripheral element $g \in G$ which is elliptic in both $T_0$ and $T_1$.
\end{enumerate}
We let $d$ denote the path metric on the graph $\ZF$.
\end{definition}

We note that for free groups, i.e., when $\calA = \emptyset$, the $\calZ$--factor graph is quasi-isometric to the Kapovich--Lustig intersection graph~\cite{ar:KL09} and also to the Dowdall--Taylor co-surface graph~\cite{ar:DT17}.  See the work of Guirardel--Horbez for these facts and more information about the $\calZ$--factor graph, including a proof of its hyperbolicity and a description of its boundary~\cite{ar:GH22}.  

There is a natural $\Out(G,\calA)$--equivariant map $\pi \from \calO \to \ZF$ defined by scaling the metric on each edge to have length one.


\section{Boundaries and lines in free products}\label{sec:boundaries}

As mentioned in the Introduction, central to the proof of Theorem~\ref{thm:bounded projections} is an analysis of the $\calZ$--splittings of a $\calZ$--simple element.  This is carried out by considering the free product $(G,\calA)$ as a relatively hyperbolic group, both with the structure given by the free factor system $\calA$, and with the structure given by including the conjugacy class of the maximal cyclic subgroup containing a given non-peripheral element.  Under these considerations, we relate $\calZ$--splittings of the free product where a given non-peripheral element is elliptic, to cut sets in the corresponding boundary.  This idea was first used to understand $\calZ$--splittings of hyperbolic groups by Bowditch~\cite{ar:Bowditch98}.  It was applied to relatively hyperbolic groups by Haulmark~\cite{ar:Haulmark19} and Haulmark--Hruska~\cite{ar:HH}.  Preceeding the work of Haulmark and Haulmark--Hruska in the general setting of relatively hyperbolic groups, is the work of Cashen--Macura~\cite{ar:CM11} and Cashen~\cite{ar:Cashen16}. These papers apply the above ideas to the case of a finitely generated free group, $F$, building on work of Otal~\cite{ar:Otal92}, to understand the $\calZ$--splittings of $F$ where a given $\calZ$--simple element is elliptic.  By focusing on the setting of a free group, they are able to obtain the types of important quantitative bounds that we need for Theorem~\ref{thm:bounded projections}.  We closely follow their analysis.  For clarity, in places we opt for a direct argument in the setting of free products as opposed to appealing to a general statement regarding all relatively hyperbolic groups.  In a different direction than what we pursue here, Barrett has been able to obtain quantitative bounds similar to those of Cashen--Macura in the setting of hyperbolic groups~\cite{ar:Barrett18}.

\subsection{Boundaries}\label{subsec:boundaries} In this section, we introduce the various notions of boundary used in the sequel and talk about how they are related.  To begin, we discuss the boundary of the pair $(G,\calA)$ as a free product as defined by Guirardel--Horbez~\cite{ar:GH19}.    

Given a Grushko tree $T \in \calO$, we define the following spaces:
\begin{align*}
\bd_\infty T &= \text{the Gromov boundary of $T$, i.e., equivalence classes of rays $r \from [0,\infty) \to T$} \\
\hT &= T \cup \bd_\infty T, \text{ with the usual topology on the union of a hyperbolic space and its boundary} \\
V_\infty(T) &= \text{the set of vertices of $T$ with non-trivial stabilizer} \\
\bd T & = V_\infty(T) \cup \bd_\infty T \subset \hT_{\rm obs}
\end{align*}
In the last line, $\hT_{\rm obs}$ is the set $T \cup \bd_\infty T$ considered with the observers topology.  A basis element for the observers topology consists of a connected component (in the usual sense) of $\hT - P$, where $P \subset \hT$ is a finite set of points.  In other words, a basis consists of intersections of finitely many directions.  In particular, a sequence $(p_n) \subset T$ converges to $p_\infty \in \hT_{\rm obs}$ in the observers topology if for all $q \in \hT_{\rm obs} - \{p_\infty\}$, the direction at $q$ that contains $p_\infty$ also contains $p_n$ for all but finitely many $n$.  With the observers topology, $\hT_{\rm obs}$ is a compact space and $\bd T \subset \hT_{\rm obs}$ is closed, hence also compact.  

Note, a point $v \in V_\infty(T)$ can also be viewed as the set of rays $r \from [0,\Delta] \to T$ where $r(\Delta) = v$.

As described by Guirardel--Horbez, for any two Grushko trees $T,T' \in \calO$, there is a canonical $G$--equivariant homeomorphism $h \from \bd T \to \bd T'$ such that $h(V_\infty(T)) = V_\infty(T')$ and $h(\bd_\infty T) = \bd_\infty T'$~\cite[Lemma~2.2]{ar:GH19}.  Hence the free product $(G,\calA)$ has a well-defined boundary $\bd(G,\calA)$ and this boundary contains well-defined subsets $V_\infty(G,\calA)$ and $\bd_\infty(G,\calA)$.  

Moreover, we have that $\bd(G,\calA)$ is $G$--equivariantly homeomorphic to the Bowditch boundary of $G$ as a hyperbolic group relative to $\calA$.  Indeed, using the language and notation of Bowditch's work, $T$ is a fine hyperbolic graph and the action of $G$ on $T$ satisfies the second definition of relative hyperbolicity with respect to the collection $\calA$~\cite{ar:Bowditch12}.  The induced topology on $\bd T$ as a subset of $\hT_{\rm obs}$ defined above is exactly the topology described by Bowditch on $\Delta T$, once vertices with trivial stabilizers are removed from $\Delta T$.  Such points are isolated.  It is then shown by Bowditch that $\Delta T$ minus the isolated points is $G$--equivariantly homeomorphic to the boundary of $(G,\calA)$ as a relatively hyperbolic group~\cite[Proposition~9.1]{ar:Bowditch12}.  This connection was also observed by Knopf~\cite[Definition~2.2]{ar:Knopf}.

\subsection{Lines}\label{subsec:lines} Recall that for a non-peripheral element $g \in G$ and a Grushko tree $T \in \calO$, the axis of $g$ is denoted $T_g$.  This axis defines two points $T_g^+,T_g^- \in \bd_\infty T$ and consequently two points $g^\infty, g^{-\infty} \in \bd_\infty(G,\calA)$.  These points are characterized by the fact that for any point $p \in T$ we have $\lim_{m \to \infty} g^m p = T_g^+$ and $\lim_{m \to \infty} g^{-m} p = T_g^-$.  A two point set of the form $\{g^\infty,g^{-\infty}\} \subset \bd_\infty(G,\calA)$ is called a \emph{periodic line}.  Given a periodic line $\ell = \{g^\infty,g^{-\infty}\}$ and a Grushko tree $T \in \calO$, by $\ell_T$ we denote the axis $T_g$.          
  
If $g \in G$ is non-peripheral we set $\calL_g =  \{ \{ ag^\infty,ag^{-\infty} \} \mid a \in G \}$.  In terms of a Grushko tree $T \in \calO$, the lines in $\calL_g$ correspond to the endpoints of the axes $T_{aga^{-1}}$ for each $a \in G$.  We observe that $\calL_g = \calL_{g^m}$ for any nonzero $m \in \ZZ$.

A \emph{periodic line collection} is a collection of the form:
\begin{equation*}
\calL = \calL_{g_1} \cup \cdots \cup \calL_{g_m}
\end{equation*}
where each $g_1,\ldots,g_m \in G$ is a non-peripheral element.  We call the elements $g_1, \ldots, g_m$ the \emph{generators} of the periodic line collection. For $\calL$ as above, we define
\begin{equation*}
\abs{\calL}_T = \sum_{j = 1}^m \abs{g_j}_T.
\end{equation*}  

In a free-product, every non-periphereal element $g \in G$ is contained in a unique maximal 2--ended subgroup denoted $N_g$.  In fact $N_g$ is both the normalizer of $g$ in $G$ and the stabilizer of the axis $T_g \subset T$.  Thus, for a general free product, $N_g$ is isomorphic to either $\ZZ$ or to $\ZZ_2 \ast \ZZ_2$.  As we are assuming that our free product $(G,\calA)$ is torsion-free, we must have that $N_g$ is isomorphic to $\ZZ$.  Given a periodic line collection $\calL = \calL_{g_1} \cup \cdots \cup \calL_{g_m}$, we denote $\calN_\calL = \{ [N_{g_1}], \ldots,[N_{g_m}]\}$.  When $\calL$ has a single generator $g$, we will use the notation $\calN_g$.  Note that this is just the conjugacy class of the subgroup $\I{g'}$ where $g'$ is an indivisible root of $g$.

\subsection{Decomposition spaces}\label{subsec:decomp}

In the setting of free groups, decomposition spaces were first introduced by Otal to study free splittings of free groups and related items~\cite{ar:Otal92}.

\begin{definition}\label{def:decomposition space}
Suppose $(G,\calA)$ is a torsion-free free product and let $\calL$ be a periodic line collection.  The \emph{decomposition space for $\calL$}, denoted $\calD(\calL)$, is the quotient of $\bd(G,\calA)$ by the collection $\calL$.  In other words, if $\calL = \calL_{g_1} \cup \cdots \cup \calL_{g_m}$, then we identify the points $ag_j^{-\infty} \sim ag_j^\infty$ for each $a \in G$ and $j = 1,\ldots,m$.  The quotient map is denoted by $q \from \bd(G,\calA) \to \calD(\calL)$.
\end{definition}

When necessary, we will write $\calD_{(G,\calA)}(\calL)$ and $q_{(G,\calA)}$ if we need to keep track of the free product $(G,\calA)$.

The connection between the decomposition space and the Bowditch boundary of a relatively hyperbolic group is well-known to experts and is given by the following lemma.  For a proof, see the work of Haulmark--Hruska and the references within the proof of~\cite[Proposition~6.7]{ar:HH}. 

\begin{lemma}\label{lem:boundary}
Suppose $(G,\calA)$ is a torsion-free free product and that $\calL$ is a periodic line collection.  Then $G$ is hyperbolic relative to $\calA \cup \calN_\calL$ and the Bowditch boundary of the relatively hyperbolic group $(G,\calA \cup \calN_\calL)$ is $G$--equivariantly homeomorphic to $\calD(\calL)$.
\end{lemma}


\section{Relative Whitehead graphs \`a la Guirardel--Horbez}\label{sec:whitehead}

We seek to have a finite model for understanding the topology of a decomposition space $\calD(\calL)$.  As in the case when $\calA = \emptyset$ studied by Otal~\cite{ar:Otal92} and Cashen--Macura~\cite{ar:CM11}, the object we ultimately use is a generalization of the classical notion of a Whitehead graph.  To this end, in this section we introduce the concept of a Whitehead graph for Grushko trees in $\calO$.  In the setting of free products, the Whitehead graph of a non-peripheral element was defined by Guirardel--Horbez \cite{ar:GH19}. As in the classical setting, Whitehead graphs can be used to detect if an element is simple.  We state the version of this fact in the setting of free products 
due to Guirardel--Horbez in Proposition~\ref{prop:whitehead criterion}.       

Let $(G,\calA)$ be a fixed non-sporadic free product.  Guirardel--Horbez define the notion of a Whitehead graph at a vertex $v$ in some Grushko tree $T \in \calO$ for a given periodic line collection $\calL$, denoted by $\Wh_T(\calL,v)$~\cite[Section~5.1]{ar:GH19}.  We give a slight, but equivalent, variant of their definition that naturally leads to the generalization we define in Section~\ref{sec:model}.  The vertex set of $\Wh_T(\calL,v)$ is the set of equivalence classes of directions at $v$, $Y \subset T - \{v\}$, under the action of $\Stab_T(v)$.  In each equivalence class, $[Y]$, a preferred direction is chosen; these preferred directions are enumerated by $Y_1,\ldots,Y_m$.  There is an oriented edge in $\Wh_T(\calL,v)$ from $[Y_i]$ to $[Y_j]$, labeled by $a \in \Stab_T(v)$, for each line $\ell \in \calL$ such that $\ell_T$ that meets both $Y_i$ and $aY_j$.  Note, in this case $a^{-1}\ell_T$ meets both $Y_j$ and $a^{-1}Y_i$ and represents the same edge with the opposite orientation.  We remark that this definition also makes sense for any point $p \in T$ that is not a vertex.  In this case $\Wh_T(\calL,p)$ has exactly two equivalence classes, each of which consists of a unique direction.  Further, in this case, the Whitehead graph $\Wh_T(\calL,p)$ is disconnected if and only if no geodesic in $\calL$ contains the point $p$. 

\begin{example}\label{ex:whitehead graph}
Let $G$ be the free group of rank three with basis $\{a,b,c\}$ and we consider the free factor system $\calA = \{[\I{a}]\}$.  Let $T \in \calO$ be the Grushko tree where all edges have length 1 and $T/G$ is the 2--rose where one edge is labeled by $b$ and the other labeled by $c$.  This quotient graph of groups is shown in Figure~\ref{fig:graph of groups}.

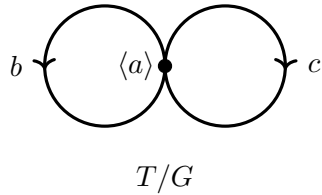
\begin{figure}[ht]
\begin{tikzpicture}
\node at (0,-1.5) {$T/G$};
\filldraw (0,0) circle [radius=2.5pt];
\draw[very thick,->-=0.51] (-0,0) arc (0:360:8mm) node[pos=0.5,label=left:{$b$}] {};
\draw[very thick,->-=0.51] (0,0) arc (180:-180:8mm) node[pos=0.5,label=right:{$c$}] {};
\node[left] at (0,0) {$\I{a}$};
\end{tikzpicture}
\caption{The quotient graph of groups $T/G$ in Example~\ref{ex:whitehead graph}.}\label{fig:graph of groups}
\end{figure}

Let $v$ be the unique vertex of $T$ fixed by $a$ and let $g = bacb^{-1}a^3c^{-1}$.  Let $e_1$ be the edge incident to $v$ and $bv$ and let $e_2$ be the edge incident to $v$ and $cv$.  We denote the direction at $v$ that contains $bv$ by $Y_b^+$ and we denote the direction at $v$ that contains $b^{-1}v$ by $Y_b^-$.  Similarly, we use $Y_c^+$ and $Y_c^-$ to the denote the directions at $v$ that contain $cv$ or $c^{-1}v$ respectively.  The axes of $g$ that meet one of these four directions are shown on the left in Figure~\ref{fig:whitehead graph}.  Using these preferred directions, the Whitehead graph $\Wh_T(\calL_g,v)$ is shown on the right in Figure~\ref{fig:whitehead graph}.

\begin{figure}[ht]
\begin{tikzpicture}
\begin{scope}[xshift=6cm]
\draw[red,very thick,->-] (0,0) -- (2,0) node[pos=0.5,label=below:{\color{black}$1$}] {};
\draw[blue,very thick,->-] (2,0) -- (2,2) node[pos=0.5,label=right:{\color{black}$a$}] {};
\draw[black!30!green,very thick,->-] (2,2) -- (0,2) node[pos=0.5,label=above:{\color{black}$1$}] {};;
\draw[purple,very thick,->-] (0,2) -- (0,0) node[pos=0.5,label=left:{\color{black}$a^3$}] {};
\filldraw (0,0) circle [radius=2.5pt];
\filldraw (2,0) circle [radius=2.5pt];
\filldraw (2,2) circle [radius=2.5pt];
\filldraw (0,2) circle [radius=2.5pt];
\node[below] at (0,0) {$[Y_c^-]$};
\node[below] at (2,0) {$[Y_b^-]$};
\node[above] at (2,2) {$[Y_c^+]$};
\node[above] at (0,2) {$[Y_b^+]$};
\node at (1,-1.5) {$\Wh_T(\calL_g,v)$};
\end{scope}
\begin{scope}
\node[draw=none,minimum size=5cm,regular polygon,regular polygon sides=8] at (0,1) (a) {};
\foreach \x in {1,6} {
	\draw[very thick,->-=0.6] (0,1) -- (a.corner \x);
}
\draw[very thick,->-=0.6] (0,1) -- (a.corner 7) node[pos=0.55,label=below:{$e_2$}] {};
\draw[very thick,->-=0.6] (0,1) -- (a.corner 8) node[pos=0.55,label=above:{$e_1$}] {};
\foreach \x in {2,5} {
	\draw[very thick,->-=0.5] (a.corner \x) -- (0,1);
}
\draw[very thick,->-=0.5] (a.corner 3) -- (0,1) node[pos=0.45,label=above:{$b^{-1}e_1$}] {};	
\draw[very thick,->-=0.5] (a.corner 4) -- (0,1) node[pos=0.45,label=below:{$c^{-1}e_2$}] {};
\foreach \x in {1,2,...,8} {
	\fill (a.corner \x) circle[radius=2pt];
}
\def\ss{6};
\def\labs{0.15}
\def\axescolor{blue}
\draw[blue,thick] ([xshift=-\ss,yshift=-\ss]a.corner 3) -- ([xshift=-\ss,yshift=-\ss]0,1) -- ([xshift=-\ss,yshift=-\ss]a.corner 6);
\draw[blue,thick] ([xshift=\ss,yshift=\ss]a.corner 2) -- ([xshift=\ss,yshift=\ss]0,1) -- ([xshift=\ss,yshift=\ss]a.corner 7);
\draw[purple,thick] ([xshift=-\ss,yshift=\ss]a.corner 4) -- ([xshift=-\ss,yshift=\ss]0,1) -- ([xshift=-\ss,yshift=\ss]a.corner 1);
\draw[purple,thick] ([xshift=\ss,yshift=-\ss]a.corner 5) -- ([xshift=\ss,yshift=-\ss]0,1) -- ([xshift=\ss,yshift=-\ss]a.corner 8);
\draw[red,thick] ([xshift=-\ss,yshift=-2*\ss]a.corner 3) -- ([xshift=-5*\ss]0,1) -- ([xshift=-\ss,yshift=2*\ss]a.corner 4);
\draw[black!30!green,thick] ([xshift=\ss,yshift=-2*\ss]a.corner 8) -- ([xshift=5*\ss]0,1) -- ([xshift=\ss,yshift=2*\ss]a.corner 7);
\node[above right=\labs] at (a.corner 1) {$a^{-3}Y_b^+$};
\node[above left=\labs] at (a.corner 2) {$a^{-1}Y_b^-$};
\node[left=\labs] at (a.corner 3) {$Y_b^-$};
\node[left=\labs] at (a.corner 4) {$Y_c^-$};
\node[below=\labs] at (a.corner 5) {$a^3Y_c^-$};
\node[below=\labs] at (a.corner 6) {$aY_c^+$};
\node[right=\labs] at (a.corner 7) {$Y_c^+$};
\node[right=\labs] at (a.corner 8) {$Y_b^+$};
\filldraw (0,1) circle [radius=2.5pt];
\end{scope}
\end{tikzpicture}
\caption{Some translates of $T_g$ in $T$ that meet $v$ and the Whitehead graph $\Wh_T(\calL_g,v)$ from Example~\ref{ex:whitehead graph}.}\label{fig:whitehead graph}
\end{figure}
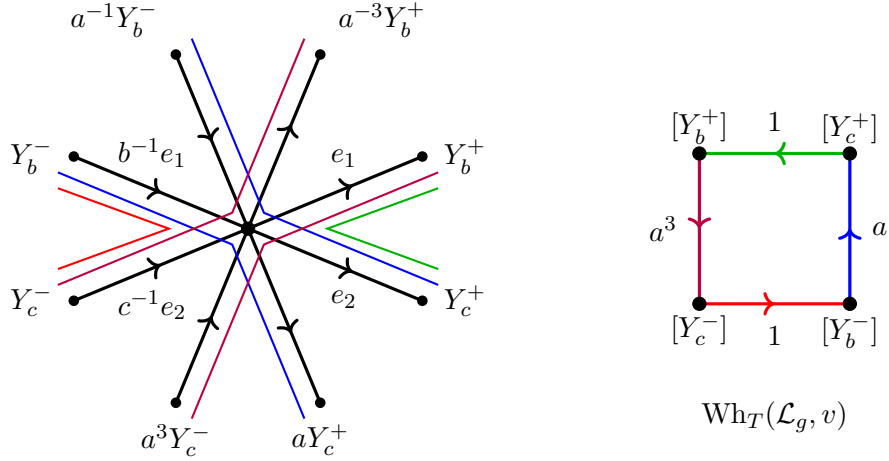
\end{example}

For a connected subgraph $U \subseteq \Wh_T(\calL,v)$ and a vertex $[Y] \in U$, the subgroup of $\Stab_T(v)$ generated by the products of labels along the closed paths based at $[Y]$ is denoted $\Mon(U,[Y])$ and is called the \emph{monodromy subgroup}.  Referring to Example~\ref{ex:whitehead graph}, we have that $\Mon(\Wh_T(\calL_g,v),[Y_c^-]) = \I{a^4}$.

As noted by Guirardel--Horbez, the conjugacy class of $\Mon(U,[Y])$ in $\Stab_T(v)$ is independent of the choice of vertex $[Y]$ and the choices of preferred representatives $Y_1,\ldots,Y_m$.  Moreover, given a maximal tree $\tau \subseteq U$, it is possible to choose preferred directions so that the label on each edge in $\tau$ is trivial.  The conjugacy class of $\Mon(U,[Y])$ in $\Stab_T(v)$ is denoted $\Mon(U)$.  In the case that $\Mon(U)$ is the conjugacy class of the trivial subgroup $\{1\}$ or the conjugacy class of $\Stab_T(v)$, we will abuse notion and consider $\Mon(U)$ as the subgroup $\{1\}$ or $\Stab_T(v)$ respectively.  

\begin{definition}\label{def:admissible cut}
Let $T \in \calO$ be a Grushko tree, let $v \in T$ be a vertex and let $\calL$ be a periodic line collection.  The Whitehead graph $\Wh_T(\calL,v)$ has an \emph{admissible cut} if either:
\begin{itemize}\itemindent=25pt
\item[(type i)] $\Wh_T(\calL,v)$ is a disjoint union $U \sqcup V$ where:
\begin{itemize}
	\item $U$ is connected with $\Mon(U) = \{1\}$, and
    \item $V$ is nonempty if $\Stab_T(v)$ is trivial; or
\end{itemize}
\item[(type ii)] $\Wh_T(\calL,v)$ is a union $U \cup V$ where $U \cap V$ is a single vertex, $U$ is connected and $\Mon(U) = \{1\}$.
\end{itemize}
\end{definition}

We remark that the second requirement for a type i admissible cut is omitted from the definition by Guirardel--Horbez.  However, this requirement is implicit in their arguments involving the notion of an admissible cut.  We record the following characterizations of admissible cuts.  The proof of this lemma follows from the proof of Lemma~5.2 in the work of Guirardel--Horbez~\cite{ar:GH19}.   

\begin{lemma}\label{lem:whitehead}
Let $T \in \calO$ be a Grushko tree, let $v \in T$ be a vertex and let $\calL$ be a periodic line collection.
\begin{enumerate}
\item\label{item:type i expand} The Whitehead graph $\Wh_T(\calL,v)$ has a type i admissible cut if and only if there is a Grushko tree $T' \in \calO$ and a nontrivial collapse $T' \to T$ that is injective when restricted to $\ell_{T'}$ for any line $\ell \in \calL$.
\item\label{item:type ii fold} The Whitehead graph $\Wh_T(\calL,v)$ has a type ii admissible cut if and only if there is a Grushko tree $T' \in \calO$ and a fold $T' \to T$ that is injective when restricted to $\ell_{T'}$ for any line $\ell \in \calL$.
\end{enumerate}
\end{lemma}

The definition of a fold is not used the sequel.  The property that we will make use of is that if there is a fold $T' \to T$, then there is a Grushko tree $T_0$ and collapse maps $T_0 \to T$ and $T_0 \to T'$.  In particular, we have that $T$ and $T'$ are compatible and thus $d(\pi(T),\pi(T')) \leq 1$. 

\begin{lemma}\label{lem:simple}
Let $T \in \calO$ be a Grushko tree and let $g \in G$ be a non-peripheral element.  If either,
\begin{enumerate}
\item\label{item:uncrossed edge simple} there is an edge $e \in T$ that is not crossed by $aT_g$ for any $a \in G$, or
\item \label{item:type i simple} there is a vertex $v \in T$ such that $\Wh_T(\calL_g,v)$ has a type i admissible cut,
\end{enumerate}
then $g$ is simple.  Moreover, there is a free splitting $S$ for $(G,\calA)$ in which $g$ is elliptic such that $d(\pi(T),S) \leq 1$.
\end{lemma}

\begin{proof}
First we assume that \eqref{item:uncrossed edge simple} holds.  We can collapse all edges of $T$ that are not in the orbit of $e$ and obtain a free splitting $S$ for $(G,\calA)$ where $g$ is elliptic.  Hence $g$ is simple.  As $T$ collapses to $S$, we have that $d(\pi(T),S) \leq 1$.

Next we assume that \eqref{item:type i simple} holds.  By Lemma~\ref{lem:whitehead}, there is a Grushko tree $T' \in \calO$ and a collapse map $T' \to T$ that restricts to an injection $T'_g \to T_g$.  Thus there is an edge in $T'$ that is not crossed by $aT'_g$ for any $a \in G$ and as in \eqref{item:uncrossed edge simple} we conclude that $T'$ collapses to a free splitting $S$ for $(G,\calA)$ where $g$ is elliptic.  Thus $g$ is simple and as $T'$ collapses to both $T$ and $S$, we have that $d(\pi(T),S) \leq 1$.
\end{proof}

\begin{definition}\label{def:whitehead reduced}
Suppose $\calL$ is a periodic line collection.  We say a Grushko tree $T \in \calO$ is \emph{Whitehead reduced for $\calL$} if $\Wh_T(\calL,v)$ does not have an admissible cut for any vertex $v \in T$.  If $\calL = \calL_g$ for some non-peripheral element $g \in G$, we say $T$ is \emph{Whitehead reduced for $g$}.  
\end{definition}

\begin{lemma}\label{lem:whitehead reduced}
Suppose $\calL$ is a periodic line collection and that $T \in \calO$ is a Grushko tree.  Then there exists a Grushko tree $T' \in \calO$ with $d(\pi(T),\pi(T')) \leq \abs{\calL}_T$ and such that one of the following holds:
\begin{enumerate}
\item\label{item:simple} $T'$ contains an edge that is not crossed by $\ell_{T'}$ for any line $\ell \in \calL$; or
\item\label{item:reduced} $T'$ is Whitehead reduced for $\calL$ and $\abs{\calL}_{T'} \leq \abs{\calL}_T$.
\end{enumerate}
\end{lemma}

\begin{proof}
If $T$ is Whitehead reduced or contains an edge that is not crossed by $\ell_T$ for any line in $\calL$, then set $T' = T$.  

Else, there is a vertex $v \in T$ such that $\Wh_T(\calL,v)$ has an admissible cut.  By Lemma~\ref{lem:whitehead}, there is a Grushko tree $T_1 \in \calO$ and a simplicial map $T_1 \to T$ where $\abs{\calL}_{T_1} = \abs{\calL}_T$ and for which $T_1/G$ has more edges than $T/G$. (Note that $T_1$ is allowed to have vertices of degree two in the case of a fold, if the degree two vertices were removed then we would have $\abs{\calL}_{T_1} < \abs{\calL}$.) Further, we see that $d(\pi(T_1),\pi(T)) \leq 1$.  

If $T_1$ is Whitehead reduced or contains an edge that is not crossed by $\ell_{T_1}$ for any line in $\calL$, then set $T' = T_1$.  Else, we repeat this process of producing Grushko trees $T_i \in \calO$ where $\abs{\calL}_{T_i} = \abs{\calL}_T$, the number of edges in $T_i/G$ is at least $i+1$ and $d(\pi(T_i),\pi(T)) \leq i$ as long as $T_{i-1}$ is not Whitehead reduced for $\calL$ and every edge in $T_{i-1}$ is crossed by $\ell_{T_{i-1}}$ for some line in $\calL$. 

This process must terminate by $L = \abs{\calL}_T$.  Indeed, the number of orbits of edges in a hypothetical $T_L$ is at least $\abs{\calL}_T+1$ and thus there is an edge $e \subset T_{L}$ that is not crossed by $\ell_{T_L}$ for any line in $\calL$.
\end{proof}

Combining Lemmas~\ref{lem:simple} and \ref{lem:whitehead reduced}, we obtain the following corollary.

\begin{corollary}\label{co:whitehead reduced}
Suppose $g \in G$ is a non-peripheral element that is not simple, $L >0$ and $T \in \calO_L(g)$.  Then there exists a Grushko tree $T' \in \calO_L(g)$ that is Whitehead reduced for $g$ such that $d(\pi(T),\pi(T')) \leq L$.
\end{corollary}

Guirardel--Horbez show that their notion of a Whitehead graph can be used to detect simplicity of a non-peripheral element, generalizing Whitehead's well-known cut vertex criterion~\cite{ar:Whitehead36}. 

\begin{proposition}[{\cite[Proposition~5.1]{ar:GH19}}]\label{prop:whitehead criterion}
Suppose $g \in G$ is non-peripheral.  Then $g$ is simple if and only if for each Grushko tree $T \in \calO$, there is a vertex $v \in T$ such that $\Wh_T(\calL_g,v)$ has an admissible cut.
\end{proposition}


\section{Proof of Theorem~\ref{thm:bounded projections} for simple elements}\label{sec:simple}

Combining the statements from the previous section, we can prove Theorem~\ref{thm:bounded projections} for simple elements.

\begin{proposition}\label{prop:length bounded simple}
Let $(G,\calA)$ be a non-sporadic torsion-free free product.  For all $L > 0$, there is a $D_0 > 0$ such that if $g \in G$ is simple, then the diameter of $\pi(\calO_L(g)) \subset \ZF$ is at most $D_0$.
\end{proposition}

\begin{proof}
Given $L$, we set $D_0 = 2L + 3$.

Let $g \in G$ be a simple element and consider a Grushko tree $T_0 \in \calO_L(g)$.  As $g$ is simple, by Proposition~\ref{prop:whitehead criterion}, the second option of Lemma~\ref{lem:whitehead reduced} cannot occur.  Hence there is a Grushko tree $T \in \calO$ where $d(\pi(T_0),\pi(T)) \leq L$ and an edge $e \subset T$ that is not crossed by $aT_g$ for any $g \in G$.  By Lemma~\ref{lem:simple}, we obtain a free splitting $S_0 \in \ZF$ in which $g$ is elliptic and $d(\pi(T),S_0)\leq 1$.  Hence $d(\pi(T_0),S_0) \leq L + 1$.

Given any other Grushko tree $T_1 \in \calO_L(g)$, repeating the argument from above, we have that there is a free splitting $S_1 \in \ZF$ in which $g$ is elliptic and $d(\pi(T_1),S_1) \leq L + 1$. 

As $g$ is elliptic in both $S_0$ and $S_1$, we have that $d(S_0,S_1) \leq 1$.  Therefore, we find:
\begin{equation*}
d(\pi(T_0),\pi(T_1)) \leq d(\pi(T_0),S_0) + d(S_0,S_1) + d(S_1,\pi(T_1)) \leq 2L + 3 = D_0.\qedhere
\end{equation*} 
\end{proof}

\begin{remark}\label{rem:free factor bounded}
The above proof also shows that the diameter of the projection of $\calO_L(g)$ to the free factor graph $\FF$ is also bounded by $2L + 3$.  Indeed, one model of the free factor graph (denoted $\FF_2$ by Guirardel--Horbez~\cite[Definition~2.3]{ar:GH22}) is defined as the graph with vertex set the $\calZ$--splittings of $(G,\calA)$ where two are joined by an edge if they are compatible or if there is a non-peripheral \emph{simple} element $g \in G$ that is elliptic in both.  As the element $g$ in Proposition~\ref{prop:length bounded simple} is simple we have that the distance in the free factor graph between the $\calZ$--splittings $S_0$ and $S_1$ from above is also at most 1.  In the case where $\calA = \emptyset$, this was observed by Bestvina--Feighn with an argument that is similar to the one above~\cite[Lemma~3.3]{ar:BF14}.  
\end{remark}


\section{Connectivity properties of the decomposition space}\label{sec:connect}

To extend Proposition~\ref{prop:length bounded simple} to $\calZ$--simple elements that are not simple, we need to analyze the decomposition space $\calD(\calL)$ and its relation to the Whitehead graph $\Wh_T(\calL,v)$.  
Again, let $(G,\calA)$ be a fixed non-sporadic torsion-free free product.  

\subsection{Connectivity}\label{subsec:connectivity}

To begin, we observe that the Whitehead graph can detect if the decomposition space is disconnected.  Suppose that $Y$ is a direction at $p$ for some point $p$.  Recall that we denote by $\bd_0 Y$ the point $p$.  Further, we denote $\bd Y = \bd T \cap \overline{Y} \subset \bd(G,\calA)$ where $\overline{Y} \subset \hT_{\rm obs}$ is the closure.

\begin{lemma}\label{lem:decomposition disconnected}
If $\calL$ is a periodic line collection and $\Wh_T(\calL,v)$ has a type i admissible cut for some Grushko tree $T \in \calO$ and vertex $v \in T$, then $\calD(\calL)$ is disconnected.
\end{lemma}

\begin{proof}
As $\Wh_T(\calL,v)$ admits a type i admissible cut, by Lemma~\ref{lem:whitehead}\eqref{item:type i expand}, there is a Grushko tree $T' \in \calO$ that collapses to $T$ in which some edge $e \subset T'$ is not crossed by $\ell_{T'}$ for any line $\ell \in \calL$.  Let $p \in T'$ be the midpoint of $e$, and denote the directions at $p$ by $Y$ and $Y'$.  As these sets are disjoint, the sets $\bd Y, \bd Y' \subset \bd(G,\calA)$ are nonempty disjoint clopen sets.  Since $\ell_{T'}$ is disjoint from $e$ for all lines $\ell \in \calL$ we see that no line $\ell \in \calL$ has a point in both $\bd Y$ and $\bd Y'$.  Thus the images of the sets $\bd Y$ and $\bd Y'$ remain disjoint in $\calD(\calL)$ and also remain clopen.  Hence $\calD(\calL)$ is disconnected.
\end{proof}

We seek to prove a converse to this statement and hence to give a characterization of when the decomposition space $\calD(\calL)$ is connected.  

\begin{lemma}\label{lem:decomposition connected subset}
Let $\calL$ be a periodic line collection.  Suppose that $T \in \calO$ is a Grushko tree that is Whitehead reduced for $\calL$.  Consider an edge $e$ in $T$ and let $x$ be the midpoint of $e$.  If $Y \subset T - \{x\}$ is a direction at $x$, then the subset $q(\bd Y) \subset \calD(\calL)$ is connected.
\end{lemma}

\begin{proof}
Suppose there are open sets $A,B \in \calD(\calL)$ such that $q(\bd Y) \subseteq A \cup B$ and $A \cap B \cap q(\bd Y) = \emptyset$.  We will prove the lemma by showing that either $A \cap q(\bd Y) = \emptyset$ or $B \cap q(\bd Y) = \emptyset$.

We observe that the sets:
\begin{align*}
Y_A &= \bd Y \cap q^{-1}(A) = \bd Y \cap q^{-1}(\calD(\calL) - B), \\ 
Y_B &= \bd Y \cap q^{-1}(B) = \bd Y \cap q^{-1}(\calD(\calL) - A)
\end{align*} 
are both compact clopen sets in $\bd(G,\calA)$.  As the intersection $A \cap B \cap q(\bd Y)$ is empty, there is no line in $\calL$ with one point in $Y_A$ and the other point in $Y_B$.    

As $Y_A$ and $Y_B$ are compact clopen sets, there are finitely many edges in $T$, with midpoints $P = \{p_1,\ldots,p_m\}$, such that for a point $\xi \in \bd Y$, membership in $Y_A$ or $Y_B$ is determined by the parity of the intersection of the set $P$ with the ray $r$ representing $\xi$ based at $x$.  In this case, we say that the set $P$ \emph{determines} the partition $Y_A \sqcup Y_B$ of $\bd Y$.  Observe that there are many possible sets which can determine the same partition of $\bd Y$.

If $P$ is non-empty, we will describe a new set of points $P'$ that determines the same partition $Y_A \sqcup Y_B$ of $\bd Y$ such that $P'$ has strictly fewer points than $P$.  Repeating this process a finite number of times, we see that the partition $Y_A \sqcup Y_B$ can be determined by the empty set.  Hence either $Y_A$ or $Y_B$ is empty, proving that $q(\bd Y)$ is connected.

To this end, fix a point $p \in P$ that maximizes the distance to $x$, let $v$ be the vertex incident to the edge containing $p$ that is closest to $x$ and let $e'$ be the edge incident to $v$ in the direction at $v$ that contains $x$.  Let $q'$ be the midpoint of $e'$ and enumerate the midpoints of the edges incident to $v$ other than $e'$ by $q_1 = p,\ldots,q_i,\ldots$.  Note the possibility that $q' = x$.  There is an integer $M$ such that, after possibly reordering the $q_i$'s, we have that $q_i \in P$ if $1 \leq i \leq M$ and $q_i \notin P$ otherwise.  Without loss of generality, we can say that a geodesic ray originating at $x$ that contains $q_i$ for some $1 \leq i \leq M$ represents a point in $Y_A$ while a geodesic ray originating at $x$ that contains $q_i$ for some $i > M$ represents a point in $Y_B$.  (This uses the fact the $p$ maximizes the distance to $x$.) In particular, there is no line $\ell \in \calL$ such that $\ell_T$ contains that points $q_i$ and $q_j$ where $1\leq i \leq M$ and $j > M$.  Let $Y' \subset T - \{v\}$ be the direction that contains $q'$ and let $Y_i \subset  T - \{v\}$ be the direction that contains $q_i$ for $i \geq 1$.  As vertices in the Whitehead graph $\Wh_T(\calL,v)$, we see that there are no edges from $[Y_i]$ to $[Y_j]$ where $1 \leq i \leq M$ and $j > M$.

See Figure~\ref{fig:decomposition connected subset} for the set-up.

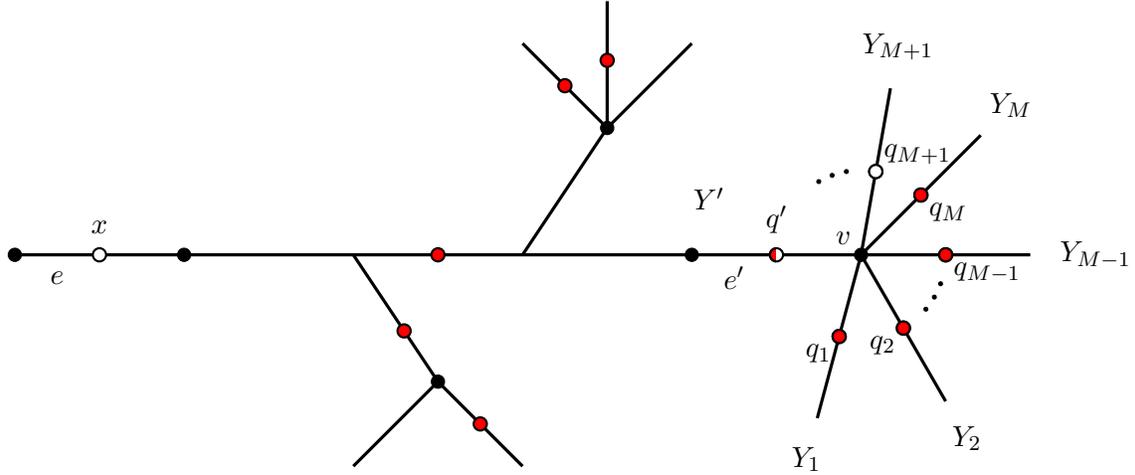
\begin{figure}[ht]
\begin{tikzpicture}
\def\ee{2.25}
\def\len{5}
\filldraw (0,0) circle [radius=2.5pt];
\draw[very thick] (0,0) -- (\len*\ee,0);
\draw[thick,fill=white] (0.5*\ee,0) circle [radius=2.5pt] node[label=above:{$x$}] {};
\draw[thick,fill=red] (\len*\ee-0.5*\ee,0) circle [radius=2.5pt] node[label=above:{$q'$}] {};
\fill[white] (\len*\ee-0.5*\ee-0.005,0.1) rectangle (\len*\ee-0.5*\ee+0.1,-0.1);
\draw[thick] (\len*\ee-0.5*\ee,0) circle [radius=2.5pt];
\filldraw (\ee,0) circle [radius=2.5pt];
\filldraw (\len*\ee-\ee,0) circle [radius=2.5pt];
\filldraw (\len*\ee,0) circle [radius=2.5pt] node[above left] {$v$};
\node at (0.25*\ee,-0.3) {$e$};
\node at (\len*\ee - 0.75*\ee,-0.3) {$e'$};
\begin{scope}[xshift=\len*\ee cm]
\foreach \x\y in {-15,30,90,135,170} {
	\draw[very thick,rotate=\x] (0,0) -- (0,-\ee);
	\draw[thick,fill=white,rotate=\x] (0,-0.5*\ee) circle [radius=2.5pt];
}
\draw[thick,fill=red,rotate=-15] (0,-0.5*\ee) circle [radius=2.5pt] node[below left=-0.05] {$q_1$};
\fill[white,rotate=-15] (0,-1.25*\ee) circle [radius=0.5pt] node[black] {$Y_1$};
\draw[thick,fill=red,rotate=30] (0,-0.5*\ee) circle [radius=2.5pt] node[below left=-0.05] {$q_2$};
\fill[white,rotate=30] (0,-1.25*\ee) circle [radius=0.5pt] node[black] {$Y_2$};
\draw[thick,fill=red,rotate=90] (0,-0.5*\ee) circle [radius=2.5pt] node[below right=-0.05] {$q_{M-1}$};
\fill[white,rotate=90] (0,-1.25*\ee) circle [radius=0.5pt] node[black,right=-0.3] {$Y_{M-1}$};
\draw[thick,fill=red,rotate=135] (0,-0.5*\ee) circle [radius=2.5pt] node[below right=-0.05] {$q_M$};
\fill[white,rotate=135] (0,-1.25*\ee) circle [radius=0.5pt] node[black] {$Y_M$};
\draw[thick,fill=white,rotate=170] (0,-0.5*\ee) circle [radius=2.5pt] node[above right=-0.05] {$q_{M+1}$};
\fill[white,rotate=170] (0,-1.25*\ee) circle [radius=0.5pt] node[black] {$Y_{M+1}$};
\foreach \x in {50,60,70,190,200,210} {
	\filldraw[rotate=\x] (0,-0.5*\ee) circle [radius=0.75pt];
}
\end{scope}
\draw[very thick] (2*\ee,0) -- (2.5*\ee,-0.75*\ee) -- (3*\ee,-1.25*\ee);
\draw[very thick] (2.5*\ee,-0.75*\ee) -- (2*\ee,-1.25*\ee);
\filldraw (2.5*\ee,-0.75*\ee) circle [radius=2.5pt];
\draw[thick,fill=red] (2.3*\ee,-0.45*\ee) circle [radius=2.5pt];
\draw[thick,fill=red] (2.75*\ee,-\ee) circle [radius=2.5pt];
\draw[thick,fill=red] (2.5*\ee,0) circle [radius=2.5pt];
\draw[very thick] (3*\ee,0) -- (3.5*\ee,0.75*\ee) -- (4*\ee,1.25*\ee);
\draw[very thick] (3*\ee,1.25*\ee) -- (3.5*\ee,0.75*\ee) -- (3.5*\ee,1.5*\ee);
\filldraw (3.5*\ee,0.75*\ee) circle [radius=2.5pt];
\draw[thick,fill=red] (3.25*\ee,\ee) circle [radius=2.5pt];
\draw[thick,fill=red] (3.5*\ee,1.15*\ee) circle [radius=2.5pt];
\node at (4.1*\ee,0.75) {$Y'$};
\end{tikzpicture}
\caption{The set-up in the proof of Lemma~\ref{lem:decomposition connected subset}.  Vertices of $T$ are filled in black, points that below in $P$ are filled in red and other midpoints of edges are filled in white.  The point $q'$ may or not not belong to $P$.}\label{fig:decomposition connected subset}
\end{figure}

First suppose that $\Stab_T(v) = \{1\}$ and let $\delta$ be the valence of $v$.   If $M < \delta-1$, then we see that $\Wh_T(\calL,v)$ has an admissible cut, contrary to our hypotheses.  Indeed, either $\Wh_T(\calL,v)$ is disconnected or $[Y']$ is a cut vertex.  Hence we have that $M = \delta-1$ and so each of the $q_i$'s belong to $P$.  In this case, let $P'$ be the symmetric difference between $P$ and the set $\{q',q_1,\ldots,q_M\} - \{x\}$.  In other words, we replace $q_1,\ldots,q_M$ by $q'$ if $q' \notin P$ and $q' \neq x$, and remove $q',q_1,\ldots,q_M$ otherwise.

We are now left with the case that $\#|\Stab_T(v)| = \infty$.  We will argue that this case cannot occur.  Let $Y_{i_1},\ldots,Y_{i_m}$ be a sequence of directions at $v$ where:
\begin{enumerate}
\item there is a line $\ell_j \in \calL$ such that $(\ell_j)_T$ contains $q_{i_j}$ and $q_{i_{j+1}}$ for $1 \leq j < m$, and
\item $Y_{i_1} = Y_1$ and $Y_{i_m} = aY_1$ for some $a \in \Stab_T(v)$.
\end{enumerate} 
Such a sequence gives rise to a closed path in $\Wh_T(\calL,v)$ based at $[Y_1]$.  Moreover, the product of the labels along this path in $\Wh_T(\calL,v)$ is $a$.  

\medskip \noindent {\it Claim.} The element $a$ is trivial.

\begin{proof}[Proof of Claim]
We observe that $1 \leq i_j \leq M$ for all $1 \leq j \leq m$ as there is no line $\ell \in \calL$ such that $\ell_T$ contains the points $q_i$ and $q_j$ where $1\leq i \leq M$ and $j > M$.  In particular, $aq_1 \in P$.  We can concatenate the original sequence with $aY_1, aY_{i_2},\ldots,aY_{i_m} = a^2Y_1$ to observe that $a^2q_1 \in P$ also.  Repeating in this manner we have that $a^nq_1 \in P$ for all $n \geq 0$.  As $P$ is a finite set, we must have that $a$ is trivial as claimed.
\end{proof}
   
This implies that $\Wh_T(\calL,v)$ has an admissible cut, contrary to our hypothesis.  Indeed, let $U$ be the connected component of the subgraph spanned by the vertices $[Y_1],\ldots,[Y_M]$ that contains the vertex $[Y_1]$.  By the claim we have that $\Mon(U) = \{1\}$.  If $U$ is a connected component of $\Wh_T(\calL,v)$, then there is a type i admissible cut.  Else, let $U'$ be the subgraph spanned by $U$ and $[Y']$.  Arguing as in the claim, we can also see that $\Mon(U') = \{1\}$.  In this case it is now clear that there is a type ii admissible cut.           
\end{proof}

We can now state our converse to Lemma~\ref{lem:decomposition disconnected}.

\begin{proposition}\label{prop:decomposition connected}
Let $\calL$ be a periodic line collection.  The following are equivalent.
\begin{enumerate}
\item\label{item:whitehead reduced} There exists a Grushko tree $T \in \calO$ that is Whitehead reduced for $\calL$.

\item\label{item:connected} The decomposition space $\calD(\calL)$ is connected.
\end{enumerate}
\end{proposition}

\begin{proof}
Suppose that there exists a Grushko tree $T \in \calO$ that is Whitehead reduced for $\calL$. Fix an edge $e \subset T$, let $x$ be the midpoint of $e$ and denote the directions at $x$ by $Y$ and $Y'$  By Lemma~\ref{lem:decomposition connected subset}, both $q(\bd Y)$ and $q(\bd Y')$ are connected.  Further, there is a geodesic that crosses the edge $e$ as otherwise the Whitehead graph based at either of the vertices incident to $e$ has an isolated vertex and hence an admissible cut.  Thus $q(\bd Y) \cap q(\bd Y') \neq \emptyset$ and therefore $\calD_\calL = q(\bd Y) \cup q(\bd Y')$ is connected.  Hence \eqref{item:whitehead reduced} implies \eqref{item:connected}.

Next suppose that there does not exist a Grushko tree $T \in \calO$ that is Whitehead reduced for $\calL$.  By Lemma~\ref{lem:whitehead reduced}, there is a Grushko tree $T' \in \calO$ that contains an edge that is not crossed by any geodesic in $\calL$.  As the proof of Lemma~\ref{lem:decomposition disconnected} shows, this implies that $\calD(\calL)$ is disconnected.  Hence  \eqref{item:connected} implies \eqref{item:whitehead reduced}.       
\end{proof}

We obtain the following corollary of Proposition~\ref{prop:decomposition connected} using Corollary~\ref{co:whitehead reduced}.

\begin{corollary}\label{co:not simple}
Suppose $g \in G$ is a non-peripheral element.  Then $g$ is not simple if and only if $\calD(\calL_g)$ is connected.
\end{corollary}

%
%
%

\subsection{Cut sets in decomposition spaces}\label{subsec:cut sets}

In this section we introduce the terminology needed to talk about cut sets in decomposition spaces and record how $\calZ$--splittings give rise to cut sets.

Let $D$ be a Peano continuum, i.e., a compact, connected, locally connected, metrizable space.  A finite set $P \subset D$ is a \emph{cut set} if $D - P$ is disconnected but $D - P'$ is connected for any proper subset $P' \subset P$.  We use the terms \emph{cut point}, respectively \emph{cut pair}, when $\#\abs{P} = 1$, or when $\#\abs{P} = 2$.  A \emph{local cut point} is a point $x \in D$ which is either a cut point or where $D - \{x\}$ is connected and has more than one end.  In the latter case, the \emph{valence} of $x$ is the number of ends of $D - \{x\}$, denoted $\val(x)$.  A cut pair $\{x,y\} \subset D$ is \emph{exact} if $\val(x) = \val(y) = \#\abs{\pi_0(D - \{x,y\})}$.  In other words, each component of $D - \{x,y\}$ has two ends, one that limits to $x$ and one that limits to $y$.  It is known that the valence of a point in an exact cut pair must be finite~\cite[Section~2.2]{ar:HH}.  The \emph{valence} of an exact cut pair is the valence of either of its two points.  We will use the terms \emph{bivalent}, or \emph{multivalent} respectively, to mean that the valence is equal to 2, or greater than or equal to 3.

An exact cut pair $\{y_0,y_1\} \subset D$ \emph{separates} an exact cut pair $\{x_0,x_1\} \subset D$, if $x_0$ and $x_1$ lie in different components of $D - \{y_0,y_1\}$.  It can be readily checked that separating is a symmetric notion among exact cut pairs~\cite[Lemma~2.3]{ar:Cashen16}. An exact cut pair $\{x_0,x_1\} \subset D$ is \emph{inseparable} if it is not separated by another exact cut pair.  As the notion of separating is symmetric, if $\{x_0,x_1\}, \{y_0,y_1\} \subset D$ are exact pairs and $\{x_0,x_1\}$ is inseparable, then we have that both $x_0$, $x_1$ belong to the same component of $D - \{y_0,y_1\}$ and $y_0$, $y_1$ belong to the same component of $D - \{x_0,x_1\}$.

Recently, Dasgupta--Hruska proved that the Bowditch boundary of a countable relatively hyperbolic group is a Peano continuum whenever it is connected~\cite[Theorem~1.1]{ar:DH}.  Hence the above discussion, by Lemma~\ref{lem:boundary}, applies to a connected decomposition space $\calD(\calL)$.  A \emph{loxodromic cut pair} in a decomposition space $\calD(\calL)$ is a cut pair of the form $P = \{q(a^\infty),q(a^{-\infty})\}$ where $a$ is a non-peripheral element of $(G,\calA \cup \calN_\calL)$.  We remark that loxodromic cut pairs are always exact~\cite[Lemma~4.1]{ar:HH}.

\begin{lemma}\label{lem:cut set}
Suppose that $g \in G$ is a $\calZ$--simple element that is not simple.  Then $\calD(\calL_g)$ contains a loxodromic cut pair.
\end{lemma}

\begin{proof}
Suppose that $g$ is $\calZ$--simple but not simple and let $S$ be a $\calZ$--splitting in which $g$ is elliptic.  As $g$ is not simple, every edge in $S$ has nontrivial stabilizer.  Fix an element $a \in G$ that generates the stabilizer of an edge in $S$.   

As $g$ is not simple, we have that $\calD(\calL_g)$ is connected by Corollary~\ref{co:not simple}.  Let $\calL = \calL_g \cup \calL_a$.  Then as $\calD(\calL)$ is a quotient of $\calD(\calL_g)$, it is also connected and by Lemma~\ref{lem:boundary}, it is homeomorphic to the Bowditch boundary of the relatively hyperbolic group $(G,\calA \cup \calN_\calL)$.  As $(G,\calA \cup \calN_\calL)$ splits over $\I{a}$, which is a peripheral subgroup of $(G,\calA \cup \calN_\calL)$, by work of Dasgupta--Hruska~\cite[Theorem~1.1]{ar:DH} we find that $\calD(\calL)$ contains a cut point which is the fixed point of the subgroup $\I{a}$.  The pre-image of this cut point in $\calD(\calL_g)$ is the two point set $\{q(a^\infty),q(a^{-\infty})\}$.  This pair is thus a loxodromic cut pair.
\end{proof}

\begin{remark}\label{rem:cut set}
When $\calA = \emptyset$, so that $G$ is a free group, it is easy to see the conclusion of Lemma~\ref{lem:cut set} directly.  Indeed, for simplicity, assume that there is a $\calZ$--splitting corresponding to an amalgamated free product decomposition $G = A \ast_{\I{a}} B$ where $g \in A$.  As $g$ is not simple, we have that $g$ is not commensurable to $a$.  There is a 2--complex $M$ obtained from two graphs $\Gamma_A$ and $\Gamma_B$ where $\pi_1(\Gamma_A) \cong A$ and $\pi_1(\Gamma_B) \cong B$ by attaching a cylinder $S^1 \times [0,1]$ where $S^1 \times \{0\}$ is glued to a loop in $\Gamma_A$ representing $a \in A$ and $S^1 \times \{1\}$ is glued to a loop in $\Gamma_B$ representing $a \in B$.  The universal cover $\tM$ is quasi-isometric to $G$ and so we can identify $\bd \tM$ with $\bd G$.  Then $S^1 \times \{1/2\} \subset M$ lifts to an embedded line $L \subset \tM$ whose endpoints correspond to $a^\infty, a^{-\infty} \in \bd G$.  By construction, $\bd G$ decomposes into two sets $C_0$ and $C_1$, corresponding to the Gromov boundaries of the two components of $\tM - L$, where $C_0 \cap C_1 = \{a^\infty, a^{-\infty}\}$.  Let $C'_0 = C_0 - \{a^\infty,a^{-\infty}\}$ and $C'_1 = C_1 - \{a^\infty,a^{-\infty}\}$.  For each $h \in G$, the pair $\{hg^\infty,hg^{-\infty}\}$ belongs either to $C'_0$ or to $C'_1$.  Hence $q(C'_0)$ and $q(C'_1)$ are open disjoint sets and $\calD(\calL_g) - \{q(a^\infty),q(a^{-\infty})\} = q(C'_0) \cup q(C'_1)$ showing that $\{q(a^\infty),q(a^{-\infty})\}$ is a cut pair.  
\end{remark}


\section{Modeling the decomposition space via Whitehead graphs}\label{sec:model}

We require a notion of the Whitehead graph based on the complement of a locally finite subtree $X \subset T$ in order to model the decomposition space.  In the setting where $\calA = \emptyset$, this idea was fully developed and investigated by Cashen--Macura~\cite{ar:CM11} (see the references within as well for early work in this direction).  In that setting, the definition using a locally finite subtree $X \subset T$ is exactly as it is for a sole vertex  $v \in T$: the vertices of the Whitehead graph are the directions of $T - X$ and two are joined by an edge for each line $\ell \in \calL$ such that $\ell_T$ meets both directions.  There is a way to build up such graphs by ``splicing'' together Whitehead graphs based at vertices.  In the general free product setting the naive approach does not work as sometimes the $\Stab_T(\bd_0Y)$--orbit of a component $Y \subset T - X$ meets $X$, which complicates this picture.  

After first defining our notion of a Whitehead graph in this general setting in Section~\ref{subsec:whitehead graph}, we relate it to cut sets in the decomposition space in Lemma~\ref{lem:hull} and Section~\ref{subsec:identify cut sets}.  

For the remained of this section, let $(G,\calA)$ be a fixed non-sporadic torsion-free free product.  

\subsection{Whitehead graph relative to a locally finite subtree}\label{subsec:whitehead graph}

We begin by producing a structure modeling the relations between the lines and directions at a vertex with non-trivial stabilizer.

Suppose $T \in \calO$ is a Grushko tree, $\calL$ is a periodic line collection, and $v \in T$ is a vertex with non-trivial stabilizer.  For simplicity, we will assume that $T$ is Whitehead reduced for $\calL$ as this is the only case we are concerned with in the sequel.  We define a graph $T_v(\calL)$ whose vertices correspond to directions at $v$ and two such directions $Y,Y' \subset T - \{v\}$ are connected by an edge for each line $\ell \in \calL$ such that $\ell_T$ meets both $Y$ and $Y'$.  Note that $T_v(\calL)$ is locally finite as $\calL$ has finitely many $G$--orbits and hence each edge $e \subset T$ meets only finitely many $\ell_T$.  

Suppose $\tU \subseteq T_v(\calL)$ is a connected subgraph and $G_\tU$ is the subgroup of $\Stab_v(T)$ that leaves $\tU$ invariant.  Then $G_\tU$ acts freely on $\tU$ and the quotient $\tU/G_\tU$ is isomorphic to a connected component $U$ of the Whitehead graph $\Wh_T(\calL,v)$.  Moreover, $G_\tU$ is in the conjugacy class $\Mon(U)$.  In particular, as $\Wh_T(\calL,v)$ does not have a type i admissible cut and each of the peripheral subgroups are torsion-free, each component of $T_v(\calL)$ contains infinitely many vertices.  

Consider a locally finite subtree $X \subset T$.  Note that we allow $X$ to have vertices of degree one that are not vertices of $T$.  We define $T_v(\calL) - X$ to be the subgraph of $T_v(\calL)$ spanned by the directions at $v$ that are disjoint from $X$.  That is, we remove from $T_v(\calL)$ the directions at $v$ that meet $X$ and all of these incident edges.  We will use the structure of this graph to create a set, based on the components of $T_v(\calL) - X$, that will appear as vertices in the Whitehead graph $\Wh_T(\calL,X)$.  Specifically, by $V_X(\calL,v)$ we denote the union of the vertices of $T_v(\calL) - X$ that belong to a finite component of $T_v(\calL) - X$ with an additional element denoted $\hv$.  The element $\hv$ represents each direction at $v$ that belongs to an infinite component of $T_v(\calL) - X$.  As $X$ is a locally finite tree, $T_v(\calL)$ is locally finite, and each component of $T_v(\calL)$ contains infinitely many vertices, the set $V_X(\calL,v)$ is a finite set.  

The Whitehead graph based at $X$ is defined as follows.  The vertex set of $\Wh_T(\calL,X)$ is the union of $V_X(\calL,v)$ over all vertices $v$ of $T$ contained in $X$ with infinite stabilizer, together with the set of directions $Y \subset T-X$ where $\Stab_T(\bd_0 Y)$ is trivial.  Two distinct vertices of $\Wh_T(\calL,X)$ are connected by an edge for each line $\ell \in \calL$ such that $\ell_T$ has non-trivial intersection with the directions represented by the two vertices.  As was the case for the version of the Whitehead graph defined in Section~\ref{sec:whitehead}, the vertex representing a direction $Y \subset T - X$ is denoted by $[Y]$.  The vertex $[Y]$ in $\Wh_T(\calL,X)$ is called \emph{trivial} if $\Stab_T(\bd_0 Y) = \{1\}$ and is called \emph{infinite} if $\#|\Stab_T(\bd_0 Y)| = \infty$.  We note that this dichotomy is well-defined as $[Y] = [Y']$ implies that $\bd_0 Y = \bd_0 Y'$.     

When $X$ is a single vertex $v$ where $\Stab_T(v) = \{1\}$ this notion of a Whitehead graph agrees with the definition of Guirardel--Horbez given above, which is the classical notion of a Whitehead graph.  When $\#\abs{\Stab_T(v)} = \infty$ the above notion of Whitehead graph will always consist of a single vertex and no edges.  We will never make use of this notion.  Hence, in the sequel when we are referring to a Whitehead graph relative to a single vertex, i.e., $\Wh_T(\calL,v)$, we are always using the Guirardel--Horbez definition.

\begin{example}\label{ex:whitehead graph subtree}
We consider the free product $(G,\calA)$, Grushko tree $T \in \calO$ and element $g \in G$ as in Example~\ref{ex:whitehead graph}.  We will also make use of the notation from that example.  The graph $T_v(\calL_g)$ is the disjoint union of four lines as shown in Figure~\ref{fig:TvL}.

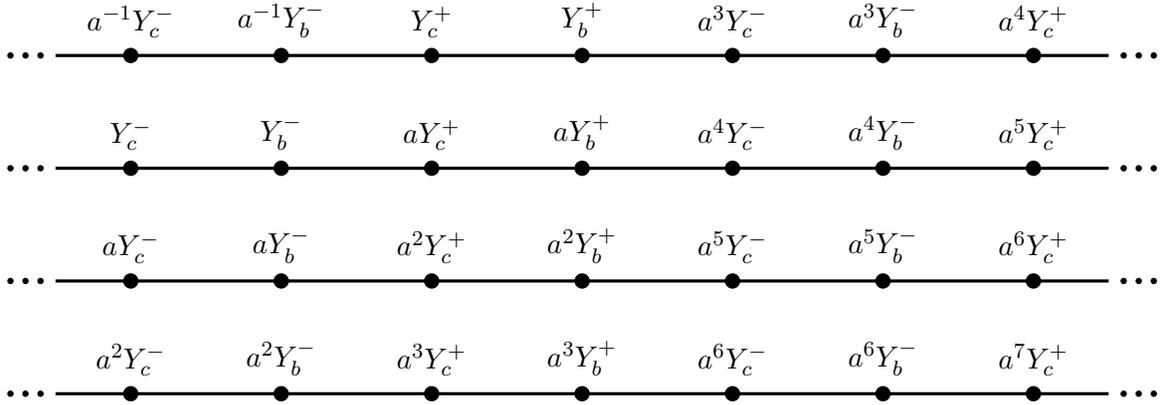
\begin{figure}[h]
\begin{tikzpicture}
\def\ee{2}
\def\len{6}
\filldraw (-0.8*\ee,0) circle [radius=1pt];
\filldraw (-0.7*\ee,0) circle [radius=1pt];
\filldraw (-0.6*\ee,0) circle [radius=1pt];
\filldraw (\len*\ee + 0.8*\ee,0) circle [radius=1pt];
\filldraw (\len*\ee + 0.7*\ee,0) circle [radius=1pt];
\filldraw (\len*\ee + 0.6*\ee,0) circle [radius=1pt];
\draw[very thick] (-0.5*\ee,0) -- (\len*\ee+0.5*\ee,0);
\draw[thick,fill=black] (0*\ee,0) circle [radius=2.5pt] node[above=0.1] {$a^{-1}Y_c^-$};
\draw[thick,fill=black] (1*\ee,0) circle [radius=2.5pt] node[above=0.1] {$a^{-1}Y_b^-$};
\draw[thick,fill=black] (2*\ee,0) circle [radius=2.5pt] node[above=0.1] {$Y_c^+$};
\draw[thick,fill=black] (3*\ee,0) circle [radius=2.5pt] node[above=0.1] {$Y_b^+$};
\draw[thick,fill=black] (4*\ee,0) circle [radius=2.5pt] node[above=0.1] {$a^3Y_c^-$};
\draw[thick,fill=black] (5*\ee,0) circle [radius=2.5pt] node[above=0.1] {$a^3Y_b^-$};
\draw[thick,fill=black] (6*\ee,0) circle [radius=2.5pt] node[above=0.1] {$a^4Y_c^+$};
\begin{scope}[yshift=-1.5cm]
\filldraw (-0.8*\ee,0) circle [radius=1pt];
\filldraw (-0.7*\ee,0) circle [radius=1pt];
\filldraw (-0.6*\ee,0) circle [radius=1pt];
\filldraw (\len*\ee + 0.8*\ee,0) circle [radius=1pt];
\filldraw (\len*\ee + 0.7*\ee,0) circle [radius=1pt];
\filldraw (\len*\ee + 0.6*\ee,0) circle [radius=1pt];
\draw[very thick] (-0.5*\ee,0) -- (\len*\ee+0.5*\ee,0);
\draw[thick,fill=black] (0*\ee,0) circle [radius=2.5pt] node[above=0.1] {$Y_c^-$};
\draw[thick,fill=black] (1*\ee,0) circle [radius=2.5pt] node[above=0.1] {$Y_b^-$};
\draw[thick,fill=black] (2*\ee,0) circle [radius=2.5pt] node[above=0.1] {$aY_c^+$};
\draw[thick,fill=black] (3*\ee,0) circle [radius=2.5pt] node[above=0.1] {$aY_b^+$};
\draw[thick,fill=black] (4*\ee,0) circle [radius=2.5pt] node[above=0.1] {$a^4Y_c^-$};
\draw[thick,fill=black] (5*\ee,0) circle [radius=2.5pt] node[above=0.1] {$a^4Y_b^-$};
\draw[thick,fill=black] (6*\ee,0) circle [radius=2.5pt] node[above=0.1] {$a^5Y_c^+$};
\end{scope}
\begin{scope}[yshift=-3cm]
\filldraw (-0.8*\ee,0) circle [radius=1pt];
\filldraw (-0.7*\ee,0) circle [radius=1pt];
\filldraw (-0.6*\ee,0) circle [radius=1pt];
\filldraw (\len*\ee + 0.8*\ee,0) circle [radius=1pt];
\filldraw (\len*\ee + 0.7*\ee,0) circle [radius=1pt];
\filldraw (\len*\ee + 0.6*\ee,0) circle [radius=1pt];
\draw[very thick] (-0.5*\ee,0) -- (\len*\ee+0.5*\ee,0);
\draw[thick,fill=black] (0*\ee,0) circle [radius=2.5pt] node[above=0.1] {$aY_c^-$};
\draw[thick,fill=black] (1*\ee,0) circle [radius=2.5pt] node[above=0.1] {$aY_b^-$};
\draw[thick,fill=black] (2*\ee,0) circle [radius=2.5pt] node[above=0.1] {$a^2Y_c^+$};
\draw[thick,fill=black] (3*\ee,0) circle [radius=2.5pt] node[above=0.1] {$a^2Y_b^+$};
\draw[thick,fill=black] (4*\ee,0) circle [radius=2.5pt] node[above=0.1] {$a^5Y_c^-$};
\draw[thick,fill=black] (5*\ee,0) circle [radius=2.5pt] node[above=0.1] {$a^5Y_b^-$};
\draw[thick,fill=black] (6*\ee,0) circle [radius=2.5pt] node[above=0.1] {$a^6Y_c^+$};
\end{scope}
\begin{scope}[yshift=-4.5cm]
\filldraw (-0.8*\ee,0) circle [radius=1pt];
\filldraw (-0.7*\ee,0) circle [radius=1pt];
\filldraw (-0.6*\ee,0) circle [radius=1pt];
\filldraw (\len*\ee + 0.8*\ee,0) circle [radius=1pt];
\filldraw (\len*\ee + 0.7*\ee,0) circle [radius=1pt];
\filldraw (\len*\ee + 0.6*\ee,0) circle [radius=1pt];
\draw[very thick] (-0.5*\ee,0) -- (\len*\ee+0.5*\ee,0);
\draw[thick,fill=black] (0*\ee,0) circle [radius=2.5pt] node[above=0.1] {$a^2Y_c^-$};
\draw[thick,fill=black] (1*\ee,0) circle [radius=2.5pt] node[above=0.1] {$a^2Y_b^-$};
\draw[thick,fill=black] (2*\ee,0) circle [radius=2.5pt] node[above=0.1] {$a^3Y_c^+$};
\draw[thick,fill=black] (3*\ee,0) circle [radius=2.5pt] node[above=0.1] {$a^3Y_b^+$};
\draw[thick,fill=black] (4*\ee,0) circle [radius=2.5pt] node[above=0.1] {$a^6Y_c^-$};
\draw[thick,fill=black] (5*\ee,0) circle [radius=2.5pt] node[above=0.1] {$a^6Y_b^-$};
\draw[thick,fill=black] (6*\ee,0) circle [radius=2.5pt] node[above=0.1] {$a^7Y_c^+$};
\end{scope}
\end{tikzpicture}
\caption{The graph $T_v(\calL_g)$ in Example~\ref{ex:whitehead graph subtree}.}\label{fig:TvL}
\end{figure}

Subdivide the edge $e_1$ into $e_1^+e_1^-$ and likewise subdivide $e_2$ into $e_2^+e_2^-$.  Let $X$ be the subtree of $T$ that is the union of the edges $a^{-1}b^{-1}e_1^-$, $e_1$, $bc^{-1}e_2^-$ and $ba^{-4}b^{-1}e_1^-$.  There are two vertices of $T$ that belong to $X$, namely $v$ and $bv$.  For simplicity, we denote $bv$ by $w$.  There are only two vertices of $T_v(\calL_g)$ that meet $X$, specifically $a^{-1}Y_b^-$ and $Y_b^+$.  Referring to Figure~\ref{fig:TvL}, we see that removing these vertices from $T_v(\calL_g)$ results in a graph with six connected components: three that are lines, two that are rays and one that only consists of the single vertex $Y_c^+$.  Thus $V_X(\calL_g,v)$ consists of two elements: $Y_c^+$ and $\hv$.  

Translation by $b$ defines a bijection between the directions at $v$ and the directions at $w$.  We use the letter ``$Z$'' to denote the corresponding directions at $w$, e.g., $Z_b^+ = bY_b^+$, $aZ_c^- = b(aY_c^-)$, etc.  We see that $X$ meets the directions $Z_b^-$, $Z_c^-$, and $a^{-4}Z_b^-$.  Using the Figure for $T_v(\calL_g)$ in Figure~\ref{fig:TvL}, we find that $V_X(\calL_g,w)$ consists of three elements: $a^{-3}Z_b^+$, $a^{-3}Z_c^+$, and $\hw$.  

There are three directions $Y \subset T - X$ where $\Stab_T(\bd_0 Y)$ is trivial.  The direction contained in $a^{-1}Y_b^-$ is denoted $a^{-1}W_b^-$ (note: $a^{-1}Y_b^- = a^{-1}b^{-1}e_1^- \cup a^{-1}W_b^-$).  Likewise we denote the direction contained in $Z_c^-$ by $W_c^-$ and the direction contained in $a^{-4}Z_b^-$ by $a^{-4}W_b^-$.  This set-up is shown in Figure~\ref{fig:whitehead set-up}.  The Whitehead graph $\Wh_T(\calL_g,X)$ is shown in Figure~\ref{fig:whitehead graph subtree}.

\begin{figure}[h]
\begin{tikzpicture}
\draw[line width=2.5mm,yellow] (-1,1) --(0,0) -- (7,0) -- (6,1);
\draw[line width=2.5mm,yellow] (7,0) -- (8,-1);
\draw[very thick,->-] (0,0) -- (7,0) node[pos=0.5,above=0.5] {$e_1$};
\def\ee{2.5}
\def\sqr{0.707106781}
\def\os{0.25}
\begin{scope}
\draw[very thick,->-=0.5,rotate=135] (\ee,0) -- (0,0) node[pos=0,label=above:{$a^{-1}W_b^- \subset a^{-1}Y_b^-$}] (a1) {};
\filldraw (a1) circle [radius=2pt];
\draw[very thick,->-=0.6,rotate=45] (0,0) -- (\ee,0) node[pos=1,label=above right:{$Y_c^+$}] (a2) {};
\filldraw (a2) circle [radius=2pt];
\draw[very thick,->-=0.5,rotate=-45] (\ee,0) -- (0,0) node[pos=0,label=below right:{$a^3Y_c^-$}] (a3) {};
\filldraw (a3) circle [radius=2pt];
\draw[very thick,->-=0.5,rotate=-135] (\ee,0) -- (0,0) node[pos=0,label=below left:{$a^{-1}Y_c^-$}] (a4) {};
\filldraw (a4) circle [radius=2pt];
\draw[thick,red] (-\sqr*\ee+\sqr*\os,\sqr*\ee+\sqr*\os) -- (0,2*\sqr*\os) -- (\sqr*\ee-\sqr*\os,\sqr*\ee+\sqr*\os);
\draw[thick,blue] (\sqr*\ee+\sqr*\os,\sqr*\ee-\sqr*\os) -- (3.4*\sqr*\os,1.4*\sqr*\os) -- (3.5,1.4*\sqr*\os); 
\draw[thick,teal] (\sqr*\ee+\sqr*\os,-\sqr*\ee+\sqr*\os) -- (3.4*\sqr*\os,-1.4*\sqr*\os) -- (3.5,-1.4*\sqr*\os); 
\draw[thick,violet] (-\sqr*\ee-\sqr*\os,\sqr*\ee-\sqr*\os) -- (-2*\sqr*\os,0) -- (-\sqr*\ee-\sqr*\os,-\sqr*\ee+\sqr*\os);
\filldraw (0,0) circle [radius=2.5pt];
\end{scope}
\begin{scope}[xshift=7cm]
\draw[very thick,->-=0.5,rotate=135] (\ee,0) -- (0,0) node[pos=0,label=above:{$W_c^- \subset Z_c^-$}] (b1) {};
\filldraw (b1) circle [radius=2pt];
\draw[very thick,->-=0.6,rotate=90] (0,0) -- (\ee,0) node[pos=1,label=above:{$a^{-3}Z_b^+$}] (b2) {};
\filldraw (b2) circle [radius=2pt];
\draw[very thick,->-=0.6,rotate=45] (0,0) -- (\ee,0) node[pos=1,label=above right:{$a^{-3}Z_c^+$}] (b3) {};
\filldraw (b3) circle [radius=2pt];
\draw[very thick,->-=0.5,rotate=-45] (\ee,0) -- (0,0) node[pos=0,label={[xshift=0.5cm,yshift=-1cm]$a^{-4}W_b^- \subset a^{-4}Z_b^-$}] (b4) {};
\filldraw (b4) circle [radius=2pt];
\draw[very thick,->-=0.5,rotate=-90] (\ee,0) -- (0,0) node[pos=0,label=below:{$a^{-4}Z_c^-$}] (b5) {};
\filldraw (b5) circle [radius=2pt];
\draw[very thick,->-=0.5,rotate=-135] (\ee,0) -- (0,0) node[pos=0,label=below left:{$aZ_c^+$}] (b6) {};
\filldraw (b6) circle [radius=2pt];
\draw[thick,blue] (-3.5,1.4*\sqr*\os) -- (-3.4*\sqr*\os,1.4*\sqr*\os) -- (-\sqr*\ee-\sqr*\os,\sqr*\ee-\sqr*\os); 
\draw[thick,black!30!green] (-\sqr*\ee+\sqr*\os,\sqr*\ee+\sqr*\os) -- (-\os,2*\sqr*\os+\os) -- (-\os,\ee);
\draw[thick,purple] (\os,\ee) -- (\os,2*\sqr*\os+\os) -- (\sqr*\ee-\sqr*\os,\sqr*\ee+\sqr*\os);
\draw[thick,cyan] (\sqr*\ee+\sqr*\os,\sqr*\ee-\sqr*\os) -- (2*\sqr*\os,0) -- (\sqr*\ee+\sqr*\os,-\sqr*\ee+\sqr*\os);
\draw[thick,brown] (\os,-\ee) -- (\os,-2*\sqr*\os-\os) -- (\sqr*\ee-\sqr*\os,-\sqr*\ee-\sqr*\os);
\draw[thick,teal] (-3.5,-1.4*\sqr*\os) -- (-3.4*\sqr*\os,-1.4*\sqr*\os) -- (-\sqr*\ee-\sqr*\os,-\sqr*\ee+\sqr*\os); 
\filldraw (0,0) circle [radius=2.5pt];
\end{scope}
\end{tikzpicture}
\caption{The set-up for the subtree $X$, shown in yellow, and some translates of $T_g$ in Example~\ref{ex:whitehead graph subtree}.}\label{fig:whitehead set-up}
\end{figure}
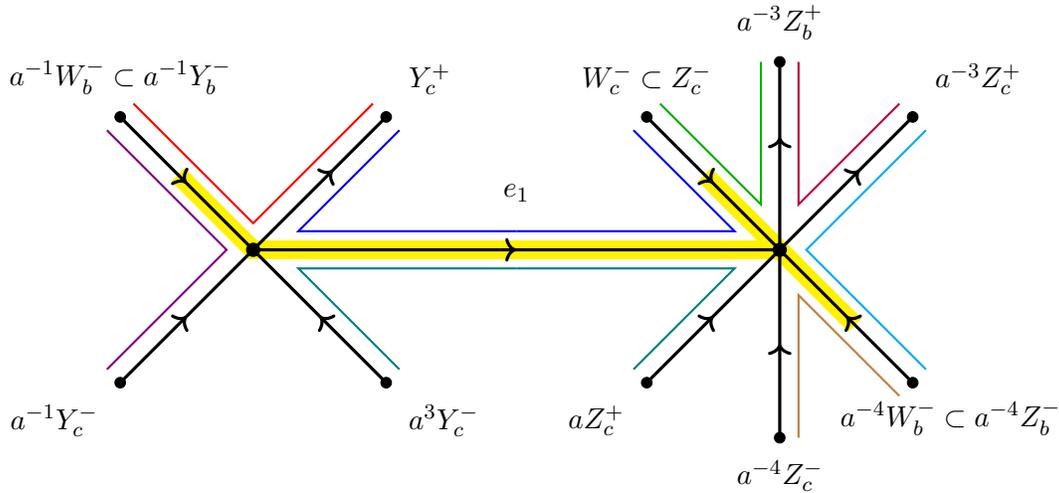

\begin{figure}[h]
\begin{tikzpicture}
\node[draw=none,minimum size=5cm,regular polygon,regular polygon sides=8] at (0,0) (a) {};
\draw[very thick,red] (a.corner 1) -- (a.corner 2);
\draw[very thick,violet] (a.corner 2) -- (a.corner 3);
\draw[very thick,teal] (a.corner 3) -- (a.corner 4);
\draw[very thick,brown] (a.corner 4) -- (a.corner 5);
\draw[very thick,cyan] (a.corner 5) -- (a.corner 6);
\draw[very thick,purple] (a.corner 6) -- (a.corner 7);
\draw[very thick,black!30!green] (a.corner 7) -- (a.corner 8);
\draw[very thick,blue] (a.corner 8) -- (a.corner 1);
\node at (a.corner 1) [label=above:{$[Y_c^+]$}] {};
\node at (a.corner 2) [label=above:{$[a^{-1}W_b^-]$}] {};
\node at (a.corner 3) [label=left:{$\hv$}] {};
\node at (a.corner 4) [label=left:{$\hw$}] {};
\node at (a.corner 5) [label=below:{$[a^{-4}W_b^-]$}] {};
\node at (a.corner 6) [label=below:{$[a^{-3}Z_c^+]$}] {};
\node at (a.corner 7) [label=right:{$[a^{-3}Z_b^+]$}] {};
\node at (a.corner 8) [label=right:{$[W_c^-]$}] {};
\foreach \x in {1,2,...,8} {
	\fill (a.corner \x) circle[radius=2pt];
}
\end{tikzpicture}
\caption{The Whitehead graph $\Wh_T(\calL_g,X)$ in Example~\ref{ex:whitehead graph subtree}.}\label{fig:whitehead graph subtree}
\end{figure}
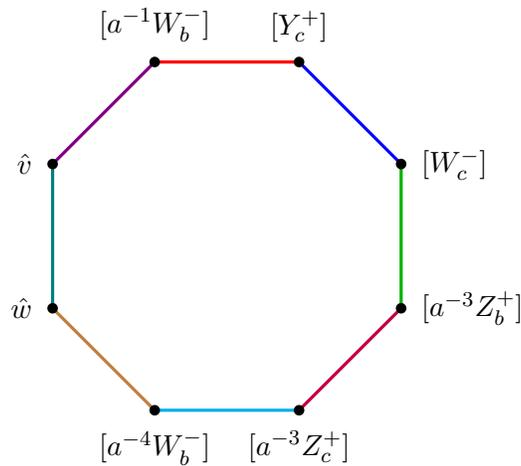
\end{example}

The following two lemmas will be used to help analyze the structure of $\Wh_T(\calL,X)$.  An \emph{embedded line} in a graph $\Gamma$ is an injective function $\alpha \from \RR \to \Gamma$ which is simplicial using the standard simplicial structure on $\RR$ where the vertex set is equal to $\ZZ$.  

\begin{lemma}\label{lem:infinite line}
Let $\calL$ be a periodic line collection.  Suppose that $T \in \calO$ is a Grushko tree 
that is Whitehead reduced for $\calL$.  For any vertex $v \in V_\infty(T)$ and vertex $Y$ of $T_v(\calL)$, there is an embedded line $\alpha_Y \from \RR \to T_v(\calL)$ such that $\alpha_Y(0) = Y$.
\end{lemma}

\begin{proof}
Fix a vertex $Y \in T_v(\calL)$, let $\tU$ be the component of $T_v(\calL)$ that contains $Y$, and let $G_{\tU}$ be the subgroup of $\Stab_T(v)$ that leaves $\tU$ invariant.  As above, we identify $\tU/G_{\tU}$ with a component $U$ in $\Wh_T(\calL,v)$ where $G_\tU$ is in the conjugacy class of $\Mon(U)$.  In particular, since $\Wh_T(\calL,v)$ does not have a type i admissible cut, we have that $G_\tU$ is nontrivial.  As $\Stab_T(v)$ is torsion-free, every nontrivial element in $G_\tU$ has infinite order.  

First, suppose that $[Y]$ is a cut vertex of $U$.  Write $U$ as the union $U_1 \cup U_2$ where $U_1$ and $U_2$ are connected, and $[Y] = U_1 \cap U_2$.  As $\Wh_T(\calL,v)$ does not have a type ii admissible cut both of the subgraphs $U_1$ and $U_2$ have non-trivial monodromy.  Thus, there are embedded cycles $\gamma_1 \subseteq U_1$ and $\gamma_2 \subseteq U_2$, where $\Mon(\gamma_1)$ and $\Mon(\gamma_2)$ are non-trivial, i.e., the products of the labels along both $\gamma_1$ and $\gamma_2$ are non-trivial.  Let $\beta$ be the embedded path (possibly trivial) in $U$ that connects $\gamma_1$ to $\gamma_2$.  Necessarily the path $\beta$ contains $[Y]$.  The subgraph $\gamma_1 \cup \beta \cup \gamma_2$ lifts to an infinite tree in $\tU$ that contains $Y$ but no valence one vertices.  See the graph on the left in Figure~\ref{fig:infinite line}.  Hence there is an embedded line that contains $Y$.   
    
Next suppose that $[Y]$ is not a cut vertex of $U$.  Let $U_0$ be the biconnected component of $U$ that contains $[Y]$.   As $\Wh_T(\calL,v)$ does not have a type ii admissible cut, the subgraph $U_0$ has non-trivial monodromy.  

\medskip \noindent {\it Claim.} The vertex $[Y]$ lies on an embedded cycle $\gamma \subseteq U_0$ where $\Mon(\gamma)$ is non-trivial.

\begin{proof}[Proof of Claim]
Fix some embedded cycle $\gamma_0 \subseteq U_0$ where $\Mon(\gamma_0)$ is non-trivial.  Assuming that $[Y]$ does not lie on $\gamma_0$, there is an embedded path $\beta$ that contains $[Y]$ and only intersects $\gamma_0$ in its endpoints.  This follows as $U_0$ is biconnected.  The endpoints of $\beta$ decompose $\gamma_0$ into two subpaths $\gamma_0'$ and $\gamma_0''$.  Then $[Y]$ lies on the two cycles $\beta \cdot \gamma_0'$ and $\beta \cdot \gamma_0''$.  As $\gamma_0$ has non-trivial monodromy, at least one of these cycles has non-trivial monodromy.  See the graph on the right in Figure~\ref{fig:infinite line}.   
\end{proof}
 
The cycle $\gamma$ lifts to an embedded line in $\tU$ that contains $[Y]$.
\end{proof}

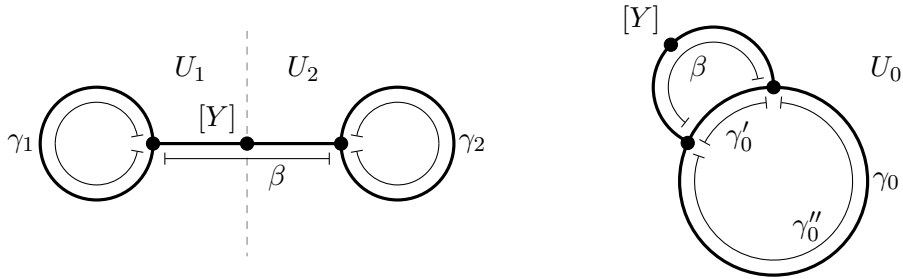
\begin{figure}[h]
\begin{tikzpicture}
\begin{scope}
\draw[gray,dashed] (0,1.5) -- (0,-1.5);
\node at (-0.75,1) {$U_1$}; 
\node at (0.75,1) {$U_2$};
\draw[very thick] (-1.25,0) -- (1.25,0);
\filldraw (-1.25,0) circle [radius=2.5pt];
\filldraw (0,0) circle [radius=2.5pt] node[above left] {$[Y]$} node[label={[xshift=0.4cm,yshift=-0.9cm]$\beta$}] {};
\filldraw (1.25,0) circle [radius=2.5pt];
\draw[very thick] (-2,0) circle [radius=0.75cm];
\draw[very thick] (2,0) circle [radius=0.75cm];
\node at (-3,0) {$\gamma_1$};
\node at (3,0) {$\gamma_2$};
\draw[-|] (-2.55,0) arc (180:10:0.55cm);
\draw[-|] (-2.55,0) arc (180:350:0.55cm);
\draw[-|] (2.55,0) arc (0:170:0.55cm);
\draw[-|] (2.55,0) arc (360:190:0.55cm);
\draw[|-|] (-1.1,-0.2) -- (1.1,-0.2);
\end{scope}
\begin{scope}[xshift=7cm,yshift=-0.5cm]
\draw[very thick] (-0.8,1.25) circle [radius=0.8cm];
\draw[very thick,fill=white] (0,0) circle [radius=1.25cm];
\filldraw (0,1.25) circle [radius=2.5pt];
\filldraw (156 : 1.25) circle [radius=2.5pt];
\filldraw (-1.365685425,1.815685425) circle [radius=2.5pt] node[above left] {$[Y]$};
\node at (1.5,0) {$\gamma_0$};
\node at (127 : 0.75) {$\gamma'_0$};
\node at (307 : 0.75) {$\gamma''_0$};
\node at (-1.0,1.5) {$\beta$};
\draw[|-|] (85 : 1.05) arc (85:-199:1.05cm);
\draw[|-|] (95 : 1.05) arc (95:151:1.05cm);
\draw[|-|] (-0.209115348,1.354188907) arc (10:232:0.6cm);
\node at (1.5,1.5) {$U_0$};
\end{scope}
\end{tikzpicture}
\caption{The two cases of Lemma~\ref{lem:infinite line}.}\label{fig:infinite line}
\end{figure}

Let $v$ be a vertex in $T$.  A \emph{finite star about $v$} is the minimal connected subset $Z \subset T$ that contains a finite set of points that are midpoints of edges incident on $v$.

\begin{lemma}\label{lem:whitehead stars}
Let $\calL$ be a periodic line collection.  Suppose that $T \in \calO$ is a Grushko tree that is Whitehead reduced for $\calL$ and $v \in V_\infty(T)$.  Let $Z \subset T$ be a finite star about $v$.  Then $\Wh_T(\calL,Z)$ is connected and no vertex is a cut vertex other than possibly $\hv$.  
\end{lemma}

\begin{proof}
Enumerate the trivial directions at $Z$ by $Y_1,\ldots,Y_m$.  Each of these directions is contained in a unique direction at $v$.  Abusing notation, we will use $Y_1,\ldots,Y_m$ to denote these corresponding directions at $v$.  These are exactly the vertices deleted when forming $T_v(\calL) - Z$.  Notice that the vertex set of $\Wh_T(\calL,Z)$ is the union of $\{[Y_1],\ldots,[Y_m]\}$ with $V_Z(\calL,v)$.  

Fix a vertex $Y \in T_v(\calL)$ and let $\alpha_Y \from \RR \to T_v(\calL)$ be an embedded line such that $\alpha_Y(0) = Y$.  Such a line exists by Lemma~\ref{lem:infinite line}.  There is some $M \geq 0$ such that $\alpha_Y(j) \notin \{Y_1,\ldots,Y_m\}$ for all $j \geq M$.  In particular, the vertex $[\alpha_Y(M)]$ is $\hv$.  As $\alpha_Y(j-1)$ and $\alpha_Y(j)$ are adjacent in $T_v(\calL)$ for all $j \in \ZZ$, we have that the corresponding vertices $[\alpha_Y(j-1)]$ and $[\alpha_Y(j)]$ in $\Wh_T(\calL,Z)$ are also adjacent if they are distinct.  This shows that there is an edge path from $[Y]$ to $\hv$.  Hence $\Wh_T(\calL,Z)$ is connected.

Suppose that $Y$ is a vertex of $T_v(\calL)$ and $[Y] \neq \hv$ in $\Wh_T(\calL,Z)$.  Thus $Y$ is the only direction at $Z$ that corresponds to $[Y]$.  If $[Y]$ is a cut vertex of $\Wh_T(\calL,Z)$, then we can write $\Wh_T(\calL,Z)$ as the union of two graphs $U$ and $V$ where $U \cap V = [Y]$.  Without loss of generality, we can assume that $\hv \notin U$ and furthermore that $U$ is connected.  Take a vertex $[Y'] \in U$ other than $Y$ and let $\alpha_{Y'} \from \RR \to T_v(\calL)$ be an embedded line such that $\alpha_{Y'}(0) = Y'$.  Again, such a line exists by Lemma~\ref{lem:infinite line}.  As above there is an $M \geq 0$ such that $\alpha_{Y'}(j) \notin \{Y_1,\ldots,Y_m\}$ for $j \leq -M$ or $j \geq M$.  This gives two paths in $\Wh_T(\calL,Z)$ from $[Y']$ to $\hv$ but only one of them can contain $[Y]$ as $\alpha_{Y'}$ is an embedded line.  This is a contradiction, hence $[Y]$ is not a cut vertex.     
\end{proof}

\subsection{Relating the decomposition space to Whitehead graphs}\label{subsec:model}

The definition of the Whitehead graph $\Wh_T(\calL,X)$ is justified by the next lemma.

\begin{lemma}\label{lem:hull}
Let $\calL$ be a periodic line collection and let $P \subseteq \calD(\calL)$ be a finite subset where $q^{-1}(P) \cap V_\infty(G,\calA)$ is empty and where $q^{-1}(P)$ contains more than one point.  Suppose that $T \in \calO$ is a Grushko tree that is Whitehead reduced for $\calL$ and let $X \subset T$ be the convex hull of $q^{-1}(P)$.  Then there is a bijection between the connected components of $\Wh_T(\calL,X)$ and the connected components of $\calD(\calL) - P$.
\end{lemma}

\begin{proof}
Suppose $Y$ is a direction at $X$ and let $\bd_0 Y = v$.  If $\Stab_T(v) = \{1\}$ or if $Y$ belongs to a finite component of $T_v(\calL) - X$, then $[Y] \neq \hv$ and moreover, the subset $q(\bd Y) \subset \calD(\calL) - P$ is connected by Lemma~\ref{lem:decomposition connected subset}.  
  
Next, let $v \in X$ be a vertex with $\Stab_T(v) \neq \{1\}$ and enumerate the directions at $v$ that belong to an infinite component of $T_v(\calL) - X$ by $\{Y_i\}$.  As above, we have that each subset $q(\bd Y_i) \subset \calD(\calL) - P$ is connected by Lemma~\ref{lem:decomposition connected subset}.   

\medskip
\noindent {\it Claim.} \ The set $\{q(v)\} \cup \bigcup_i q(\bd Y_i)$ is contained in a connected component of $\calD(\calL) - P$.  

\begin{proof}[Proof of Claim]
Observe that if, as vertices of $T_v(\calL) - X$, the directions $Y_i$ and $Y_j$ are adjacent, then $q(\bd Y_i) \cap q(\bd Y_j) \neq \emptyset$.  Indeed, as these directions are adjacent in $T_v(\calL) - X$ there is a line $\ell \in \calL$ such that $\ell_T$ meets both $Y_i$ and $Y_j$ and whose endpoints do not lie in $P$.  The endpoints of this line lie in $\bd Y_i$ and $\bd Y_j$ respectively and are mapped by $q$ to the same point in $\calD(\calL)$.  Continuing in this manner, it follows that if $Y_i$ and $Y_j$ lie in the same component of $T_v(\calL) - X$ then $q(\bd Y_i)$ and $q(\bd Y_j)$ lie in the same component of $\calD(\calL) - P$.    

Let $\{Y_{i_j}\}$ be an enumeration of the directions at $v$ that belong to a fixed infinite component of $T_v(\calL) - X$.  To complete the proof of the claim we will show that $q(v)$ belongs to the same component of $\calD(\calL) - P$ as $\bigcup_j q(\bd Y_{i_j})$.  To see this, fix points $y_j \in \bd Y_{i_j}$.  As $y_j \to v$ in the observers topology, we have that $q(y_j) \to q(v)$ since $q$ is continuous and therefore $q(v)$ lies in the same component as $\bigcup_j q(\bd Y_{i_j})$.
\end{proof}

Now consider a connected component $U \subseteq \Wh_T(\calL,X)$.  Let $\{Y_i\}$ be the complete collection of directions at $X$ that correspond to vertices in $U$ and let $\{v_j\}$ be the collection of vertices in $X$ where $\Stab_T(v) \neq \{1\}$ and such that $v_j = \bd_0 Y_i$ for some $Y_i$.  Then the above observations imply that the set:
\begin{equation}\label{eq:union}
\bigcup_j q(v_j) \cup \bigcup_iq(\bd Y_i)
\end{equation}
is connected and open in $\calD(\calL) - P$.  As the same is true for any connected component of $\Wh_T(\calL,X)$ and as $\bd(G,\calA) - q^{-1}(P)$ is the union of $q(\bd Y)$ over all directions $Y \subset T - X$ together with the union of $q(v)$ over all vertices $v \in V_\infty(T)$, we see that the set in~\eqref{eq:union} is also closed and therefore a connected component of $\calD(\calL) - P$.  
\end{proof}

\subsection{Splicing Whitehead graphs}\label{subsec:splicing}

There is a splicing procedure for building $\Wh_T(\calL,X)$ as described by Manning~\cite{ar:Manning10} and used by Cashen--Macura~\cite{ar:CM11} in the setting where $\calA = \emptyset$.  This procedure allows one to study $\Wh_T(\calL,X)$ by considering the Whitehead graphs relative to the vertices in $X$.  We will describe a modification of splicing for the free-product setting.  In this setting, we consider a locally finite subtree $X \subset T$ as being constructed by gluing together finite stars.  

To this end let $X \subset T$ be a finite subtree and let $Z \subset X$ be a finite star about some vertex $v$ in $X$ such that $X' = X - Z$ is connected.  We will describe how to construct $\Wh_T(\calL,X)$ using $\Wh_T(\calL,X')$ and $\Wh_T(\calL,Z)$.  Let $Y \subset T - X'$ be the direction that contains $Z$ and let $Y' \subset T - Z$ be the direction that contains $X'$.  We observe that $\bd_0Y = \bd_0Y'$ is the midpoint of an edge, which we denote by $e$.  As the stabilizer of this point is trivial, we see that $Y$ is the only direction at $X'$ that represents the vertex $[Y]$ in $\Wh_T(\calL,X')$.  Likewise $Y'$ is the only direction at $Z$ that represents $[Y']$ in $\Wh_T(\calL,Z)$.  Let $\calL_e$ be the collection of lines in $\calL$ such that $\ell_T$ contains $e$.  Each line in $\calL_e$ corresponds to a unique edge incident to $[Y]$ in $\Wh_T(\calL,X')$ and also to a unique edge incident to $[Y']$ in $\Wh_T(\calL,Z)$.  This defines a bijection between the set of edges incident to $[Y]$ and the set of edges incident to $[Y']$.  Now we remove the vertex $[Y]$ from $\Wh_T(\calL,X')$ and the vertex $[Y']$ from $\Wh_T(\calL,Z)$, maintaining the edges with ``loose ends.''  We will denote these graphs with loose ends by $\Wh_T(\calL,X') - \{[Y]\}$ and $\Wh_T(\calL,Z) - \{[Y']\}$ respectively.  We now \emph{splice} together these loose ends using the bijection coming from $\calL_e$.  The resulting graph is $\Wh_T(\calL,X)$.  

\begin{example}\label{ex:splicing}
We consider the free product $(G,\calA)$, Grushko tree $T \in \calO$ and element $g \in G$ as in Examples~\ref{ex:whitehead graph} and \ref{ex:whitehead graph subtree}.  We will make use of the notation from those examples.  Let $p$ be the midpoint of the edge $e_1$.  We can decompose $X$ along this point and write $X$ as the union of two finite stars $Z$ and $Z'$ where $v \in Z$, $w \in Z'$ and $Z \cap Z'  = \{p\}$.  The Whitehead graphs $\Wh_T(\calL_g,Z)$ and $\Wh_T(\calL_g,Z')$ are shown in Figure~\ref{fig:whitehead graph stars}.  We obtain $\Wh_T(\calL_g,X)$ by removing the vertex $[Y_b^+]$ from $\Wh_T(\calL_g,Z)$ and vertex $[Z_b^-]$ and splicing the like colored loose ends together.

\begin{figure}[h]
\begin{tikzpicture}
\begin{scope}
\node[draw=none,minimum size=3cm,regular polygon,regular polygon sides=4,rotate=-45] at (0,0) (a) {};
\draw[very thick,blue] (a.corner 1) -- (a.corner 2);
\draw[very thick,red] (a.corner 2) -- (a.corner 3);
\draw[very thick,violet] (a.corner 3) -- (a.corner 4);
\draw[very thick,teal] (a.corner 4) -- (a.corner 1);
\node at (a.corner 1) [label=right:{$[Y_b^+]$}] {};
\node at (a.corner 2) [label=above:{$[Y_c^+]$}] {};
\node at (a.corner 3) [label=left:{$[a^{-1}W_b^-]$}] {};
\node at (a.corner 4) [label=below:{$\hv$}] {};
\foreach \x in {1,2,3,4} {
	\fill (a.corner \x) circle[radius=2pt];
}
\end{scope}
\begin{scope}[xshift=7cm]
\node[draw=none,minimum size=3cm,regular polygon,regular polygon sides=6,rotate=0] at (0,0) (a) {};
\draw[very thick,black!30!green] (a.corner 1) -- (a.corner 2);
\draw[very thick,blue] (a.corner 2) -- (a.corner 3);
\draw[very thick,teal] (a.corner 3) -- (a.corner 4);
\draw[very thick,brown] (a.corner 4) -- (a.corner 5);
\draw[very thick,cyan] (a.corner 5) -- (a.corner 6);
\draw[very thick,purple] (a.corner 6) -- (a.corner 1);
\node at (a.corner 1) [label=above:{$[a^{-3}Z_b^+]$}] {};
\node at (a.corner 2) [label=above:{$[W_c^-]$}] {};
\node at (a.corner 3) [label=left:{$[Z_b^-]$}] {};
\node at (a.corner 4) [label=below:{$\hw$}] {};
\node at (a.corner 5) [label=below:{$[a^{-4}W_b^-]$}] {};
\node at (a.corner 6) [label=right:{$[a^{-3}Z_c^+]$}] {};
\foreach \x in {1,2,...,6} {
	\fill (a.corner \x) circle[radius=2pt];
}
\end{scope}
\end{tikzpicture}
\caption{The Whitehead graphs $\Wh_T(\calL_g,Z)$ and $\Wh_T(\calL_g,Z')$ in Example~\ref{ex:splicing}.}\label{fig:whitehead graph stars}
\end{figure}
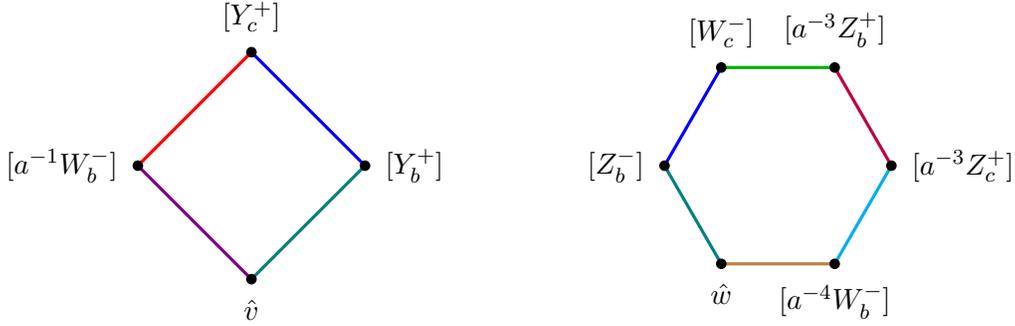

\end{example}

\subsection{Identifying cut sets in $\calD(\calL)$ via Whitehead graphs}\label{subsec:identify cut sets}

Suppose that $T \in \calO$ is Whitehead reduced for $\calL$ and $v$ is a vertex of $T$.  Using the Whitehead graph $\Wh_T(\calL,v)$, we can identify two types of cut sets in $\calD(\calL)$.  These will be used in the next two sections to break the decomposition space up along cut points and to characterize quadratic elements in $G$.

\begin{definition}\label{def:peripheral cut}
If either:
\begin{enumerate}[label=(\roman*)]
\item $\Wh_T(\calL,v)$ is disconnected, or 
\item $\Wh_T(\calL,v)$ is connected and $\Mon(\Wh_T(\calL,v))$ is not equal to $\Stab_T(v)$,
\end{enumerate} 
then $q(v) \in \calD(\calL)$ is called a \emph{peripheral cut point}.
\end{definition}

\begin{definition}\label{def:edge cut}
Fix an edge $e$ that is incident to $v$ and let $\calL_e \subset \calL$ be the collection of lines in $\calL$ such that $\ell_T$ contains $e$ for each $\ell \in \calL_e$.  Then the set $q(\calL_e) \subset \calD(\calL)$ is called an \emph{edge cut set}.  
\end{definition}

\begin{lemma}\label{lem:parabolic cut}
Suppose that $\calL$ is a periodic line collection and that $\calD(\calL)$ is connected.  A peripheral cut point is a cut point.  Moreover, a point $x \in \calD(\calL)$ is a cut point only if either $x$ is a peripheral cut point or $x = q(\ell)$ for some $\ell \in \calL$.   
\end{lemma}

\begin{proof}
First we observe that peripheral cut points are indeed cut points of $\calD(\calL)$.  To this end, let $v \in T$ be a vertex such that $q(v)$ is a peripheral cut point.  Notice that either of the conditions in Definition~\ref{def:peripheral cut} imply that $T_v(\calL)$ is not connected.  Let $\tU \subset T_v(\calL)$ be a connected component.  Then the set:
\begin{equation*}
C_\tU = \bigcup_{Y \in V(\tU)} q(\bd(Y)) \subset \calD(\calL)
\end{equation*}
is an open set.  Indeed, the union of the sets $\bd(Y) \subset \bd(G,\calA)$ over the directions at $v$ determined by $\tU$ is an open saturated set.  Further it is connected by Lemma~\ref{lem:decomposition connected subset} and since for $Y,Y' \in V(\tU)$ that are adjacent in $\tU$, we have that $q(\bd Y) \cap q(\bd Y')$ is non-empty.  As $\calD(\calL) - \{q(v)\}$ is the union of $C_{\tU'}$ over all components $\tU' \subset T_v(\calL)$ and $C_\tU \cap C_{\tU'} = \emptyset$ for $\tU \neq \tU'$, we see that $C_\tU$ is closed as well.  This shows that $q(v)$ is indeed a cut point.  

Next, suppose that $x \in \calD(\calL)$ is a cut point.  By Lemma~\ref{lem:boundary}, the decomposition space $\calD(\calL)$ is the boundary of the relatively hyperbolic group $(G,\calA \cup \calN_\calL)$.  Dasgupta--Hruska proved that every cut point of the boundary of a relatively hyperbolic group is fixed by a subgroup conjugate into $\calA \cup \calN_\calL$~\cite[Theorem~1.1]{ar:DH}.  Hence either $x = q(v)$ for some $v \in V_\infty(G,\calA)$ or $x = q(\ell)$ for some $\ell \in \calL$.  We claim in the former case that $q(v)$ is a peripheral cut point.  Indeed, by Proposition~\ref{prop:decomposition connected}, there exists a Grushko tree $T \in \calO$ that is Whitehead reduced for $\calL$.  Suppose that $q(v)$ is not a peripheral cut point.  Then $\Wh_T(\calL,v)$ is connected and
$\Mon(\Wh_T(\calL,v)) = \Stab_T(v)$.   Hence $T_v(\calL)$ is connected and the above argument shows that $\calD(\calL) - \{q(v)\}$ is connected as well.  Indeed, in the notation above we have $C_\tU = \calD(\calL) - \{q(v)\}$.  This contradicts the assumption that $q(v)$ is a cut point.  Hence we have that $q(v)$ is a peripheral cut point.
\end{proof}

We need the following notation.  Suppose that $X \subset T$ is a locally finite tree and $\calL'$ is a subset of lines in $\calL$.  Then $\Wh_T(\calL,X) - \calL'$ is the subgraph of $\Wh_T(\calL,X)$ obtained by removing every edge that corresponds to a line in $\calL'$. 

\begin{lemma}\label{lem:edge cut set}
Let $\calL$ be a periodic line collection.  If $\calD(\calL)$ is connected and every cut point is a peripheral cut point, then an edge cut set is a cut set.
\end{lemma}

\begin{proof}
The statement and proof are similar to statements by Cashen--Macura~\cite[Lemma~4.20 \& Proposition~4.21]{ar:CM11}.

Fix an edge $e \subset T$ and let $\calL_e \subset \calL$ be the collection of lines in $\calL$ as given by Definition~\ref{def:edge cut}.  Let $Y \subset T - \{v\}$ be the direction at $v$ that meets $e$. 
  
Let
\begin{equation*}
X = \bigcup_{\ell \in \calL_e} \ell_T \subset T.
\end{equation*}  
By Lemma~\ref{lem:hull} we have that there is a bijection between the components of $\calD(\calL) - q(\calL_e)$ and $\Wh_T(\calL,X)$.

Let $Z_0 \subset X$ be the minimal closed subset such that $X$ equals the union of $Z_0$ together with the images of the $2\abs{\calL_e}$ geodesic rays $\rho_\ell^\pm \from [0,\infty] \to \hT$ with disjoint interiors where $\rho_\ell^\pm(\infty) = \ell_T^\pm$.  This set $Z_0$ is what Cashen--Macura call the \emph{core}~\cite[Section~4.6]{ar:CM11}.  We need to modify it slightly for our purposes.  We set
\begin{equation*}
Z = Z_0 \cup \bigcup_{\ell \in \calL_e} \rho_\ell^+([0,1/2]) \cup \bigcup_{\ell \in \calL_e} \rho_\ell^-([0,1/2]).
\end{equation*}    
In other words, we add on the initial half-edge of each of the geodesic rays $\rho_\ell^\pm$ to the set $Z_0$.  See Figure~\ref{fig:edge cut}.  Note that $Z$ is the finite union of finite stars.  We remark that each line $\ell \in \calL_e$ determines an edge in $\Wh_T(\calL,Z)$ that is incident to two trivial vertices.

\begin{figure}[h]
\begin{tikzpicture}
\def\ee{1.8}
\def\sqr{0.707106781}
\node at (0,0) [above=0.25] {$e$};
\draw[black!30!green,line width=4mm,line cap=round] (-1.5*\ee-1.5*\sqr*\ee,1.5*\sqr*\ee) -- (-1.5*\ee,0) -- (3*\ee,0);
\draw[black!30!green,line width=4mm,line cap=round] (-2*\ee-\sqr*\ee,\sqr*\ee) -- (-1.5*\ee-\sqr*\ee,\sqr*\ee); 
\draw[black!30!green,line width=4mm,line cap=round] (-2*\ee,0) -- (-1.5*\ee,0); 
\draw[black!30!green,line width=4mm,line cap=round] (1.5*\ee,0) -- (1.5*\ee+0.5*\sqr*\ee,0.5*\sqr*\ee); 
\draw[black!30!green,line width=4mm,line cap=round] (2.5*\ee,0) -- (2.5*\ee+0.5*\sqr*\ee,0.5*\sqr*\ee); 
\draw[yellow,line width=2.5mm,line cap=round] (-1.5*\ee-\sqr*\ee,\sqr*\ee) -- (-1.5*\ee,0) -- (2.5*\ee,0);
\draw[very thick] (-4*\ee,0) -- (4*\ee,0);
\draw[very thick] (-3*\ee-\sqr*\ee,\sqr*\ee) -- (-1.5*\ee-\sqr*\ee,\sqr*\ee) -- (-1.5*\ee,0) (-1.5*\ee-\sqr*\ee,\sqr*\ee) -- (-1.5*\ee-2.5*\sqr*\ee,2.5*\sqr*\ee);
\draw[very thick] (1.5*\ee,0) -- (1.5*\ee+2.5*\sqr*\ee,2.5*\sqr*\ee);
\draw[very thick] (2.5*\ee,0) -- (2.5*\ee+1.5*\sqr*\ee,1.5*\sqr*\ee);
\draw[thick,blue] (-4*\ee,-0.4) -- (4*\ee,-0.4);
\draw[thick,blue] (-1.5*\ee-2.5*\sqr*\ee+\sqr*0.3,2.5*\sqr*\ee+\sqr*0.3) -- (-1.5*\ee+0.1,0.3) -- (1.5*\ee-0.1,0.3) -- (1.5*\ee+2.5*\sqr*\ee-0.3*\sqr,2.5*\sqr*\ee+0.3*\sqr);
\draw[thick,blue] (-3*\ee-\sqr*\ee,\sqr*\ee-0.3) -- (-1.5*\ee-\sqr*\ee-0.1,\sqr*\ee-0.3) -- (-1.5*\ee-0.1,-0.3) -- (2.5*\ee+0.1,-0.3) -- (2.5*\ee+1.5*\sqr*\ee+0.3,1.5*\sqr*\ee-0.3);
\filldraw (-0.5*\ee,0) circle [radius=2.5pt];
\filldraw (0.5*\ee,0) circle [radius=2.5pt];
\filldraw (1.5*\ee,0) circle [radius=2.5pt];
\filldraw (-1.5*\ee,0) circle [radius=2.5pt];
\filldraw (2.5*\ee,0) circle [radius=2.5pt];
\filldraw (-2.5*\ee,0) circle [radius=2.5pt];
\filldraw (-3.5*\ee,0) circle [radius=2.5pt];
\filldraw (3.5*\ee,0) circle [radius=2.5pt];
\filldraw (-1.5*\ee-2*\sqr*\ee,2*\sqr*\ee) circle [radius=2.5pt];
\filldraw (-2.5*\ee-\sqr*\ee,\sqr*\ee) circle [radius=2.5pt];
\filldraw (-1.5*\ee-\sqr*\ee,\sqr*\ee) circle [radius=2.5pt];
\filldraw (1.5*\ee+\sqr*\ee,\sqr*\ee) circle [radius=2.5pt];
\filldraw (1.5*\ee+2*\sqr*\ee,2*\sqr*\ee) circle [radius=2.5pt];
\filldraw (2.5*\ee+\sqr*\ee,\sqr*\ee) circle [radius=2.5pt];
\draw[thick,fill=white] (-2*\ee,0) circle [radius=2.5pt];
\draw[thick,fill=white] (-2*\ee-\sqr*\ee,\sqr*\ee) circle [radius=2.5pt];
\draw[thick,fill=white] (-1.5*\ee-1.5*\sqr*\ee,1.5*\sqr*\ee) circle [radius=2.5pt];
\draw[thick,fill=white] (3*\ee,0) circle [radius=2.5pt];
\draw[thick,fill=white] (2.5*\ee+0.5*\sqr*\ee,0.5*\sqr*\ee) circle [radius=2.5pt];
\draw[thick,fill=white] (1.5*\ee+0.5*\sqr*\ee,0.5*\sqr*\ee) circle [radius=2.5pt];
\end{tikzpicture}
\caption{The lines in $\calL_e$ (blue) together with a schematic of the subsets $Z_0$ (yellow) and $Z$ (green) from Lemma~\ref{lem:edge cut set}.  Vertices of $T$ are shown in black and midpoints of edge are shown in white.}\label{fig:edge cut}
\end{figure}
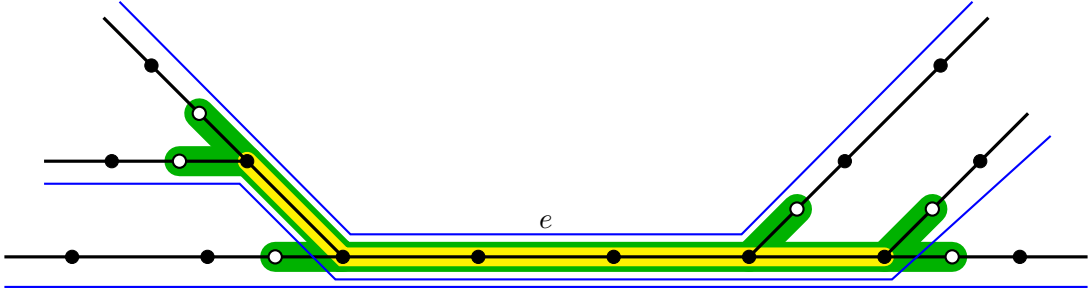

\medskip
\noindent {\it Claim 1.} \ There is a bijection between the connected components of $\Wh_T(\calL,X)$ and $\Wh_T(\calL,Z) - \calL_e$.  

\begin{proof}[Proof of Claim 1]
For each geodesic ray $\rho_\ell^\pm$, we let $R_\ell^\pm = \rho_\ell^\pm([1/2,\infty)) \subset T$ and we let $Y_\ell^\pm$ be the direction at $\rho_\ell^\pm(1/2)$ that meets $Z$.  Observe that $\Wh_T(\calL,R_\ell^+) - \{[Y_\ell^+]\}$ is connected.  Indeed, as $q(\ell)$ is not a cut point, by Lemma~\ref{lem:hull}, we have that $\Wh_T(\calL,\ell_T)$ is connected.  Now $\Wh_T(\calL,R_\ell^+,\calL) - \{[Y_\ell^+]\}$ is an infinite end of $\Wh_T(\calL,\ell_T)$ and as the stabilizer of $\ell_T$ acts cocompactly on $\Wh_T(\calL,\ell_T)$ we see that $\Wh_T(\calL,R_\ell^+) - \{[Y_\ell^+]\}$ must be connected as well.  The same holds for $\Wh_T(\calL,R_\ell^-) - \{[Y_\ell^-]\}$.

Now let $\oY_\ell^\pm$ be the direction at $\rho_\ell^\pm(1/2)$ that does not meet $Z$.  Then $\Wh_T(\calL,X)$ is obtained from $\Wh_T(\calL,Z) - \calL_e$ by removing the vertices $[\oY_\ell^\pm]$ and splicing this graph with loose ends together with the graphs with loose ends $\Wh_T(\calL,R_\ell^\pm) - \{[Y_\ell^\pm]\}$.  As the latter graphs are connected, there is a clear bijection between the connected components of $\Wh_T(\calL,X)$ and $\Wh_T(\calL,Z) - \calL_e$.  
\end{proof}

Next, we will show how we can trim the set $Z$ down to the set which is essentially the edge $e$.  Specifically, for each of the lines $\ell \in \calL_e$, let $e_\ell^\pm$ be the two edges contained in $\ell$ that share a single vertex with $e$ and so that $e_\ell^\pm$ separates $e$ from $\ell_T^\pm$.  We note the possibility that some of the edges $e_\ell^\pm$ might be the same.  Subdivide each $e_\ell^\pm$ into $e_{\ell,0}^\pm e_{\ell,1}^\pm$ where $e_{\ell,0}^\pm$ shares a vertex with $e$.  Finally, we set:
\begin{equation*}
\he = e \cup \bigcup_{\ell \in \calL_e} e_{\ell,0}^+ \cup \bigcup_{\ell \in \calL_e} e_{\ell,0}^-.
\end{equation*}  
As for $Z \subset T$, each line $\ell \in \calL_e$ determines an edge in $\Wh_T(\calL,\he)$ that is incident to two trivial vertices.  

The following claim, and its proof, are similar to the discussion of the \emph{pruned core}~\cite[Section~4.6]{ar:CM11}.

\medskip
\noindent {\it Claim 2.} \ There is a bijection between the connected components of $\Wh_T(\calL,Z) - \calL_e$ and $\Wh_T(\calL,\he) - \calL_e$.

\begin{proof}[Proof of Claim 2]
Let $Z' \subset Z$ be a finite star about some vertex $v \in Z$ such that $Z - Z'$ is connected and contains $e$.  Set $\hZ = Z - Z'$.  We will show that there is a bijection between the connected components of $\Wh_T(\calL,Z) - \calL_e$ and $\Wh_T(\calL,\hZ) - \calL_e$.  The claim then follows by repeating this process until we have trimmed the set $Z$ down to $\he$.

We consider the Whitehead graph $\Wh_T(\calL,Z')$.  Let $Y$ be the direction at $Z'$ that contains $e$ and let $Y'$ be the direction at $\hZ$ that contains $Z'$.  We observe that $\Wh_T(\calL,Z) - \calL_e$ is obtained from $\Wh_T(\calL,\hZ) - \calL_e$ by removing the vertex $[Y']$ and splicing this to the graph with loose ends obtained from $\Wh_T(Z',\calL) - \calL_e$ by removing $[Y]$.  

The Whitehead graph $\Wh_T(\calL,Z')$ is connected and $[Y]$ is not a cut vertex of this graph.  Indeed, if $\Stab_T(v) = \{1\}$, then this follows as $\Wh_T(\calL,Z') $ is isomorphic to $\Wh_T(\calL,v)$ and so $\Wh_T(\calL,Z')$ does not have an admissible cut. Else, if $\#\abs{\Stab_T(v)} = \infty$ this follows by Lemma~\ref{lem:whitehead stars}.

As each of the edges in $\Wh_T(\calL,Z')$ corresponding to one of the lines in $\calL_e$ is incident to $[Y]$, we have that the graph with loose ends $\Wh_T(\calL,Z') - \calL_e$ obtained by removing $[Y]$ is also connected.  Hence as in Claim 1, we are removing a vertex and splicing on a connected graph with loose ends and there is a clear bijection between the connected components.
\end{proof}

The lemma now follows as $\Wh_T(\calL,\he) - \calL_e$ contains exactly two components (this follows from another application of Lemma~\ref{lem:whitehead stars}) and adding any of the edges corresponding to one of the lines $\ell \in \calL_e$ connects these two components.
\end{proof}


\section{Decomposing $\calD(\calL_g)$ along cut points}\label{sec:cut points}

Now that we have a model for the decomposition space, we want to work with it to identify $\calZ$--splittings of our free product $(G,\calA)$ where a given $\calZ$--simple element is elliptic.  To this end, it will become necessary to be able to promote the existence of a loxodromic cut pair in the decomposition space to the existence of an inseparable loxodromic cut pair, as the latter actually corresponds to a $\calZ$--splitting of the free product.  We start with a proposition (Proposition~\ref{prop:trichotomy}) that allows for such a promotion when the decomposition space is not homeomorphic to the circle nor contains cut points.  We deal with the case that the decomposition space is homeomorphic to a circle in Section~\ref{sec:quadratic elements}.  The remainder of this section is dealing with the case of cut points in the decomposition space.  We will apply recent work of Dasgupta--Hruska~\cite{ar:DH} that constructs a tree that models the structure of cut points in the boundary of a relatively hyperbolic group.  The key features of this tree appear in Lemma~\ref{lem:B}.  We will see that we can apply Proposition~\ref{prop:trichotomy} to a certain vertex stabilizer in this tree.  The ultimate statement regarding the existence of inseparable loxodromic cut pairs appears in Section~\ref{sec:quadratic elements} as Corollary~\ref{co:inseparable} as we must further analyze the case that the aforementioned vertex stabilizer also produces a decomposition space that is homeomorphic to a circle.

Let $(G,\calA)$ be a fixed non-sporadic torsion-free free product.  

\begin{proposition}\label{prop:trichotomy}
Let $\calL$ be a periodic line collection and suppose that $\calD(\calL)$ is connected and contains a loxodromic cut pair.  Then one of the following holds:  
\begin{enumerate}
\item\label{item:tri-circle} $\calD(\calL)$ is homeomorphic to a circle.

\item\label{item:tri-point} $\calD(\calL)$ contains a cut point.

\item\label{item:tri-pair} $\calD(\calL)$ contains an inseparable loxodromic cut pair.
\end{enumerate}
\end{proposition}

\begin{proof}
Assume that $\calL$ is a periodic line collection such that $\calD(\calL)$ is connected and contains a loxodromic cut pair.  Further suppose that $\calD(\calL)$ is not homeomorphic to the circle and also that $\calD(\calL)$ does not contain a cut point.  

Let $\{q(a^\infty),q(a^{-\infty})\}$ be a loxodromic cut pair in $\calD(\calL)$.  It is well-known that multivalent exact cut pairs are inseparable, see~\cite[Lemma~2.3]{ar:Cashen16} or \cite[Lemma~2.6]{ar:HH} for instance.  Hence, we can assume that $\{q(a^\infty),q(a^{-\infty})\}$ is bivalent.  In the setting of boundaries of relatively hyperbolic groups, Haulmark--Hruska~\cite[Proposition~4.6]{ar:HH} characterize bivalent local cut points when the boundary is connected without cut points and not homeomorphic to the circle.  Their result applies to $\calD(\calL)$ by Lemma~\ref{lem:boundary}.  Their characterization implies that a bivalent exact cut pair is either an inseparable loxodromic cut pair or it is part of a \emph{necklace}.  The definition of a necklace is not necessary, we only require the fact that each necklace contains a \emph{jump}, which itself is an inseparable loxodromic cut pair (cf.~\cite[Lemma~4.19]{ar:Haulmark19}).  
\end{proof}

In the case that $\calA = \emptyset$, the above proposition appears as Corollary~4.10 in work of Cashen~\cite{ar:Cashen16}.  Again, the proof given by Cashen applies mutatis mutandis to the free product setting.  We chose the above proof for brevity.

Let $g \in G$ be a non-peripheral element that is not simple.  By Corollary~\ref{co:not simple} we have that $\calD(\calL_g)$ is connected.  Further, if we asume that $g$ is $\calZ$--simple, then $\calD(\calL_g)$ contains a loxodromic cut pair by Lemma~\ref{lem:cut set} and so Proposition~\ref{prop:trichotomy} applies.  As mentioned above, in the next section we will analyze the case that $\calD(\calL_g)$ is homeomorphic to a circle, or more generally an inverse limit of a tree of circles.  In the remainder of this section we will see how we can decompose the decomposition space along cut points.   

Dasgupta--Hruska have described a tree $\calT_{\rm cut}$ that models the structure of cut points in $\calD(\calL_g)$~\cite[Theorem~1.2]{ar:DH}.  The tree has a natural bipartite structure on $\calT_{\rm cut}$ where each vertex is either:
\begin{enumerate}
\item a nontrivial maximal subcontinuum of $\calD(\calL_g)$ that is not separated by a cut point in $\calD(\calL_g)$, or
\item a cut point in $\calD(\calL_g)$.
\end{enumerate}
The set of vertices of the first type is denoted $V_0(\calT_{\rm cut})$; the set of vertices of the second type is denoted $V_1(\calT_{\rm cut})$.  A vertex $B \in V_0(\calT_{\rm cut})$ is adjacent to a vertex $x \in V_1(\calT_{\rm cut})$ if $x \in B$.  The group $G$ acts on $\calT_{\rm cut}$ in the obvious way.  If $\calD(\calL_g)$ does not have any cut points, then $\calT_{\rm cut}$ consists of a single vertex corresponding to $\calD(\calL_g)$.       
  
\begin{lemma}\label{lem:cut point tree}
Let $g \in G$ be a non-peripheral element that is not simple.  Then the following hold.
\begin{enumerate}
\item\label{item:unique vertex} There is a unique vertex in $\calT_{\rm cut}$ that is fixed by $g$.  Moreover, this vertex belongs to $V_0(\calT_{\rm cut})$.
\item\label{item:peripheral edge stab} Any edge stabilizer in $\calT_{\rm cut}$ is peripheral.
\end{enumerate}
\end{lemma}

\begin{proof}
Let $B \subseteq \calD(\calL_g)$ be a maximal subcontinuum of $\calD(\calL_g)$ that is not separated by a cut point in $\calD(\calL_g)$ and that contains $q(g^\infty)$.    

If $g$ does not fix $B$ as a vertex in $V_0(\calT_{\rm cut})$, then $q(g^\infty)$ is a cut point and hence a vertex in $V_1(\calT_{\rm cut})$.  Let $S$ be the $(G,\calA)$ tree obtained by collapsing all edges of $\calT_{\rm cut}$ other than those in the orbit of the edge connecting $q(g^\infty)$ to $B$.  The result is a $\calZ$--splitting of $(G,\calA)$ where the stabilizer of some edge in $S$ is a subgroup of $\I{g}$.  As we assumed that $g$ is not simple, this is a contradiction.  Indeed, if this edge stabilizer is trivial, then $S$ witnesses $g$ as simple.  Else, if this edge stabilizer is nontrivial, then Lemma~\ref{lem:edge stabs in Z-splittings} implies that some power of $g$, and hence also $g$, is simple.  Therefore $g$ fixes $B$. 

If $g$ fixes a vertex other than $B$, then $g$ fixes a cut point in $\calD(\calL_g)$, as is seen using the bipartite structure of the tree $\calT_{\rm cut}$.  As $g$ is non-peripheral, by Lemma~\ref{lem:parabolic cut}, this can only happen if $q(g^\infty)$ is a cut point.  As above, this is a contradiction.  This shows~\eqref{item:unique vertex}.

By the bipartite structure of $\calT_{\rm cut}$, any edge stabilizer must fix a cut point in $\calD(\calL_g)$.  By Lemma~\ref{lem:parabolic cut}, any cut point is either a peripheral cut point or equal to $q(ag^\infty)$ for some $a \in G$.  Once again, the above shows that the latter possibility cannot occur, hence any cut point is a peripheral cut point and therefore stabilized by a peripheral subgroup.  Hence any edge stabilizer is also a peripheral subgroup.  This shows~\eqref{item:peripheral edge stab}.
\end{proof}

For a vertex $B \in V_0(\calT_{\rm cut})$ we let $G_{B} = \Stab_{\calT_{\rm cut}}(B)$.  By $\calN_g^{B}$ we denote the $G_{B}$--conjugacy class of $N_{g'}$ where $g' \in B$ is conjugate to $g$.  If no such conjugate exists we set $\calN_g^{B} = \emptyset$.  By Lemma~\ref{lem:cut point tree}\eqref{item:unique vertex} this is well-defined.  

In order to state the following lemma, which will allow us to promote the existence of a loxodromic cut pair to the existence of an inseparable loxodromic cut pair even in the setting of cut points with one additional hypothesis, we need the following notation introduced by Guirardel--Levitt~\cite{ar:GL17}.  Again, consider a vertex $B \in V_0(\calT_{\rm cut})$.  By $\Inc_{B}$ we denote the set of $G_{B}$--conjugacy classes of stabilizers of the edges in $\calT_{\rm cut}$ incident to $B$.  By $\calA |_{G_{B}}$ we denote the set of $G_{B}$--conjugacy classes of subgroups $hA_ih^{-1}$ for $h \in G$ and $i = 1,\ldots,k$ that fix $B$ but not any incident edge.  Lastly, we set $\calP_{B} = \Inc_{B} \cup \, \calA |_{G_{B}}$.  By Lemma~\ref{lem:cut point tree}\eqref{item:peripheral edge stab}, we have that any subgroup of $G_B$ whose conjugacy class is in $\calP_B$ is necessarily peripheral. 

\begin{lemma}\label{lem:B}
Let $g \in G$ be a non-peripheral element that is $\calZ$--simple but not simple and let $B \in V_0(\calT_{\rm cut})$ be the unique vertex fixed by $g$.  Then the following hold.
\begin{enumerate}

\item\label{item:B non-peripheral} The pair $(G_B,\calP_B)$ is a torsion-free free product and $g$ is non-peripheral with respect to $\calP_B$.

\item\label{item:B simple} The element $g$ is $\calZ$--simple and not simple with respect to $\calP_B$.

\item\label{item:B single} The vertex set $V_0(\calT_{\rm cut})$ has a single orbit.

\item\label{item:B boundary} For the periodic line collection $\calL_g^B = \{ \{ag^\infty,ag^{-\infty}\} \mid a \in G_B \}$ on the free product $(G_B,\calP_B)$, the decomposition space $\calD_{(G_B,\calP_B)}(\calL_g^B)$ is homeomorphic to the boundary of the relatively hyperbolic group $(G_B,\calP_B \cup \calN_g^B)$. 

\item\label{item:B dichotomy} Either the decomposition space $\calD_{(G_B,\calP_B)}(\calL_g)$ is homeomorphic to a circle, or it contains an inseparable loxodromic cut pair.

\item\label{item:B cut pairs} Any inseparable loxodromic cut pair in $\calD_{(G_B,\calP_B)}(\calL_g)$ is an inseparable loxodromic cut pair in $\calD(\calL_g)$.

\end{enumerate}

\end{lemma}

\begin{proof}
First, Dasgupta--Hruska observe that $G_{B'}$ is finitely generated for any $B' \in V_0(\calT_{\rm cut})$~\cite[Proposition~2.13]{ar:DH}.  This implies that $(G_B,\calP_B)$ is a torsion-free free product.  Indeed, take any Grushko tree $T \in \calO$ and consider $T_B \subseteq T$ the minimal subtree for the action of $G_B$ on $T$.  Suppose that $H$ is a subgroup of $G_B$ whose $G_B$--conjugacy class is in $\calP_{B}$.  Then $H$ is peripheral and hence has a unique fixed point in $T_B$.  Moreover, any non-trivial vertex stabilizer in $T_B$ is of the form $H_B = G_B \cap hA_ih^{-1}$ for some $i = 1,\ldots,k$ and $h \in G$.  If $H_B$ is a proper subgroup of $hA_ih^{-1}$ then we claim that $H_B$ fixes an edge incident to $B$.  Indeed, let $v \in V_\infty(T) \subseteq \bd(G,\calA)$ be the vertex fixed by $hA_ih^{-1}$.  As $H_B$ is a proper subgroup of $hA_ih^{-1}$, then $hA_ih^{-1}$ does not fix $B$.  Thus $q(v)$ is a cut point of $\calD(\calL_g)$, and hence a vertex in $V_1(\calT_{\rm cut})$.  Then clearly the edge incident on $B$ and $q(v)$ is stabilized by $H_B$, completing our claim.  Therefore the conjugacy class of any vertex stabilizer in $T_B$ is in $\calP_B$.  As edge stabilizers in $T_B$ are trivial and $G_B$ is finitely generated, this shows that $(G_B,\calP_B)$ is a torsion-free free product.  

As $g$ is non-peripheral with respect to $(G,\calA)$ and all subgroups represented in $\calP_B$ are peripheral with respect to $\calA$, we see that $g$ is non-peripheral with respect to $\calP_B$ as well.   This shows~\eqref{item:B non-peripheral}.
  
Let $S$ be a $\calZ$--splitting of $(G,\calA)$ where $g$ is elliptic and let $S_B \subseteq S$ be the minimal subtree with respect to the action of $G_B$ on $S$.  As any subgroup of $G_B$ whose conjugacy class is in $\calP_B$ is necessary peripheral with respect to $\calA$, we see that $S_B$ is a $\calZ$--splitting of $(G_B,\calP_B)$ where $g$ is elliptic.  This shows that $g$ is $\calZ$--simple with respect to $\calP_B$.

Suppose that $g$ is simple with respect to $\calP_B$ and let $S$ be a free splitting of $(G_B,\calP_B)$ in which $g$ is elliptic.  Then we can blow-up $\calT_{\rm cut}$ by equivariantly replacing the vertex $B$ by a copy of the tree $S$.  Indeed, the stabilizers of edges incident on $B$ fix a unique vertex in $S$ and so this is well-defined.  The result is a $(G,\calA)$ tree which we denote $S'$.  Fix a single edge $e$ in $S \subseteq S'$.  Then the result of collapsing all edges of $S'$ outside of the orbit of $e$ is a free splitting of $(G,\calA)$ in which $g$ is elliptic.  This would imply that $g$ is simple with respect to $\calA$.  Therefore, as $g$ is not simple with respect to $\calA$, it is also not simple with respect to $\calP_B$.  This shows~\eqref{item:B simple}.     

Next, suppose that there is a vertex $B' \in V_0(\calT_{\rm cut})$ that is not in the orbit of $B$.  As no conjugate of $g$ fixes $B'$, we have that $\calN^{B'}_g = \emptyset$.  As in \eqref{item:B non-peripheral}, we have that $(G_{B'},\calP_{B'})$ is a torsion-free free product.  As in \eqref{item:B simple}, we have that $(G_{B'},\calP_{B'})$ does not have any free splitting. Therefore $(G_{B'},\calP_{B'})$ must be either $(\{1\},\emptyset)$, $(G_{B'},\{[G_{B'}]\})$, or $(\ZZ,\emptyset)$.  As the action of $G$ on $\calT_{\rm cut}$ is minimal, we must have that $\Inc_{B'} = \emptyset$ in all cases.  This implies that any edge incident to $B'$ has trivial stabilizer, which can be used to produce a free splitting of $(G,\calA)$, in which $g$ is elliptic.  This contradicts the assumption that $g$ is not simple with respect to $\calA$.  Hence there is only one orbit of vertices in $V_0(\calT_{\rm cut})$.  This shows~\eqref{item:B single}.   

Dasgupta--Hruska prove that $(G_B,\calP_B \cup \calN^B_g)$ is relatively hyperbolic with boundary homeomorphic to $B$~\cite[Theorem~1.2(2)]{ar:DH}.  Therefore, we have that \eqref{item:B boundary} follows immediately from Lemma~\ref{lem:boundary}.

For \eqref{item:B dichotomy}, we apply Proposition~\ref{prop:trichotomy} to the decomposition space $\calD_{(G_B,\calP_B)}(\calL_g^B)$ and note that $B$, which is homeomorphic to $\calD_{(G_B,\calP_B)}(\calL_g^B)$, is connected and does not have a cut point.

To prove \eqref{item:B cut pairs}, we start with $\{x_0,x_1\}$ that is an inseparable loxodromic cut pair in $\calD_{(G_B,\calP_B)}(\calL_g^B)$.  As this cut pair is loxodromic, we have that $x_0 = q_{(G_B,\calP_B)}(a^\infty)$ and $x_1 = q_{(G_B,\calP_B)}(a^\infty)$ for some non-peripheral element $a \in G_B$.  The element $a$ is non-peripheral with respect to $\calA$ as described above in~\eqref{item:B simple} and hence neither $x_0$ nor $x_1$ are cut points of $\calD(\calL_g)$.  It is now clear that $\{x_0,x_1\} \subset \calD(\calL_g)$ is a loxodromic cut pair.  Since any cut pair in $\calD(\calL_g)$ is contained in some $B' \in V_0(\calT_{\rm cut})$, (for instance, see \cite[Lemma~2.5]{ar:HH}) we have that $\{x_0,x_1\}$ is also an inseparable loxodromic cut pair in $\calD(\calL_g)$.      
\end{proof}


\section{Quadratic elements in free products}\label{sec:quadratic elements}

In this section we investigate the case where $\calD(\calL_g)$, or $\calD_{(G_B,\calP_B)}(\calL_g^B)$ in the case of cut points, is homeomorphic to a circle.  The main result in this section is Proposition~\ref{prop:quadratic} which characterizes such elements $g \in G$.  To begin, we recall the definition of a quadratic element in the setting of free products due to Guirardel--Horbez~\cite[Definition~2.14]{ar:GH22}.  In the setting of free groups, such an element corresponds to the boundary of a connected surface with a single boundary component.

Let $(G,\calA)$ be a fixed non-sporadic torsion-free free product.      

\begin{definition}\label{def:quad}
Let $g$ be a non-peripheral element of $(G,\calA)$ that is not simple.  We say $g$ is \emph{quadratic} if there is a (possibly trivial) $(G,\calA)$--tree $Q$ such that:
\begin{enumerate}

\item For some vertex $v$ in $Q$, there is a compact, connected 2--orbifold $\Sigma$ such that $\Stab_Q(v) \cong \pi_1(\Sigma)$.

\item Every edge of $Q$ has a translate that is incident to $v$.

\item Every incident edge group and every peripheral subgroup in $\Stab_Q(v)$ is conjugate into a boundary or a conical subgroup of $\pi_1(\Sigma)$.

\item The stabilizer of any edge is peripheral and the stabilizer of any vertex that is not in the orbit of $v$ is peripheral.

\item The element $g$ is the generator for a boundary component of $\Sigma$.

\end{enumerate}
\end{definition}

As remarked by Guirardel--Horbez~\cite{ar:GH22}, a subgroup of $\pi_1(\Sigma)$ corresponding to a boundary component of $\Sigma$ other than the one generated by $g$ is necessarily peripheral as $g$ is not simple.  As we are assuming that $G$ is torsion-free, we may assume that $\Sigma$ in the definition is a surface and hence there are no conical subgroups. 

\begin{example}\label{ex:quadratic}
Let $G$ be the free group of rank five with a basis $\{a,b,c,d,e\}$ and consider the free factor system  $\calA = \{[\I{a,b}],[\I{e}]\}$.  Fix a non-trivial element $x \in \I{a,b}$ and we consider the free product with amalgamation decomposition of $G$ as $\I{a,b} *_{\I{x}} \I{x,c,d,e}$.  There is an identification of $\I{x,c,d,e}$ with $\pi_1(\Sigma)$ where $\Sigma$ is a torus with three boundary components, one corresponding to $x$, another to $e$ and the last to $g = d^{-1}xcdc^{-1}e$.  The Bass-Serre tree $Q$ associated to the splitting shows that $g$ is quadratic.  The graph of spaces corresponding to this splitting is shown in Figure~\ref{fig:quadratic}.  

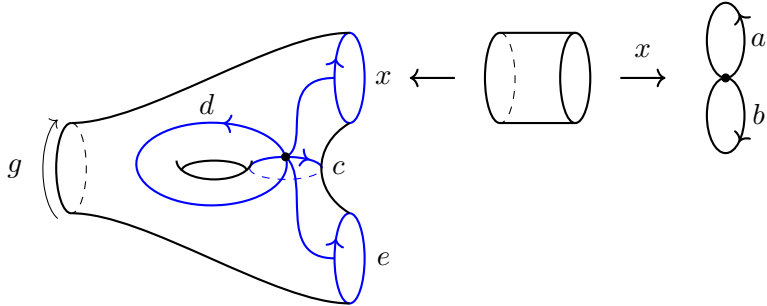
\begin{figure}[h]
\begin{tikzpicture}
\draw[thick,blue,->-=0.8] (0.67,0.0) arc (180:0:0.48cm and 0.15cm);
\draw[dashed,blue] (0.67,0.0) arc (180:360:0.48cm and 0.15cm);
\draw[thick,blue,->-=0.2] (1.15,0.15) arc (10:370: 1cm and 0.55cm);
\draw[thick,blue] (1.15,0.15) .. controls (1.5,-0.15) and (1,-1.2) .. (1.8,-1.2);
\draw[thick,blue] (1.15,0.15) .. controls (1.5,0.3) and (1,1.2) .. (1.8,1.2);
\draw[dashed] (-1.7,0.6) arc (90:-90:0.2cm and 0.6cm);
\draw[thick] (-1.7,0.6) arc (90:270:0.2cm and 0.6cm);
\draw[thick,blue,->-=0.1] (1.8,1.2) arc (180:-180:0.2cm and 0.6cm); 
\draw[thick,blue,->-=0.1] (1.8,-1.2) arc (180:-180:0.2cm and 0.6cm);
\draw[thick] (2,0.6) .. controls (1.5,0.3) and (1.5,-0.3) .. (2,-0.6);
\draw[thick] (-1.7,0.6) .. controls (-0.7, 0.6) and (1,1.8) .. (2,1.8);
\draw[thick] (-1.7,-0.6) .. controls (-0.7, -0.6) and (1,-1.8) .. (2,-1.8);
\draw[->] (-1.88,-0.67) arc (250:110:0.3cm and 0.7cm) node[pos=0.5,label=left:{$g$}] {};
\filldraw (1.15,0.15) circle [radius=1.5pt];
\node at (1.85,0) {$c$};
\node at (0.1,0.85) {$d$};
\node at (2.45,-1.2) {$e$};
\node at (2.45,1.2) {$x$};
\begin{scope}[xshift=0.2cm,yshift=0.1cm]
\draw[thick] (-0.43,-0.13) arc (180:0:0.43 cm and 0.13cm);	
\draw[thick] (0.5,0) arc (0:-180:0.5cm and 0.25cm);
\end{scope}
\begin{scope}[xshift=4.5cm,yshift=1.2cm]
\draw[thick] (-0.5,0.6) -- (0.5,0.6) (-0.5,-0.6) -- (0.5,-0.6);
\draw[dashed] (-0.5,0.6) arc (90:-90:0.2cm and 0.6cm);
\draw[thick] (-0.5,0.6) arc (90:270:0.2cm and 0.6cm);
\draw[thick] (0.5,0) circle [x radius=0.2cm, y radius=0.6cm];
\draw[thick,->] (1.1,0) -- (1.7,0) node[pos=0.5,label=above:{$x$}] {};
\draw[thick,->] (-1.1,0) -- (-1.7,0);
\end{scope}
\begin{scope}[xshift=7cm,yshift=1.2cm]
\draw[thick,->-=0.4] (0,0) arc (-90:270:0.25cm and 0.5cm) node[pos=0.2,label={[xshift=0.2cm,yshift=-0.2cm]$a$}] {};
\draw[thick,->-=0.4] (0,0) (0,0) arc (90:-270:0.25cm and 0.5cm) node[pos=0.2,label={[xshift=0.2cm,yshift=-0.55cm]$b$}] {};
\filldraw (0,0) circle [radius=1.5pt];
\end{scope}
\end{tikzpicture}
\caption{The quadratic element in $(G,\calA)$ from Example~\ref{ex:quadratic}.}\label{fig:quadratic}
\end{figure}

We will now describe the decomposition space $\calD(\calL_g)$.  Let $v$ be the vertex of $Q$ fixed by $\I{c,d,e,x}$, let $w$ be the vertex of $Q$ fixed by $\I{a,b}$ and let $\varepsilon$ be the edge incident to $v$ and $w$.  Set $G_v = \{c,d,e,x\}$ and $\calP = \{[\I{e}],[\I{x}],[\I{g}]\}$.  Then as a relatively hyperbolic group, the boundary of $(G_v,\calP)$ is homeomorphic to a circle, which we denote by $C_v$.  We replace the vertex $v$ in $Q$ by $C_v$ and for each $h \in G_v$, we attach the end of the edge $h\varepsilon$ incident to $v$ to the unique point fixed by $hxh^{-1}$ in $C_v$.  Repeat this step in an equivariant way at each vertex in the orbit of $v$ replacing the vertex $v'$ by the circle $C_{v'}$ and attaching the edges as before.  Now collapse every translate of $\varepsilon$ in the resulting space.  This new space, which we will denote $Q_C$, can be described as the union of the circles $C_{v'}$ over vertices $v'$ in the orbit of $v$ where $C_{v'} \cap C_{v''}$ consists of a single point if the simplicial distance between $v'$ and $v''$ in $Q$ is 2, and is empty otherwise.  The decomposition space $\calD(\calL_g)$ is the union of $Q_C \cup \bd_\infty Q$.  The topology on this space is similar to the observers topology.  Specifically, suppose that $(x_n)$ is a sequence of points in $Q_C$ and let $v_n$ be a corresponding sequence of vertices in $Q$ where $x_n \in C_{v_n}$ (note this sequence $(v_n)$ is not unique).  By passing to a subsequence of $(v_n)$, we can assume that $v_n \to v_\infty$ in $\hT_{\rm obs}$.  There are three cases that we describe now.
\begin{enumerate}
\item If $v_\infty \in \bd_\infty Q$, then $x_n \to v_\infty$. 
\item Suppose that $v_n$ is eventually constant.  Thus $x_n \in C_{v_\infty}$ for large enough $n$.  As $C_{v_\infty}$ is compact, we can pass to a subsequence to get $x_n \to x_\infty$ where $x_\infty \in C_{v_\infty}$.
\item Otherwise, for each $n$, there is a unique point $x'_n \in C_{v_\infty}$ such that $Q_C$ decomposes into two connected subsets $Q_0$ and $Q_1$ where $Q_0 \cap Q_1 = \{x'_n\}$, $x_n \in Q_0$ and $C_{v_\infty} \subset Q_1$.  Again, as $C_{v_\infty}$ is compact, we can pass to a subsequence to get $x'_n \to x_\infty$ where $x_\infty \in C_{v_\infty}$.  In this case $x_n \to x_\infty$ also.
\end{enumerate}

We note that $Q$ is the Dasgupta--Hruska cut point tree $\calT_{\rm cut}$ for $\calD(\calL_g)$.  The above discussion shows that $\calD(\calL_g)$ is the \emph{inverse limit of a tree of circles}~\cite[Section~4.1]{ar:DH}.
\end{example}

We will now show that quadratic elements in $G$ are precisely those for which the decomposition space $\calD(\calL_g)$, or $\calD_{(G_B,\calP_B)}(\calL_g^B)$ in the case of cut points as in Example~\ref{ex:quadratic}, is homeomorphic to a circle.  In the setting $\calA = \emptyset$, this was originally shown by Otal~\cite[Theorem~2]{ar:Otal92}, see also the work of Cashen--Macura~\cite[Theorem~6.1]{ar:CM11}.  First, we require one more definition due to Guirardel--Horbez~\cite[Lemma~2.19]{ar:GH22}.   

\begin{definition}\label{def:quad in T}
Let $g$ be a non-peripheral element in $G$ that is not simple and let $T \in \calO$ be a Grushko tree.  We say $g$ is \emph{quadratic in $T$} if its axis $T_g$ intersects every orbit of edges in $T$ exactly twice (regardless of orientation).
\end{definition}

Note, the definition of \emph{quadratic in T} by Guirardel--Horbez also requires that $g$ is not conjugate to it inverse.  As we are requiring that $G$ is torsion-free, this is not possible.  (Although, it can be shown using Whitehead graphs that Definition~\ref{def:quad in T} implies that $g$ is not conjugate to its inverse.)  For a definition of the notion of an inverse limit of a tree of compacta see the work of Dasgupta--Hruska and the references within~\cite[Definition~4.3]{ar:DH}. 

\begin{proposition}\label{prop:quadratic}
Let $g$ be a non-peripheral element of $(G,\calA)$ that is not simple.  The following are equivalent.
\begin{enumerate}
\item\label{item:quad-quad} The element $g$ is quadratic.

\item\label{item:quad-circle} The decomposition space $\calD(\calL_g)$ is homeomorphic to an inverse limit of a tree of circles (possibly trivial).  

\item\label{item:quad-WH circle} For any Grushko tree $T \in \calO$ that is Whitehead reduced for $g$ and for any vertex $v \in T$, the Whitehead graph $\Wh_T(\calL_g,v)$ is a circle.

\item\label{item:quad-quad in T} For any Grushko tree $T \in \calO$ that is Whitehead reduced for $g$, $g$ is quadratic in $T$.

\end{enumerate}
\end{proposition}

\begin{proof}
Let $g$ be a non-peripheral element of $(G,\calA)$ that is not simple.

First we suppose that $g$ is quadratic.  Let $Q$ be the $(G,\calA)$--tree as in Definition~\ref{def:quad} and let $v$ be the vertex in $Q$ where $\Stab_Q(v) \cong \pi_1(\Sigma)$ for some compact, connected surface $\Sigma$.    

There are two cases, depending on whether or not $Q$ is trivial, i.e., a single vertex.  If $Q$ is trivial, then every boundary subgroup of $\Sigma$, other than the one generated by $g$, is a maximal peripheral subgroup.  This implies that each of the peripheral factors in $\calA$ is cyclic and that $G$ is a free group.  The universal cover of $\Sigma$ embeds in the hyperbolic plane and the composition of the quotient maps $\bd G \to \bd(G,\calA) \to \calD(\calL_g)$ is the circular version of the Cantor map.  See~\cite[Example~1.3]{ar:Cashen16} for more details.  Hence $\calD(\calL_g)$ is homeomorphic to a circle, which trivially can be viewed as the inverse limit of a tree of circles.

Next, we assume that $Q$ is non-trivial.  It is easy to easy that $Q$ is a JSJ decomposition of $(G,\calA \cup \calN_g)$ over parabolic subgroups and thus without loss of generality, we can assume that it represents the canonical JSJ tree of cylinders for splittings of $(G,\calA \cup \calN_g)$ over peripheral subgroups (see~\cite[Remark~2.13]{ar:GH22}).  Thus $Q$ is the cut point tree constructed by Dasgupta--Hruska~\cite[Theorem~1.2(1)]{ar:DH}.  The vertex $v \in Q$ corresponds to a vertex $B \in V_0(\calT_{\rm cut})$.  Let $G_B$, $\calP_B$ and $\calL_g^B$ be as in Lemma~\ref{lem:B}.  By the first case when $Q$ is trivial, we have that $\calD_{(G_B,\calP_B)}(\calL_g^B)$ is a circle.  Hence, as every vertex in $V_0(\calT_{\rm cut})$ is a translate of $B$, we have that $\calD(\calL_g)$ is homeomorphic to an inverse limit of a tree of circles by the result of Dasgupta--Hruska~\cite[Theorem~1.2(3)]{ar:DH}.  This shows that \eqref{item:quad-quad} implies \eqref{item:quad-circle}.

Now suppose that $\calD(\calL_g)$ is homeomorphic to an inverse limit of a tree of circles.  In particular, there are only two types of minimal cut sets in $\calD(\calL_g)$:
\begin{enumerate}[label=(\roman*)]
\item a cut point corresponding to a vertex in $V_1(\calT_{\rm cut})$; or
\item a cut pair that belong to one of the circles comprising the inverse limit.
\end{enumerate} 
In particular, any minimal cut set that is not a single cut point consists of two points.  By Lemmas~\ref{lem:parabolic cut} and \ref{lem:edge cut set}, if $T \in \calO$ is Whitehead reduced for $g$ and $v \in T$ is a vertex, then every edge cut set in $\Wh_T(\calL_g,v)$ consists of exactly two edges.  (Note $q(g^\infty)$ cannot be a cut point as $g$ is not simple, c.f. proof of Lemma~\ref{lem:cut point tree}.)  Since $g$ is not simple, the Whitehead graph $\Wh_T(\calL_g,v)$ is connected and hence it is a circle.  Thus \eqref{item:quad-circle} implies \eqref{item:quad-WH circle}.

Suppose that $T \in \calO$ is a Grushko tree such that for any vertex $v \in T$, the Whitehead graph $\Wh_T(\calL_g,v)$ is a circle.  Thus every vertex in the Whitehead graph has degree equal to two.  This implies that every edge incident to $v$ is crossed by exactly two translates of $T_g$.  As this holds for every vertex in $T$, we have that $g$ is quadratic in $T$.  Hence \eqref{item:quad-WH circle} implies \eqref{item:quad-quad in T}.

Guirardel--Horbez prove that if $g$ is quadratic in some Gruskho tree $T \in \calO$, then $g$ is quadratic whenever $g$ is not conjugate to its inverse~\cite[Lemma~2.19]{ar:GH22}.  As explained before, as we are assuming that $G$ is torsion-free, $g$ cannot be conjugate to its inverse.  Hence this result of Guirardel--Horbez gives us that \eqref{item:quad-quad in T} implies \eqref{item:quad-quad}.
\end{proof}

We record the following consequence of Lemma~\ref{lem:B} and Proposition~\ref{prop:quadratic}.

\begin{corollary}\label{co:inseparable}
If a non-peripheral element $g \in G$ is $\calZ$--simple but not simple nor quadratic, then $\calD(\calL_g)$ contains an inseparable loxodromic cut pair.
\end{corollary}

\begin{proof}
Indeed, using the notation from Lemma~\ref{lem:B}, as $g$ is not quadratic, we must have that $\calD_{(G_B,\calP_B)}(\calL_g^B)$ contains an inseparable loxodromic cut pair by Lemma~\ref{lem:B}\eqref{item:B dichotomy} and Proposition~\ref{prop:quadratic}.  By Lemma~\ref{lem:B}\eqref{item:B cut pairs} this is also an inseparable cut pair in $\calD(\calL_g)$.
\end{proof}


\section{Proof of Theorem~\ref{thm:bounded projections} for quadratic elements}\label{sec:quadratic}

Using Proposition~\ref{prop:quadratic} and a corollary of Guirardel--Horbez, we can prove Theorem~\ref{thm:bounded projections} for quadratic elements with a strategy that is similar to the one for simple elements in Section~\ref{sec:simple}.

\begin{proposition}\label{prop:length bounded quadratic}
Let $(G,\calA)$ be a non-sporadic torsion-free free product.  For all $L > 0$, there is a $D_1 > 0$ such that if $g \in G$ is quadratic, then the diameter of $\pi(\calO_L(g)) \subset \ZF$ is at most $D_1$.
\end{proposition}

\begin{proof}
Given $L$, we set $D_1 = 2L + 5$.

Let $g \in G$ be a quadratic element and consider a Grushko tree $T_0 \in \calO_L(g)$.  As $g$ is not simple, by Corollary~\ref{co:whitehead reduced} there is a Grushko tree $T \in \calO_L(g)$ where $d(\pi(T_0),\pi(T)) \leq L$ that is Whitehead reduced for $g$.  By Proposition~\ref{prop:quadratic}, we have that $g$ is quadratic in $T$.  Guirardel--Horbez proved that in this case, there is a $\calZ$--splitting $S_0 \in \ZF$ such that $d(\pi(T),S_0) \leq 2$ and in which $g$ is elliptic~\cite[Corollary~2.20]{ar:GH22}.  We have that $d(\pi(T_0),S_0) \leq L+2$.

Given any other Grushko tree $T_1 \in \calO_L(g)$, repeating the argument from above, we have that there is a $\calZ$--splitting $S_1 \in \ZF$ in which $g$ is elliptic and where $d(\pi(T_1),S_1) \leq L + 2$. 

As $g$ is elliptic in both $S_0$ and $S_1$, we have that $d(S_0,S_1) \leq 1$.  Therefore, we find:
\begin{equation*}
d(\pi(T_0),\pi(T_1)) \leq d(\pi(T_0),S_0) + d(S_0,S_1) + d(S_1,\pi(T_1)) \leq 2L + 5 = D_1.\qedhere
\end{equation*} 
\end{proof}


\section{Finding bounded length splitting elements}\label{sec:short}

The main result of this section is Proposition~\ref{prop:splitting bound}.  This proposition states that for a $\calZ$--simple element $g \in G$ that is not simple nor quadratic, given a Grushko tree $T \in \calO_L(g)$ that is Whitehead reduced for $g$, there is a $\calZ$--splitting $S$ where $g$ is elliptic and an element $a$ that fixes an edge in $S$ whose length is bounded in terms of $L$.  This is the main ingredient for proving Theorem~\ref{thm:bounded projections} for $\calZ$--simple elements, which is carried out in the next section in Proposition~\ref{prop:length bounded elliptic}.  The element $a$ essentially plays the role of a boundary curve for the subsurface filled by a given curve $\gamma \subset \Sigma$ as described in the ``non-filling'' case of Theorem~\ref{thm:bounded surface projections} in the Introduction.  The strategy for finding the element $a$ is similar to that carried out by Cashen--Macura~\cite{ar:CM11} and Cashen~\cite{ar:Cashen16} who proved a similar statement in the case when $\calA = \emptyset$  and thus $G$ is free.  Specifically, we seek to analyze cut pairs in the decomposition space $\calD(\calL_g)$.   

Let $(G,\calA)$ be a fixed non-sporadic torsion-free free product.  The next two lemmas will be used to control the number of components as we start to splice together Whitehead graphs.     

\begin{lemma}\label{lem:local-model}
Let $\calL$ be a periodic line collection.  Suppose that $T \in \calO$ is a Grushko tree that is Whitehead reduced for $\calL$.  Let $v$ be a vertex in $V_\infty(T)$, fix two distinct vertices $Y, Y' \in T_v(\calL)$, and consider the subgraph $K \subset T_v(\calL)$ spanned by the directions at $v$ other than $Y$ and $Y'$.  The following statements hold.
\begin{enumerate}
\item\label{item:local-model infinite} The vertex $Y$ is adjacent in $T_v(\calL)$ to a vertex $Z \in K$ that belongs to an infinite component of $K$.  Likewise, the vertex $Y'$ is adjacent in $T_v(\calL)$ to a vertex $Z' \in K$ that belongs to an infinite component of $K$. 

\item\label{item:local-model finite} For each finite component $K_0 \subset K$, there are vertices $Z,Z' \in K_0$ such that $Y$ is adjacent to $Z$ and $Y'$ is adjacent to $Z'$ in $T_v(\calL)$. 

\end{enumerate}
\end{lemma}

We note the possibility that $Z = Z'$ in the above statements.

\begin{proof}
We consider the embedded line $\alpha_Y \from \RR\to T_v(\calL)$ from Lemma~\ref{lem:infinite line}.  If $Y' \notin \alpha_Y([0,\infty))$ we set $Z = \alpha_Y(1)$.  Since $\alpha_Y([1,\infty)) \subseteq K$ we see that $Z$ belongs to an infinite component of $H$.  If $Y' \notin \alpha_Y((-\infty,0])$ we set $Z = \alpha_Y(-1)$ and argue similarly.  As $Y \neq Y'$ and $\alpha_Y$ is an embedding, we cannot have that both $Y' \in \alpha_Y((-\infty,0])$ and $Y' \in \alpha_Y([0,\infty)$.  Thus $Z$ is well-defined.  We can repeat the argument using $\alpha_{Y'}$ to find $Z' \in K$ that is adjacent to $Y'$ and that lies in an infinite component of $K$.  This proves~\eqref{item:local-model infinite}.

Suppose that $K_0 \subset K$ is a finite component, fix a vertex $Y_0 \in K_0$ and consider the embedded line $\alpha_{Y_0} \from \RR \to T_v(\calL)$ from Lemma~\ref{lem:infinite line}.  As $Y_0$ is in a finite component of $K$, we must have a pair of integers $i < 0 < j$ where, after possibly precomposing $\alpha_{Y_0}$ by $t \mapsto -t$, $\alpha_{Y_0}(i) = Y$ and $\alpha_{Y_0}(j) = Y'$.  We set $Z = \alpha_{Y_0}(i+1)$ and $Z' = \alpha_{Y_0}(j-1)$.  Then $Y$ is adjacent to $Z$ and $Y'$ is adjacent to $Z'$.  Moreover, $Z$ and $Z'$ belong to $K_0$ as evidenced by the subpath of $\alpha_{Y_0}$ from $\alpha_{Y_0}(i+1)$ to $\alpha_{Y_0}(j-1)$.  This proves~\eqref{item:local-model finite}.
\end{proof}

\begin{lemma}\label{lem:loose ends}
Let $\calL$ be a periodic line collection.  Suppose that $T \in \calO$ is a Grushko tree that is Whitehead reduced for $\calL$.  Let $X \subset T$ be an arc whose endpoints are the midpoints of distinct edges $e_-$ and $e_+$, and let $Y_-$ \textup{(}respectively $Y_+$\textup{)} denote the direction of $T- X$ that meets $e_-$ \textup{(}respectively $Y_+$\textup{)}.  The following are true.  
\begin{enumerate}
\item\label{item:component at end} Each component of the graph $\Wh_T(\calL,X) - \{[Y_-],[Y_+]\}$ has a loose end both at $[Y_-]$ and at $[Y_+]$.
\item \label{item:bound on components} The number of components of the graph $\Wh_T(\calL,X) - \{[Y_-],[Y_+]\}$ is bounded above by $\#\abs{\calL}_T$.
\end{enumerate}
\end{lemma}

\begin{proof}
We prove both statements at the same time using induction on the number of vertices in $T$ that belong to $X$ via splicing finite stars.  

Suppose that there is a single vertex $v$ in $T$ that belongs to $X$.  If $\Stab_T(v) = \{1\}$, then $\Wh_T(\calL,X)$ is isomorphic to $\Wh_T(\calL,v)$ as the inclusion $T - X \subset T - v$ induces a bijection on directions.  As $T$ is Whitehead reduced for $\calL$, this graph does not have an admissible cut.  Therefore it is connected and does not have a cut vertex.  If some component $U$ of $\Wh_T(\calL,X) - \{[Y_-],[Y_+]\}$ does not have a loose end at $[Y_-]$ nor at $[Y_+]$, then $U$ is also a component of $\Wh_T(\calL,X)$ and therefore $\Wh_T(\calL,X)$ is not connected, which is a contradiction.  Likewise, if some component of $\Wh_T(\calL,X) - \{[Y_-],[Y_+]\}$ has a loose end at $[Y_-]$ but not at $[Y_+]$, then $[Y_-]$ is a cut vertex for $\Wh_T(\calL,X)$, which again is a contradiction.  A similar statement holds by switching the roles of  $[Y_-]$ and $[Y_+]$.  As each component of $\Wh_T(\calL,X) - \{[Y_-],[Y_+]\}$ has a loose end at $[Y_-]$, the number of components in this graph is bounded above by the valence of the vertex $[Y_-]$ in $\Wh_T(\calL,X)$.  This quantity is bounded above by the number of lines $\ell \in \calL$ such that $\ell_T$ meets $e_-$, which is bounded above by $\abs{\calL}_T$.  This verifies the base cases for both statements when $\Stab_T(v) = \{1\}$.   
 
If $\Stab_T(v) = \{\infty\}$ then we consider the graph $T_v(\calL)$ and apply Lemma~\ref{lem:local-model}.  Both $Y_-$ and $Y_+$ represent vertices of $T_v(\calL)$ and the graph $T_v(\calL) - X$ is the subgraph spanned by the complement of these two vertices.  By Lemma~\ref{lem:local-model}\eqref{item:local-model infinite}, the vertex $\hv \in V_{X}(\calL)$ representing every infinite component of $T_v(\calL) - X$ has an edge in $\Wh_T(\calL,v)$ connecting it to $[Y_-]$ and an edge in $\Wh_T(\calL,v)$ connecting it to $[Y_+]$.  Any other vertex $Z$ in $V_{X}(\calL)$ is a vertex of $T_v(\calL) - X$ that belongs to a finite component.   Hence, by Lemma~\ref{lem:local-model}\eqref{item:local-model finite}, there is an edge path in $\Wh_T(\calL,v)$ connecting $[Z]$ to $[Y_-]$ and an edge path in $\Wh_T(\calL,v)$ connecting $[Z]$ to $[Y_+]$.  As before, this shows that the number of components is bounded above by the valence of $Y_-$ in $T_v(\calL)$, which again as before, is bounded by $\#\abs{\calL}_T$.  This completes the verification of the base case.  

For the inductive step, we just observe that splicing together two graphs with loose ends where every component in each has a loose end at each missing vertex results in a graph with loose ends that also has this property and that the number of components does not increase.
\end{proof}

The following proposition shows that every inseparable loxodromic cut pair is given by an element with bounded length.  The argument below is modeled off of that by Cashen~\cite[Lemma~3.26 \& Proposition~4.13]{ar:Cashen16} (see also~\cite[Lemma~4.14]{ar:CM11}).  A version of this argument for hyperbolic groups was given by Barrett~\cite[Section~2.3]{ar:Barrett18}.   

\begin{proposition}\label{prop:small cut pairs}
For all $L > 0$, there is an $R_0 > 0$ with the following property.  If $g$ is a non-peripheral element of $G$, $T \in \calO_L(g)$ is a Grushko tree that is Whitehead reduced for $g$, and $\calD(\calL_g)$ contains an inseparable loxodromic cut pair $\{q(h^{\infty}),q(h^{-\infty})\}$, then there exists a non-peripheral element $a \in G$ where $\abs{a}_T \leq R_0$ and $T_a = T_h$.
\end{proposition}

\begin{proof}
Fix a non-peripheral element $g \in G$ and suppose that $T \in \calO_L(g)$ is a Grushko tree that is Whitehead reduced for $g$.  By Proposition~\ref{prop:decomposition connected}, we have that $\calD(\calL_g)$ is connected.  As $\{q(h^{\infty}),q(h^{-\infty})\}$ is a cut pair, by Lemma~\ref{lem:hull} we have that $\Wh_T(\calL_g,T_h)$ is not connected.  Enumerate the components of $\Wh_T(\calL_g,T_h)$ by $1,\ldots,c$ and observe that $c \leq L$ by Lemma~\ref{lem:loose ends}\eqref{item:bound on components} since $h$ acts cocompactly on $T_h$.  Set $R_0 = 2\xi(G,\calA)c^L + 1$.  Fix a segment $X_0 \subset T_h$ of length $R_0$.  There is some edge $e \subset X_0$ whose orbit meets $X_0$ with the same orientation at least $c^L+1$ times.

Let $\calL_e$ be the subset of lines $\ell \in \calL_g$ such that $\ell_T$ contains the edge $e$.  We observe that each line in $\calL_e$ corresponds to an edge in $\Wh_T(\calL,T_h)$.  The set $\calL_e$ contains $L_e$ elements where $L_e \leq L$.  Fix a bijection $\beta \from \{1,\ldots,L_e\} \to \calL_e$.  Consider the function $f \from \{1,\ldots,L_e\} \to \{1,\ldots, c\}$ where $f(i)$ records the component of $\Wh_T(\calL_g,T_h)$ that the edge associated to $\beta(i)$ belongs to.  

If $a' \in G$ is such that $a'e$ meets $X_0$ with the preferred orientation, then there is an associated bijection $\beta_{a'} \from \calL_e \to \calL_{a'e}$ given by $\beta_{a'}(\ell) = a'\ell$.  We get another function $f_{a'} \from \{1,\ldots,L_e\} \to \{1,\ldots, c\}$ using the bijection $\beta_{a'} \circ \beta$.  In other words, the function $f_{a'}(i)$ records the component of $\Wh_T(\calL_g,T_h)$ that the edge associated to $a'\beta(i)$ belongs to.  

As the orbit of $e$ meets $X_0$ with the same orientation at least $c^L + 1$ times, this procedure produces $c^L+1$ functions $f_{a'} \from \{1,\ldots,L_e\} \to \{1,\ldots,c\}$.  Hence two must be identical functions.  Let $a_0,a_1 \in G$ be two such elements where $f_{a_0} = f_{a_1}$ and set $a = a_1a_0^{-1}$.  We observe that the edges $a_0e$ and $a_1e$ belong to $T_a$ as they are coherently oriented.  In particular, $a$ is non-peripheral and $\abs{a}_T \leq R_0$

Let $x$ be the midpoint of $a_0e$ and let $X \subset T$ be the arc from $x$ to $ax$.  Let $Y_- \subset T$ be the direction of $T - X$ that meets $a_0e$ and let $Y_+ \subset T$ be the direction of $T - X$ that meets $a_1e$.  Splicing the translates of $\Wh_T(\calL_g,X) - \{[Y_-],[Y_+]\}$ by powers of $a$ together to form $\Wh_T(\calL_g,T_a)$ may introduce some additional connections between the components.  However, by the construction of $a$, we will never splice together components that lie in separate components of $\Wh_T(\calL_g,T_h)$.  Hence, we see that the number of components of $\Wh_T(\calL_g,T_a)$ is bounded below by $c$.  Thus $\{q(a^\infty),q(a^{-\infty})\}$ is a loxodromic cut pair.  It remains to show that $T_a = T_h$, in other words, $a$ and $h$ are powers of the same element in $G$.

Without loss of generality, we assume that $a$ and $h$ translate in the same direction along $T_a \cap T_h$ which necessarily includes the segment $X$.  Suppose that $T_a \neq T_h$.  Then $T_a \cup T_h - (T_a \cap T_h)$ decomposes into four subsets $X_{a,+},X_{a,-} \subset T_a$ and $X_{h,+},X_{h,-} \subset T_h$ where $aX_{a,+} \subset X_{a,+}$, $a^{-1}X_{a,-} \subset X_{a,-}$, $hX_{h,+} \subset X_{h,+}$, and $h^{-1}X_{h,-} \subset X_{h,-}$.  See Figure~\ref{fig:inseparable}.

\begin{figure}[ht]
\centering
\begin{tikzpicture}
\def\ee{1.8}
\def\len{6}
\def\sqr{0.707106781}
\node at (-1.25*\ee,0) {$h^{-\infty}$};
\node at (\len*\ee + 1.2*\ee,0) {$h^\infty$};
\node at (\ee-2.6*\sqr*\ee,2.6*\sqr*\ee) {$a^{-\infty}$};
\node at (5*\ee+2.6*\sqr*\ee,2.6*\sqr*\ee) {$a^\infty$};
\node at (5*\ee+\sqr*\ee + 1.7*\ee,\sqr*\ee) {$ah^\infty$};
\node at (3*\ee-1.7*\sqr*\ee,-1.7*\sqr*\ee) {$ah^{-\infty}$};
\filldraw[red] (-0.8*\ee,0) circle [radius=1pt];
\filldraw[red] (-0.7*\ee,0) circle [radius=1pt];
\filldraw[red] (-0.6*\ee,0) circle [radius=1pt];
\filldraw[red] (\len*\ee + 0.6*\ee,0) circle [radius=1pt];
\filldraw[red] (\len*\ee + 0.7*\ee,0) circle [radius=1pt];
\filldraw[red] (\len*\ee + 0.8*\ee,0) circle [radius=1pt];
\filldraw[black!30!green] (5*\ee + \sqr*\ee + 1.1*\ee, \sqr*\ee) circle [radius=1pt];
\filldraw[black!30!green] (5*\ee + \sqr*\ee + 1.2*\ee, \sqr*\ee) circle [radius=1pt];
\filldraw[black!30!green] (5*\ee + \sqr*\ee + 1.3*\ee, \sqr*\ee) circle [radius=1pt];
\filldraw[black!30!green] (3*\ee-1.1*\sqr*\ee,-1.1*\sqr*\ee) circle [radius=1pt];
\filldraw[black!30!green] (3*\ee-1.2*\sqr*\ee,-1.2*\sqr*\ee) circle [radius=1pt];
\filldraw[black!30!green] (3*\ee-1.3*\sqr*\ee,-1.3*\sqr*\ee) circle [radius=1pt];
\filldraw[blue] (\ee-2.1*\sqr*\ee,2.1*\sqr*\ee) circle [radius=1pt];
\filldraw[blue] (\ee-2.2*\sqr*\ee,2.2*\sqr*\ee) circle [radius=1pt];
\filldraw[blue] (\ee-2.3*\sqr*\ee,2.3*\sqr*\ee) circle [radius=1pt];
\filldraw[blue] (5*\ee+2.1*\sqr*\ee,2.1*\sqr*\ee) circle [radius=1pt];
\filldraw[blue] (5*\ee+2.2*\sqr*\ee,2.2*\sqr*\ee) circle [radius=1pt];
\filldraw[blue] (5*\ee+2.3*\sqr*\ee,2.3*\sqr*\ee) circle [radius=1pt];
\draw[very thick,black!30!green] (3*\ee-1*\sqr*\ee,-1*\sqr*\ee) -- (3*\ee,0) (5*\ee+\sqr*\ee,\sqr*\ee) -- (5*\ee+\sqr*\ee+\ee,\sqr*\ee);
\draw[very thick,blue] (\ee-2*\sqr*\ee,2*\sqr*\ee) -- (\ee,0) (5*\ee,0) -- (5*\ee+2*\sqr*\ee,2*\sqr*\ee);
\draw[very thick,red](-0.5*\ee,0) -- (\ee,0) (5*\ee,0) -- (6.5*\ee,0);
\draw[very thick] (\ee,0) -- (5*\ee,0);
\draw[thick,rounded corners] (2.15*\ee-0.25,-.8-0.25) -- (2.15*\ee,-0.8) -- (2.15*\ee,-0.15) -- (1.35*\ee,-0.15) -- (1.35*\ee,-1.3);
\node at (1.75*\ee,-0.4) {\footnotesize $\ell_0$};
\draw[thick,rounded corners] (1.3*\ee-0.25,-1-0.25) -- (1.3*\ee,-1) -- (1.3*\ee,-0.15) -- (0.8*\ee,-0.15) -- (0.8*\ee,-1.3);
\node at (1.05*\ee,-0.4) {\footnotesize $\ell_1$};
\draw[thick,rounded corners] (4.15*\ee-0.25,-.8-0.25) -- (4.15*\ee,-0.8) -- (4.15*\ee,-0.15) -- (3.35*\ee,-0.15) -- (3.35*\ee,-1.3);
\node at (3.75*\ee,-0.4) {\footnotesize $a\ell_0$};
\draw[thick,rounded corners] (3.3*\ee-0.25,-1-0.25) -- (3.3*\ee,-1) -- (3.3*\ee,-0.15) -- (3*\ee,-0.15) -- (3*\ee-0.3,-0.15-0.3) -- (3*\ee+0.2,-0.15-0.8);
\node at (3.08*\ee,-0.4) {\footnotesize $a\ell_1$};
\filldraw (1.5*\ee,0) circle [radius=2.5pt];
\filldraw (2*\ee,0) circle [radius=2.5pt];
\node at (1.75*\ee,0) [label=above:{$e$}] {};
\filldraw (3.5*\ee,0) circle [radius=2.5pt];
\filldraw (4*\ee,0) circle [radius=2.5pt];
\node at (3.75*\ee,0) [label=above:{$ae$}] {};
\end{tikzpicture}
\caption{The set-up and proof of Proposition~\ref{prop:small cut pairs}.  The sets $X_{a,\pm}$ are shown in blue, the sets $X_{h,\pm}$ are shown in red and the set $T_a \cap T_h$ is shown in black.}\label{fig:inseparable}

\end{figure}
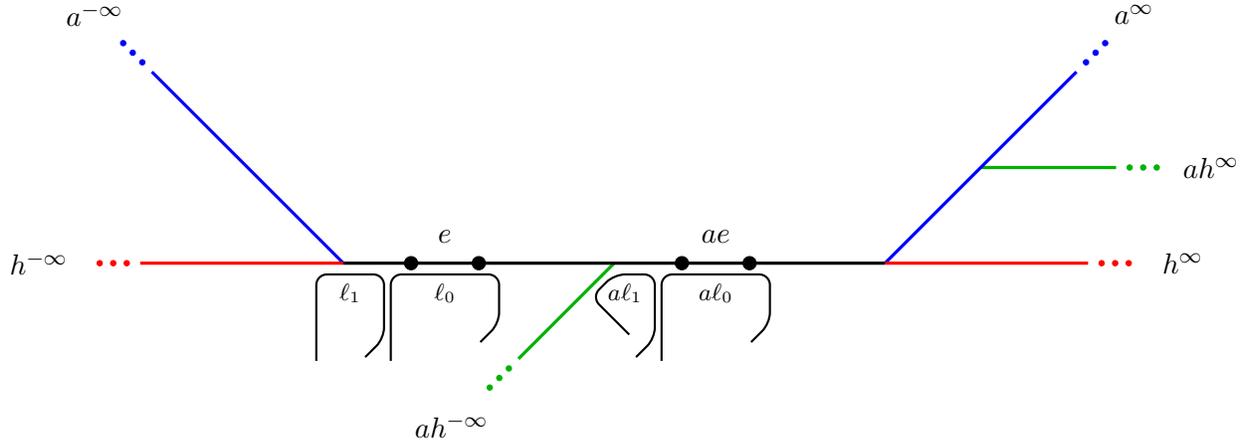

Let $Y_{a,+}$ be the direction of $T - T_h$ that contains $X_{a,+}$ and similarly define $Y_{a,-}$ as the direction of $T - T_h$ that contains $X_{a,-}$.  Since $\{q(a^{-\infty}),q(a^\infty)\}$ is a cut pair and as $\{q(h^{-\infty}),q(h^\infty)\}$ is an inseparable cut pair, the vertices $[Y_{a,-}]$ and $[Y_{a,+}]$ belong to the same component $C \subset \Wh_T(\calL_g,T_h)$.  Let $\ell$ be a geodesic that crosses the edge $e$ so that the edge of $\Wh_T(\calL_g,T_h)$ corresponding to $\ell$ does not belong to $C$.  Let $C'$ be the component $\Wh_T(\calL_g,T_h)$ that contains the edge corresponding to $\ell$. 

\medskip \noindent {\it Claim 1.} The points $q(ah^{-\infty})$ and $q(a^\infty)$ belong to different components of $\calD(\calL_g) - \{q(h^{-\infty}),q(h^\infty)\}$.

\begin{proof}[Proof of Claim 1]
There is a sequence of geodesics $\ell = \ell_0,\ldots,\ell_m$ in $\calL_g$ such that $\ell_m$ meets $X_{h,-}$ and where the edges corresponding to $\ell_{j-1}$ and $\ell_j$ in $\Wh_T(\calL_g,T_h)$ are incident on a common vertex for $j = 1,\ldots,m$.  By the choice of $a$, the edge corresponding to $a\ell$ belongs to $C'$ and the sequence of lines $a\ell = a\ell_0,\ldots,a\ell_m$ show that the edge corresponding to $a\ell_m$ also belongs to $C'$.  The direction $aX_{h,-} \subset T - T_h$ meets $a\ell_m$ and thus the vertex $[aX_{h,-}]$ also belongs to $C'$.  This verifies the claim.
\end{proof}

\medskip \noindent {\it Claim 2.} The points $q(ah^{\infty})$ and $q(a^\infty)$ belong to the same components of $\calD(\calL_g) - \{q(h^{-\infty}),q(h^\infty)\}$.

\begin{proof}[Proof of Claim 2]
This follows as the component $aX_{h,+}$ is contained in $X_{a,+}$.
\end{proof}

Claim 1 and Claim 2 imply that $\{q(h^{-\infty}),q(h^\infty)\}$ is separated by the cut pair $\{q(ah^{-\infty}),q(ah^\infty)\}$.  This is a contradiction as $\{q(h^{-\infty}),q(h^\infty)\}$ is inseparable.  Hence $T_a = T_h$ as desired.
\end{proof}

We will now prove the main result of this section.

\begin{proposition}\label{prop:splitting bound}
For all $L > 0$, there is an $R > 0$ with the following property.  If $g \in G$ is $\calZ$--simple but not simple nor quadratic and $T \in \calO_L(g)$ is a Grushko tree that is Whitehead reduced for $g$, then there exists a $\calZ$--splitting $S$ in which $g$ is elliptic and a non-peripheral element $a \in G$ with $\abs{a}_T \leq R$ and where $a$ fixes an edge in $S$. 
\end{proposition}

\begin{proof}
Let $g \in G$ be a non-peripheral $\calZ$--simple element that is not simple nor quadratic and fix a tree $T \in \calO_L(g)$ that is Whitehead reduced for $g$.  By Corollary~\ref{co:inseparable}, the decomposition space $\calD(\calL_g)$ contains an inseparable loxodromic cut pair.  

Let $R$ be the constant $R_0$ from Proposition~\ref{prop:small cut pairs} using $L$.  By Proposition~\ref{prop:small cut pairs}, there is a non-peripheral element $a \in G$ where $\abs{a}_{T} \leq R$ such that $\{q(a^\infty),q(a^{-\infty})\} \in \calD(\calL_g)$ is an inseparable loxodromic cut pair.  Thus the image of this cut pair in $\calD(\calL_g \cup \calL_a)$ is a cut point, denote it by $x$.  Let $\calT_{\rm cut}$ be the cut point tree for $\calD(\calL_g \cup \calL_a)$ and let $v \in V_1(\calT_{\rm cut})$ be the vertex corresponding to $x$.  As seen in the proof of Proposition~\ref{prop:small cut pairs}, the action of $a$ on $\calD(\calL_g) - \{q(a^\infty),q(a^{-\infty})\}$ preserves the components, and hence $a$ stabilizes an edge $e$ in $\calT_{\rm cut}$ that is incident to $x$.  Collapsing all edges of $\calT_{\rm cut}$ that do not lie in the orbit of $e$ results in the desired $\calZ$--splitting $S$.  Since the action of $G$ on $\calT_{\rm cut}$ is minimal, this $\calZ$--splitting is nontrivial.
\end{proof}


\section{Proof of Theorem~\ref{thm:bounded projections} for $\calZ$--simple elements}\label{sec:elliptic}

We can now prove Theorem~\ref{thm:bounded projections} for $\calZ$--simple elements.  As for simple elements (Proposition~\ref{prop:length bounded simple}) and quadratic elements (Proposition~\ref{prop:length bounded quadratic}), the strategy is that given $T \in \calO_L(g)$, we seek to find a $\calZ$--splitting in which $g$ is elliptic whose distance to $\pi(T)$ is bounded in terms of $L$.  To do so, we first find a $\calZ$--splitting in which $g$ is elliptic and where some edge stabilizer has length bounded in terms of $L$ using the work in Section~\ref{sec:short}.

\begin{proposition}\label{prop:length bounded elliptic}
Let $(G,\calA)$ be a non-sporadic torsion-free free product.  For all $L > 0$, there is a $D_2 > 0$ such that if $g \in G$ is $\calZ$--simple then the diameter of $\pi(\calO_L(g))$ is at most $D_2$.
\end{proposition}

\begin{proof}
Let $R$ be the constant $R$ from Proposition~\ref{prop:splitting bound} for $L$ and set $D_2 = 2L + 2R + 5$.
        
By Proposition~\ref{prop:length bounded simple}, if $g$ is simple then the diameter of $\pi(\calO_L(g))$ is at most $D_0 = 2L + 3 \leq D_2$.

By Proposition~\ref{prop:length bounded quadratic}, if $g$ is quadratic then the diameter of $\pi(\calO_L(g))$ is at most $D_1 = 2L + 5 \leq D_2$.

Lastly, we suppose that $g$ is not simple nor quadratic and consider a Grushko tree $T_0 \in \calO_L(g)$.   As $g$ is not simple, by Corollary~\ref{co:whitehead reduced} there is a Grushko tree $T \in \calO_L(g)$ where $d(\pi(T_0),\pi(T)) \leq L$ that is Whitehead reduced for $g$.  By Proposition~\ref{prop:splitting bound}, there is a $\calZ$--splitting $S_0 \in \ZF$ in which $g$ is elliptic and an element $a \in G$ where $\abs{a}_{T} \leq R$ and where $a$ fixes an edge in $S_0$.  By Lemma~\ref{lem:edge stabs in Z-splittings}, the element $a$ is simple.  As shown in the proof of Proposition~\ref{prop:length bounded simple}, since $a$ is simple, there is a $\calZ$--splitting $S \in \ZF$ where $a$ has a fixed point and such that $d(\pi(T),S) \leq R + 1$.  Since $a$ has a fixed point in both $S$ and $S_0$, we have $d(S,S_0) \leq 1$.  Hence 
\begin{equation*}
d(\pi(T_0),S) \leq d(\pi(T_0),\pi(T)) + d(\pi(T),S) + d(S,S_0) \leq L + R + 2.
\end{equation*}

Repeating this argument for another tree $T_1 \in \calO_L(g)$, we find another $\calZ$--splitting $S_1 \in \ZF$ in which $g$ is elliptic and such that $d(\pi(T_1),S_1) \leq L + R + 2$. Since $g$ is elliptic in both $S_0$ and $S_1$, we have $d(S_0, S_1)\leq 1$, and hence
\begin{equation*}
d(\pi(T_0),\pi(T_1)) \leq d(\pi(T_0),S_0) + d(S_0,S_1) + d(S_1,\pi(T_1)) \leq 2L + 2R + 5.\qedhere
\end{equation*}

\end{proof}


\section{Proof of Theorem~\ref{thm:bounded projections}}\label{sec:main proof}

To complete the proof of Theorem~\ref{thm:bounded projections} we use contradiction.  Supposing that for some fixed $L$ the diameter of $\pi(\calO_L(g_n))$ is unbounded for some sequence $(g_n)$ of non-peripheral elements, we will find a $\calZ$--simple element $g$ where the diameter of $\pi(\calO_{L'}(g))$ is infinite for some $L'$.  This is a contradiction by Proposition~\ref{prop:length bounded elliptic}.  

In order to obtain this contradiction, we must work with the closure $\overline{\calO}$ of the deformation space and we recall this space now.  For more information, see the work of Guirardel--Levitt~\cite{ar:GL07} or Horbez~\cite{ar:Horbez17}.  Given an Grushko tree $T \in \calO$ and a non-peripheral element $g \in G$, by $\|g\|_T$ we denote the translation length of $g$, that is $\|g\|_T = d_T(x,gx)$ for any $x \in T_g$.  Thus each tree determines a \emph{length function} $\| \param \|_T \from G \to \RR$.  Culler--Morgan proved that the assignment $T \mapsto \| \param \|_T$ defines an injective function $\calO \to \RR^G$~\cite{ar:CM87}.  The closure of $\calO$ in $\RR^G$ is denoted $\overline{\calO}$.  Horbez proved that $\overline{\calO}$ is projectively compact and identified this closure with a space of action of $G$ on $\RR$--trees~\cite[Proposition~2.3]{ar:Horbez17}.  Specifically, given a sequence $(T_n) \subset \overline{\calO}$, after passing to a subsequence, there are real numbers $(\lambda_n)$ such that the length functions $\frac{1}{\lambda_n}\| \param \|_{T_n}$ converge to the length function $\| \param \|_T$ of an $\RR$--tree equipped with an action $G$~\cite[Theorem~1]{ar:Horbez17}.  We observe that if $\| g \|_{T_n} \to \infty$ for some $g$, then necessarily we have that $\lambda_n \to \infty$.  

We repeat the statement of Theorem~\ref{thm:bounded projections} for the convenience of the reader.  

\begin{restate}{Theorem}{thm:bounded projections}
Let $(G,\calA)$ be a non-sporadic torsion-free free product.  For all $L > 0$, there is a $D > 0$ such that for any non-peripheral element $g \in G$, the diameter of $\pi(\calO_L(g)) \subset \ZF$ is at most $D$.
\end{restate}

\begin{proof}
Suppose that the theorem is false.  Then there is an $L$ such that for all $n \geq 0$ there is a non-peripheral element $g_n \in G$ for which the diameter of $\pi(\calO_L(g_n))$ is greater than $n$.  We fix Grushko trees $S_n,T_n \in \calO_L(g_n)$ with $d(\pi(S_n),\pi(T_n)) > n$.  As scaling the length of edges does not change $\abs{g_n}_{S_n}$ nor the image $\pi(S_n)$, we can scale all of the edges of $S_n$ to have length one so that $\abs{g_n}_{S_n} = \| g_n \|_{S_n}$.  We similarly scale the edges on each $T_n$.  After passing to subsequences and possibly replacing $g_n$, $S_n$ and $T_n$ respectively by $\theta_n^{-1}(g_n)$, $S_n\theta_n$ and $T_n\theta_n$ for some $\theta_n \in \Out(G,\calA)$, we can assume that $(S_n)$ is a constant sequence, $S_n = S$.  With $S$ and these new sequences $(g_n)$ and $(T_n)$, we have that $\|g_n\|_{S}, \|g_n\|_{T_n} \leq L$ and $d(\pi(S),\pi(T_n)) \to \infty$.  

If $h \in G$ is simple, then we must have that $\|h\|_{T_n} \to \infty$ by Proposition~\ref{prop:length bounded simple} as $d(\pi(S),\pi(T_n)) \to \infty$.  Therefore, as described in the beginning of this section, after passing to subsequences, we have a sequence $(\lambda_n)$ of positive real numbers and an $\RR$--tree $T \in \overline{\calO}$ where $\frac{1}{\lambda_n}\|\param\|_{T_n} \to \|\param\|_T$ and for which $\lambda_n \to \infty$.

Next, by passing to subsequences and replacing each element in the sequence $(g_n)$ by a conjugate if necessary, we can assume that there is an edge $e$ in $S$ that is contained in each of the axes $S_{g_n}$.  As $\bd(G,\calA)$ is compact, after passing to a subsequence, we have that $g^\infty_n \to \alpha \in \bd(G,\calA)$ and $g_n^{-\infty} \to \omega \in \bd(G,\calA)$.  As the axes $S_{g_n}$ all contain the edge $e$, we must have that $\alpha \neq \omega$.  


\medskip \noindent {\it Claim.} There is a non-peripheral element $g \in G$ for which $\|g\|_{T_n}$ stays bounded. 

\begin{proof}[Proof of Claim]
As the axes $S_{g_n}$ are periodic and $\abs{g_n}_S$ is bounded, we either have that $\alpha,\omega \in \bd_\infty(G,\calA)$ or that $\alpha,\omega \in V_\infty(G,\calA)$.  We deal with these cases one at a time.

First, assume that $\alpha,\omega \in \bd_\infty(G,\calA)$.  Thus, the axes of the elements $g_n$ overlap in longer and longer segments in $S$.  As the translation lengths are bounded, this means that the sequence is eventually constant, hence $g_n = g$ for some $g \in G$ and large enough $n$.  In this case, this is our non-peripheral element for which $\|g\|_{T_n}$ stays bounded.

Otherwise, we have that $\alpha,\omega \in V_\infty(G,\calA)$.  Let $v$ and $w$ be the distinct vertices of $S$ that realize $\alpha$ and $\omega$.  Fix nontrivial elements $h_0 \in \Stab_S(v)$ and $h_1 \in \Stab_S(w)$ and set $g = h_0h_1$.  Now in $T_n$ the distance between $v_n$ and $w_n$, the realizations of $\alpha$ and $\omega$, stays bounded as $\|g_n\|_{T_n}$ is bounded.  This means that the translation length $\|g\|_{T_n}$ is bounded as well (it is twice the distance between $v_n$ and $w_n$).  In this case, this is our non-peripheral element for which $\|g\|_{T_n}$ stays bounded.      
\end{proof}

To conclude, we observe that as $\lambda_n \to \infty$, since $\|g\|_{T_n}$ stays bounded, we have that $g$ is elliptic in $T$ as $\|g\|_T = \lim_{n \to \infty} \frac{\|g\|_{T_n}}{\lambda_n}$.  It then follows from work of Guirardel--Levitt 
that $g$ is $\calZ$--simple~\cite[Corollary~9.10]{ar:GL15}.  However, boundedness of $\|g\|_{T_n} = \abs{g}_{T_n}$ implies that the diameter of $\pi(\calO_{L'}(g))$ is infinite for some $L'$.  This is a contradiction of Proposition~\ref{prop:length bounded elliptic}. 
\end{proof}


%
%
%
%


\bibliography{bib}

\begin{thebibliography}{10}

\bibitem{ar:Barrett18}
{\sc Barrett, B.}
\newblock Computing {JSJ} decompositions of hyperbolic groups.
\newblock {\em Journal of Topology 11}, 2 (2018), 527--558.

\bibitem{ar:BF14}
{\sc Bestvina, M., and Feighn, M.}
\newblock Hyperbolicity of the complex of free factors.
\newblock {\em Adv. Math. 256\/} (2014), 104--155.

\bibitem{ar:Birman}
{\sc Birman, J.~S.}
\newblock Mapping class groups and their relationship to braid groups.
\newblock {\em Comm. Pure Appl. Math. 22\/} (1969), 213--238.

\bibitem{ar:Bowditch98}
{\sc Bowditch, B.~H.}
\newblock Cut points and canonical splittings of hyperbolic groups.
\newblock {\em Acta Math. 180}, 2 (1998), 145--186.

\bibitem{ar:Bowditch12}
{\sc Bowditch, B.~H.}
\newblock Relatively hyperbolic groups.
\newblock {\em Internat. J. Algebra Comput. 22}, 3 (2012), 1250016, 66.

\bibitem{ar:Brink}
{\sc Brinkmann, P.}
\newblock Hyperbolic automorphisms of free groups.
\newblock {\em Geom. Funct. Anal. 10}, 5 (2000), 1071--1089.

\bibitem{ar:Cashen16}
{\sc Cashen, C.~H.}
\newblock Splitting line patterns in free groups.
\newblock {\em Algebraic \& Geometric Topology 16}, 2 (2016), 621--673.

\bibitem{ar:CM11}
{\sc Cashen, C.~H., and Macura, N.}
\newblock Line patterns in free groups.
\newblock {\em Geom. Topol. 15}, 3 (2011), 1419--1475.

\bibitem{ar:CM87}
{\sc Culler, M., and Morgan, J.~W.}
\newblock Group actions on {${\bf R}$}-trees.
\newblock {\em Proc. London Math. Soc. (3) 55}, 3 (1987), 571--604.

\bibitem{ar:CV86}
{\sc Culler, M., and Vogtmann, K.}
\newblock Moduli of graphs and automorphisms of free groups.
\newblock {\em Invent. Math. 84}, 1 (1986), 91--119.

\bibitem{ar:DH}
{\sc Dasgupta, A., and Hruska, G.~C.}
\newblock Local connectedness of boundaries for relatively hyperbolic groups.
\newblock {\em J. Topol. 17}, 2 (2024), Paper No. e12347, 34.

\bibitem{ar:DT17}
{\sc Dowdall, S., and Taylor, S.~J.}
\newblock The co-surface graph and the geometry of hyperbolic free group
  extensions.
\newblock {\em Journal of Topology 10}, 2 (2017), 447--482.

\bibitem{ar:DT18}
{\sc Dowdall, S., and Taylor, S.~J.}
\newblock Hyperbolic extensions of free groups.
\newblock {\em Geom. Topol. 22}, 1 (2018), 517--570.

\bibitem{ar:FM02}
{\sc Farb, B., and Mosher, L.}
\newblock Convex cocompact subgroups of mapping class groups.
\newblock {\em Geom. Topol. 6\/} (2002), 91--152.

\bibitem{ar:GH19}
{\sc Guirardel, V., and Horbez, C.}
\newblock Algebraic laminations for free products and arational trees.
\newblock {\em Algebr. Geom. Topol. 19}, 5 (2019), 2283--2400.

\bibitem{ar:GH22}
{\sc Guirardel, V., and Horbez, C.}
\newblock Boundaries of relative factor graphs and subgroup classification for
  automorphisms of free products.
\newblock {\em Geom. Topol. 26}, 1 (2022), 71--126.

\bibitem{ar:GL07}
{\sc Guirardel, V., and Levitt, G.}
\newblock The outer space of a free product.
\newblock {\em Proc. Lond. Math. Soc. (3) 94}, 3 (2007), 695--714.

\bibitem{ar:GL15}
{\sc Guirardel, V., and Levitt, G.}
\newblock Splittings and automorphisms of relatively hyperbolic groups.
\newblock {\em Groups, Geometry, and Dynamics 9}, 2 (2015), 599--663.

\bibitem{ar:GL17}
{\sc Guirardel, V., and Levitt, G.}
\newblock J{SJ} decompositions of groups.
\newblock {\em Ast\'{e}risque}, 395 (2017), vii+165.

\bibitem{un:Ham}
{\sc Hamenst{\"a}dt, U.}
\newblock Word hyperbolic extensions of surface groups.
\newblock {P}reprint,
  \href{https://arxiv.org/abs/math/0505244}{arxiv:math/0505244}, 2005.

\bibitem{ar:Haulmark19}
{\sc Haulmark, M.}
\newblock Local cut points and splittings of relatively hyperbolic groups.
\newblock {\em Algebraic \& Geometric Topology 19}, 6 (2019), 2795--2836.

\bibitem{ar:HH}
{\sc Haulmark, M., and Hruska, G.~C.}
\newblock On canonical splittings of relatively hyperbolic groups.
\newblock {\em Israel J. Math. 258}, 1 (2023), 249--286.

\bibitem{ar:Horbez17}
{\sc Horbez, C.}
\newblock The boundary of the outer space of a free product.
\newblock {\em Israel Journal of Mathematics 221}, 1 (2017), 179--234.

\bibitem{ar:KL09}
{\sc Kapovich, I., and Lustig, M.}
\newblock Geometric intersection number and analogues of the curve complex for
  free groups.
\newblock {\em Geom. Topol. 13}, 3 (2009), 1805--1833.

\bibitem{ar:KentLeininger}
{\sc Kent, A.~E., and Leininger, C.~J.}
\newblock Shadows of mapping class groups: capturing convex cocompactness.
\newblock {\em Geometric and Functional Analysis 18}, 4 (2008), 1270--1325.

\bibitem{ar:Knopf}
{\sc Knopf, S.}
\newblock Acylindrical actions on trees and the {F}arrell-{J}ones conjecture.
\newblock {\em Groups Geom. Dyn. 13}, 2 (2019), 633--676.

\bibitem{ar:LL}
{\sc Levitt, G., and Lustig, M.}
\newblock Irreducible automorphisms of {$F_n$} have north-south dynamics on
  compactified outer space.
\newblock {\em J. Inst. Math. Jussieu 2}, 1 (2003), 59--72.

\bibitem{ar:Manning10}
{\sc Manning, J.~F.}
\newblock Virtually geometric words and {W}hitehead's algorithm.
\newblock {\em Math. Res. Lett. 17}, 5 (2010), 917--925.

\bibitem{ar:Otal92}
{\sc Otal, J.-P.}
\newblock Certaines relations d'\'{e}quivalence sur l'ensemble des bouts d'un
  groupe libre.
\newblock {\em Journal of the London Mathematical Society. Second Series 46}, 1
  (1992), 123--139.

\bibitem{ar:Shenitzer55}
{\sc Shenitzer, A.}
\newblock Decomposition of a group with a single defining relation into a free
  product.
\newblock {\em Proc. Amer. Math. Soc. 6\/} (1955), 273--279.

\bibitem{ar:Swarup86}
{\sc Swarup, G.~A.}
\newblock Decompositions of free groups.
\newblock {\em J. Pure Appl. Algebra 40}, 1 (1986), 99--102.

\bibitem{thurston98b}
{\sc Thurston, W.~P.}
\newblock Hyperbolic {S}tructures on 3-manifolds, {II}: {S}urface groups and
  3-manifolds which fiber over the circle.
\newblock {P}reprint,
  \href{https://arxiv.org/abs/math/9801045}{arxiv:math/9801045}, 1998.

\bibitem{ar:Whitehead36}
{\sc Whitehead, J. H.~C.}
\newblock On equivalent sets of elements in a free group.
\newblock {\em Ann. of Math. (2) 37}, 4 (1936), 782--800.

\end{thebibliography}
\bibliographystyle{acm}

\end{document}